\documentclass[10pt, a4paper]{article}
\usepackage[margin=1in]{geometry}
\usepackage{color}

\usepackage{graphicx}

\usepackage{amsthm, amsfonts,amssymb} 
\usepackage{lmodern}
\usepackage[T1]{fontenc}
\usepackage{mathtools}
\usepackage{tikz}
\newtheorem{thm}{Theorem}[section]
\newtheorem{lem}[thm]{Lemma}
\newtheorem{prop}[thm]{Proposition}
\newtheorem{cor}[thm]{Corollary}

\newtheorem{mydef}{Definition}[section]

\newtheorem{rem}{Remark}[section]
\newcommand{\bFormula}[1]{
\begin{equation} \label{#1}}
\newcommand{\eF}{\end{equation}}
\newcommand{\vt}{\vartheta}
\newcommand{\vr}{\rho}
\newcommand{\vq}{{\bf q}}
\newcommand{\vu}{{\bf u}}
\newcommand{\bn}{{\bf n}}
\newcommand{\bu}{\mathbf{u}}

\newcommand{\diver}{{\rm div}}
\newcommand{\R}{{\mathbb{R}}}
\newcommand{\Z}{{\mathbb{Z}}}

\theoremstyle{remark}
\newcommand\bb[1]{\mathbf{#1}}
\newcommand\ddfrac[2]{\frac{\displaystyle #1}{\displaystyle #2}}

\newcommand\bint[1]{\displaystyle\int #1}

\usepackage[font=footnotesize,labelfont=bf]{caption}
\numberwithin{equation}{section}

\newcommand{\Addresses}{{
  \bigskip
  \footnotesize

V\'{a}clav M\'{a}cha, \textsc{Institute of Mathematics of the Academy of Sciences of the Czech Republic}\par\nopagebreak
  \textit{E-mail address}: \texttt{macha@math.cas.cz}

  \medskip

  Boris Muha, \textsc{Department of Mathematics, Faculty of Science, University of Zagreb}\par\nopagebreak
  \textit{E-mail address} \texttt{borism@math.hr}

  \medskip

\v{S}\'{a}rka Ne\v{c}asov\'{a}, \textsc{Institute of Mathematics of the Academy of Sciences of the Czech Republic}\par\nopagebreak
  \textit{E-mail address}: \texttt{matus@math.cas.cz}

\medskip
Arnab Roy, \textsc{BCAM, Basque Center for Applied Mathematics, Mazarredo 14, E48009 Bilbao, Bizkaia, Spain}\par\nopagebreak
  \textit{E-mail address}: \texttt{royarnab244@gmail.com}
  
\medskip
Sr\dj{}an Trifunovi\'{c}, \textsc{Department of Mathematics and Informatics, Faculty of Sciences, University of Novi Sad}\par\nopagebreak
  \textit{E-mail address}: \texttt{srdjan.trifunovic@dmi.uns.ac.rs}

}}

\author{V\'{a}clav M\'{a}cha, Boris Muha, \v{S}\'{a}rka Ne\v{c}asov\'{a}, Arnab Roy $\&$ Sr\dj{}an Trifunovi\'{c}}
\title{Existence of a weak solution to a nonlinear fluid-structure interaction problem with heat exchange}

\date{}                     

\begin{document}

\maketitle
\begin{abstract}
    In this paper, we study a nonlinear interaction problem between a thermoelastic shell and a heat-conducting fluid. The shell is governed by linear thermoelasticity equations and encompasses a time-dependent domain which is filled with a fluid governed by the full Navier-Stokes-Fourier system. The fluid and the shell are fully coupled, giving rise to a novel nonlinear moving boundary fluid-structure interaction problem involving heat exchange. The existence of a weak solution is obtained by combining three approximation techniques -- decoupling, penalization and domain extension. In particular, the penalization and the domain extension allow us to use the methods already developed for compressible fluids on moving domains. In such a way, the proof is more elegant and the analysis is drastically simplified. Let us stress that this is the first time the heat exchange in the context of fluid-structure interaction problems is considered.

\end{abstract}

\section{Introduction}

\subsection{Motivation and literature review}
The existence of global-in-time weak solutions to the equations related to fluid dynamics is one of the fundamental questions in the modern mathematical theory of fluid mechanics. In the case of the incompressible Navier-Stokes equations, the concept of a weak solution was introduced in seminal work of Leray \cite{Leray34}, where also existence results are proved. The corresponding theory for a barotropic compressible fluids is significantly more complicated and was developed much later, starting by pioneering works by Lions \cite{LionsBook2} and Feireisl et. al. \cite{FeireislCompressible01}. However, in many applications the assumption that the fluid is barotropic is too restrictive, e.g. \cite{MR832100,Wright07,gruber2003two}. Therefore, the mathematical theory of the full Navier-Stokes-Fourier system describing heat conducting fluid was developed quite recently, e.g. \cite{MR713680,MR826865,FeireislHeat04,FeireislBook04,feireislnovotny,MR3916800,MR3537466,MR3616663}. 

On the other hand,  the fluid interaction with elastic structures is common in many real life situations and understanding this interactions is of vital importance for applications, e.g. \cite{FSIBioMed,KalKusLasTriWeb}. We refer to such systems as fluid-structure interaction (FSI) systems. The mathematical analysis of FSI problems has been extensively studied in the last two decades and a lot of the progress has been achieved. However, most of the results concern the incompressible fluid case. The results on the existence of a weak solution typically deal with FSI problems where the elastic structure is described by a lower dimensional model of a plate/shell type, see \cite{SubBorARMA,CDEG05,MR2438783,LenRuz, BorSeb,srdjan1} and references within. Exceptions are works \cite{benevsova2020variational,multilayered} where the existence of a weak solution to FSI problems involving regularized, nonlinear, $3D$ viscoelastic structure and linear multilayered structure, respectively, were proven. All these works consider large data case and a solution existing as long as geometry does not degenerate, i.e. self-contact does not occur. There are also lot of results on the local-in-time or small data existence of strong solutions to FSI problems, see e.g. \cite{CS05,IKLT14,Raymond,MR4189724,MR3955112} and references within. We conclude the literature review about FSI models with incompressible fluid with a recent paper \cite{HilGra16}, where global-in-time solution to a $2D-1D$ FSI model with a viscoelastic beam was proven. 

The mathematical literature dealing with FSI problems with compressible fluids is scarce. In \cite{boulakia,KukTuff12} the authors prove the existence of local-in-time regular solutions. Recently, local-in-time existence results for strong solutions have been established for the compressible fluid-damped beam interaction in a 2D/1D framework in \cite{SM18} and for the compressible fluid-undamped wave interaction in a 3D/2D framework in \cite{maity2021existence}. The existence of a weak solution was proven in \cite{compressible,trwa3} and in \cite{breit2021compressible}, a weak solution was obtained for an interaction problem between a compressible fluid and a 3D viscoelastic structure. To the best of our knowledge there are only a few very recent papers dealing with the mathematical analysis of FSI problems with a heat conducting fluid. In \cite{MaityTak21} the existence of a strong solution for small time or small data is proven in the case when there is an additional damping on the structure, while in \cite{breit2021navier} existence of a weak solution is obtained for an FSI problem with a nonlinear Koiter shell. In both of these papers the structure does not conduct heat.

\subsection{Problem description}
We consider the flow of heat conducting compressible fluid in a $3D$ container with the  heat conducting elastic boundary. The fluid domain is determined by elastic displacement which is in turn obtained by solving the linearized Koiter shell equation with the forcing coming from the fluid, i.e. the fluid and the structure are fully coupled and we consider a moving boundary problem. 

\subsubsection{The problem geometry}\label{GeometryDef}
Let $\Omega\subset\R^3$ be an open, connected, bounded domain whose boundary $\partial\Omega$ is parametrized by an injective mapping $\varphi\in C^3(\Gamma;\R^3)$ such that $\partial \Omega= \varphi (\Gamma)$, where $\Gamma=\R^2/\Z^2$ is the flat torus (or $\Gamma=\R/\Z$ is a circle in 2D case). $\Omega$ represents the reference fluid configuration. We denote by $\bn$ the unit outer normal to $\Omega$. The assumption that $\Gamma$ is the flat torus is not very restrictive and is introduced for technical and presentational simplicity. It corresponds to the flow through a pipe with periodic boundary conditions which is common in applications. With slight abuse of notation we will identify functions defined on $\Gamma$ and $\partial\Omega$. We assume that the structure displacement is of the form $w(t,y)\bn (y)$, $y\in\Gamma$. The  elastic boundary at time $t$ is given by the following mapping (see Figure \ref{domain}):
\begin{eqnarray}\label{ab}
    \Phi_w(t,y)=\varphi(y)+w(t,y)\bn(y) \quad y\in \Gamma.
\end{eqnarray}
By a classical result on the tubular neighborhood, e.g. \cite[Section 10]{LeeIntroduction}, there exist numbers $a_{\partial\Omega}$, $b_{\partial\Omega}$ such that mapping $\Phi_w(t,.)$ is injective for $w(t,y)\in (a_{\partial\Omega},b_{\partial\Omega})$. Therefore the middle line of the shell at time $t$ occupies the following region:
\begin{eqnarray*}
    \Gamma^w(t)&: = \{\Phi_w(t,y): y \in \Gamma\}.
\end{eqnarray*}
\begin{figure}[h!]
\centering\includegraphics[scale=0.32]{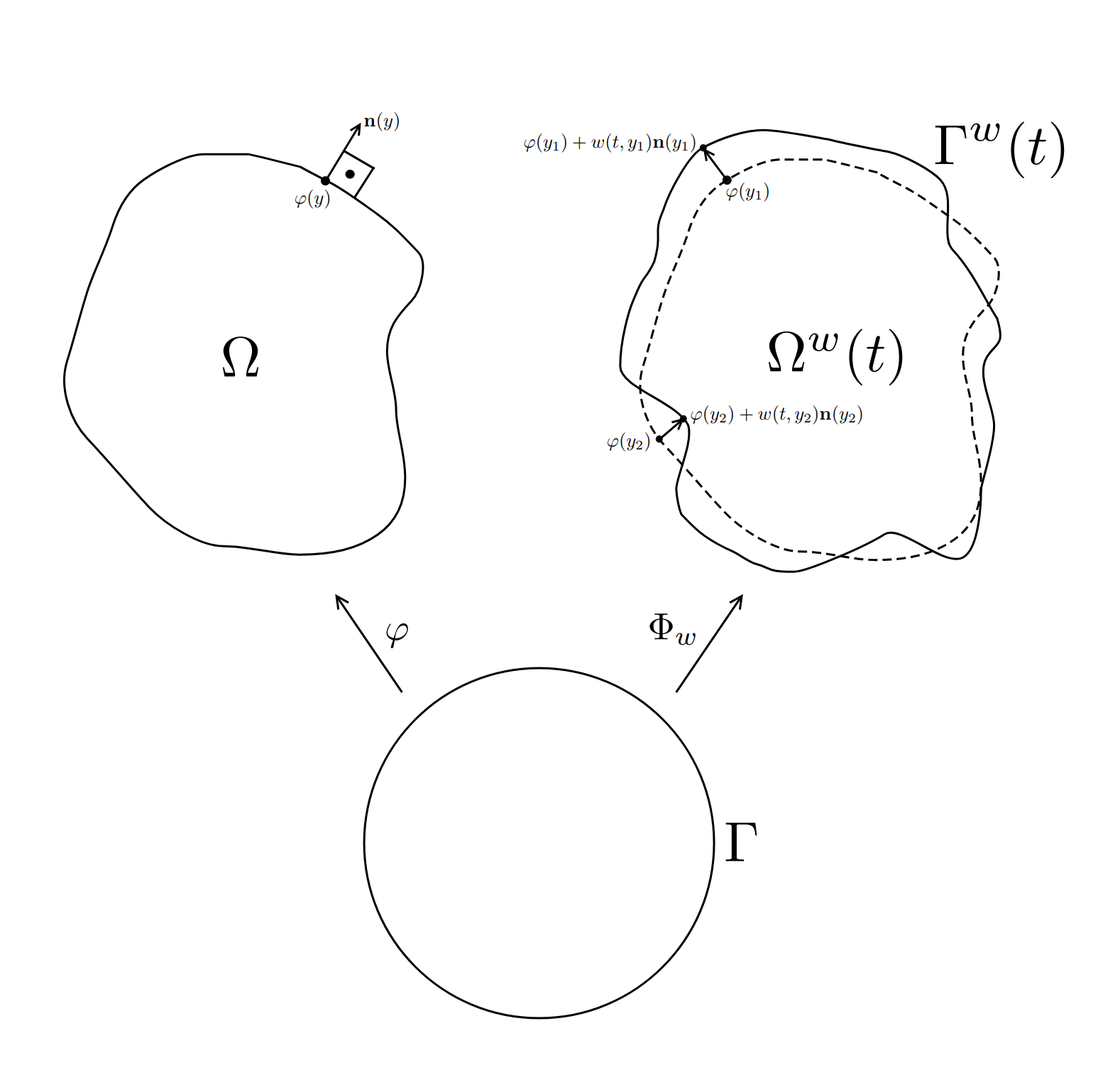}
\caption{The fluid domain $\Omega^w(t)$ determined by $\Gamma^w(t)$ - the shell at time $t$. The figure represents a vertical section of the torus in 3D or the entire model in 2D.}
\label{domain}
\end{figure}

The fluid domain at time $t$, $\Omega^w(t)\subset \mathbb{R}^3$, is defined as the interior region of $\Gamma^w(t)$. More precisely, let $\tilde{\Phi}_w$ be an arbitrary injective smooth extension of $\Phi_w$ to $\Omega$\footnote{An extension is explicitly constructed in $\eqref{domainflow}$. Note that its explicit form is not necessary for the definition of the problem.}. Then the fluid domain at time $t$ is defined by $\Omega^w(t):=\tilde{\Phi}_w(t,\Omega).$  Note that the fluid domain is well defined if $\Phi_w(t,.)$ is injective, which is true if condition  $w(t,y)\in (a_{\partial\Omega},b_{\partial\Omega})$ is satisfied. Therefore, our existence result is valid as long as this conditions holds. This is a consequence of the physical nature of the problem. Namely, for large displacement $w$, function $\Phi_w$ is not necessarily injective which introduces a domain degeneracy in the problem that corresponds to the self-contact of the shell. The question of how to analyse contact in the fluid-structure interaction is still largely open (see e.g. \cite{MR2354496,MR3281946,MR2592281,CasGraHil,HilTak09,gravina2020contactless,von2021falling} and reference within) and is outside of the scope of this paper. 
\begin{rem}
Even though the assumption that the shell deforms only in $\bn$ direction is somewhat restrictive from the physical point of view, it is nevertheless standard in the literature on weak solutions to FSI problems, e.g. \cite{CDEG05,LenRuz,SubBorARMA,BorSeb}. Namely, the existence of a weak solution to FSI problems with vector displacement is still out of reach with current state-of-the-art techniques and only available results in this direction use some kind of additional structure regularization, e.g. \cite{BorSunSlip,SunMarBor}.
\end{rem}

Finally, we introduce some notation related to the geometry. Let $\Gamma_T=(0,T)\times\Gamma$ be the shell space-time domain and 
$$
Q_T^w=\bigcup_{t\in (0,T)}\{t\}\times \Omega^w(t),\;
\Gamma^w_T=\bigcup_{t\in (0,T)}\{t\}\times \Gamma^w(t)
$$
be the Eulerian fluid domain and the Eulerian elastic interface domain, respectively. Note that the fluid domain depends on $w$ which is an unknown of the considered FSI problem. The outer unit normal to the deformed configuration $\Omega^w(t)$ is denoted by $\bn^w(t)$ and is given by formula $\bb{n}^w=\frac{\nabla\tilde{\Phi}_w^{-\tau}\bb{n}}{|\nabla\tilde{\Phi}_w^{-\tau}\bb{n}|}$. Finally, the surface element on the interface $\Gamma^w(t)$ is given by
$$
d\Gamma^w(t)=S^w d\Gamma, \quad S^w := \text{det} \nabla\tilde{\Phi}_w|\nabla\tilde{\Phi}_w^{-\tau}\bb{n}|.
$$
\subsubsection{The shell}
We use the following linear thermoelastic shell model to describe the dynamics of the elastic boundary \cite[section 1.5]{lagnese1989boundary}: 
\begin{eqnarray}
\partial_t^2 w+\Delta^2 w +\Delta \theta -\alpha_1 \Delta \partial_t w - \alpha_2 \partial_t^2 \Delta w &=&S^w{\mathbf F} \cdot \bb{n}, \label{shelleq}\\
\theta_t - \Delta \theta - \Delta w_t &=& S^w q. \label{shellheateq}
\end{eqnarray}
Here $w$ denotes the displacement of the shell central surface in the direction $\bn$ with respect to the reference configuration $\Gamma$ and $\theta\geq 0$ denotes the shell temperature. Here and in continuation of the paper we normalize all strictly positive physical positive constants since the proofs do not depend on their concrete values. $\alpha_1$ and $\alpha_2$ are coefficients of viscoelasticity and rotational inertia, respectively and we assume $\alpha_1,\alpha_2\geq 0$. Moreover, $\mathbf{F}$ is surface force density acting on the shell. Since we consider a linear model for the thermoelastic shell, equation \eqref{shellheateq} is the entropy equation, and $q$ is the entropy flux.
\subsubsection{The fluid}
We consider the flow of the heat conducting compressible fluid which is described by the three dimensional Navier-Stokes-Fourier system (see e.g. \cite{feireislnovotny}) defined on the moving domain $Q^w_T$:
\begin{eqnarray}
\partial_t \rho + \nabla \cdot (\rho \mathbf{u}) &=& 0, \label{conteq}\\[2mm]
\partial_t (\rho\mathbf{u}) + \nabla \cdot (\rho\mathbf{u}\otimes \mathbf{u})+\nabla p &=& \nabla \cdot\mathbb{S}, \label{momeq} \\[2mm]
\partial_t (\rho s) + \nabla \cdot (\rho s \mathbf{u})+ \nabla \cdot \Big( \displaystyle{\frac{\bb{q}}{\vartheta} }\Big)&=& \sigma. \label{enteq}
\end{eqnarray}
The unknown of the fluid system are the fluid densisty $\rho$, the fluid velocity $\bu$ and the fluid temperature $\vartheta\geq 0$. The equations \eqref{conteq}, \eqref{momeq} and \eqref{enteq} represent mass conservation, balance of momentum and entropy balance, respectively. The pressure $p$ and the specific entropy $s$ are the thermodynamical variables and depends on the density $\rho$ and the temperature $\vartheta$. Moreover, $\sigma$ is the entropy production rate and is given by formula:
\begin{equation}\label{sigma1}
\sigma = \frac{1}{\vt} \left( \mathbb{S} : \nabla \bb{u}  -\frac{\vq}{\vt} \cdot \nabla \vt \right).
\end{equation}
The viscous stress tensor is given by the Newton's rheological law:
\begin{eqnarray*}
\mathbb{S}(\vartheta, \nabla \bb{u}):=\mu(\vartheta) \big( \nabla \bb{u} + \nabla^\tau \bb{u}-\frac{2}{3} \nabla \cdot \bb{u} \big) + \zeta(\vartheta) \nabla \cdot \bb{u},\quad \mu(\vartheta) >0, \quad \zeta(\vartheta) > 0,
\end{eqnarray*}
and heat flux $\bb{q}$ by the Fourier's law:
\begin{eqnarray*}
\bb{q}:= - \kappa(\vartheta) \nabla\vartheta, \quad \kappa(\vartheta) > 0,
\end{eqnarray*}
where $\mu$, $\zeta$ are the viscosity coefficients and $\kappa$ is the heat coefficient. Observe that, every smooth solution to \eqref{conteq}-\eqref{enteq} with \eqref{sigma1} satisfy the following energy identity in $(0,T)$:
\begin{eqnarray}
&&\displaystyle{\frac{d}{dt} \int_{\Omega^w(t)} \Big( \frac{1}{2} \rho |\bb{u}|^2 + \rho e  \Big) +  \frac{d}{dt} \int_{\Gamma} \Big(\frac{1}{2}|\partial_t w|^2 + \frac{1}{2} |\Delta w|^2 +\frac{\alpha_2}{2} |\nabla \partial_t w|^2+ \frac{1}{2} |\theta|^2  \Big)} \nonumber \\
&&=-\displaystyle{\int_\Gamma|\nabla \theta|^2-\alpha_1 | \nabla\partial_t w|^2}. \label{enid}
\end{eqnarray}

\subsubsection{The coupling conditions}
Since we are considering a moving boundary problem we need to prescribe two sets of coupling conditions. The kinematic coupling conditions state that the velocity and the temperature are continuous on the interface $\Gamma_T$:
\begin{eqnarray}
\textbf{Continuity of the velocity:}&\partial_t w  \bb{n} &= \quad\mathbf{u} \circ \Phi_w, \label{kinc}\\
\textbf{Continuity of the temperature:}&\theta &= \quad\vartheta \circ \Phi_w. \label{heat}
\end{eqnarray}
Dynamic coupling conditions describe the balance of forces and the balance of entropy on $\Gamma_T:$
\begin{eqnarray}
\mathbf{F}&=\quad -\big[\big( p - \mathbb{S} \big) \bb{n}^w\big] \circ \Phi_w, \label{dync} \\
 q &=\quad-\Big(\frac{\bb{n}^w \bb{q}(\vartheta)}{\vartheta}\Big)\circ \Phi_w. \label{heatf}
\end{eqnarray}
\subsubsection{The initial conditions}

Finally, the initial data are prescribed:
\begin{eqnarray}\label{initialdata}
&&\vartheta(0,\cdot)=\vartheta_0(\cdot),~\rho(0,\cdot) = \rho_0, ~(\rho\mathbf{u})(0, \cdot)=(\rho\bb{u})_0, \\\label{initialdata1}
&&w(0,\cdot) = w_0,~\partial_t w(0,\cdot) = v_0, ~ \theta(0,\cdot) = \theta_0.
\end{eqnarray}
We assume that initial data satisfy the following regularity properties:
\begin{eqnarray} 
&&\rho_0\in L^{\frac{5}{3}}(\Omega^{w_0}), \quad \rho_0 \geq 0, \quad \rho_0 \not\equiv 0, \quad \rho_{0}|_{\mathbb{R}^3 \setminus \Omega^{w_0}} =0,\label{initial1}\\[3mm]
&&\rho_0>0 \text{ in } \{(X,z)\in \Omega^{w_0}: (\rho\bb{u})_0(X,z) >0\}, \quad  \frac{(\rho \bb{u})_0^2}{\rho_0} \in L^1(\Omega^{w_0}),\label{initial2}\\[2mm]
&&\vartheta_0 >0 \quad \text{a.e in } \Omega^{w_0}, \quad (\rho s)_0 = \rho_0 s(\vartheta_0,\rho_0)\in L^1(\Omega^{w_0}), \quad\quad\quad\quad \label{initial3}\\[3mm]
&& v_0\in L^2(\Gamma), \quad  \sqrt{\alpha_2} v_0 \in H^1(\Gamma), \ w_0 \in H^2(\Gamma),\ \theta_0 \in L^2(\Gamma),\quad \theta_0\geq 0, \\[3mm]
&& { E_0:=\displaystyle{\int_{\Omega^{w_0}} \Big( \frac{1}{2\rho_0} |(\rho\bb{u})_0|^2 + \rho_0 e(\vartheta_0,\rho_0)  \Big) +   \int_{\Gamma} \Big(\frac{1}{2}|v_0|^2 + \frac{1}{2} |\Delta w_0|^2 +\frac{\alpha_2}{2}|\nabla v_0|^2+ \frac{1}{2} |\theta_0|^2  \Big)} < \infty,} \quad\quad\quad\quad\label{initial4}
\end{eqnarray}
and the following compatibility condition:
\begin{eqnarray}\label{ICComp}
  a_{\partial\Omega}< w_0<b_{\partial\Omega} \quad \text{on }\Gamma.
\end{eqnarray}

\subsection{Main result and significance}

Before stating the main result, we need to introduce \textbf{Constitutive relations}
of the quantities in \eqref{conteq}-\eqref{enteq} in terms of the independent state variables that characterize the material properties of the fluid. Here we use the quite general constitutive relations from \cite[Section 1.4]{feireislnovotny} which we briefly list for the convenience of the reader.
We assume the viscosity coefficients $\mu$ and $\zeta$ are continuously differentiable functions of the absolute temperature, namely $\mu,\ \zeta \in C^1[0,\infty)$ and satisfy\footnote{The strict positivity of $\zeta$ ensures that $||\nabla\bb{u}+\nabla^\tau \bb{u}||_{L_{t,x}^2}$ can be controlled by the energy, which gives us the uniform bounds for $\bb{u}$ (see Lemma $\ref{korn}$ and $\eqref{nablau2}$).}
	\bFormula{mu_1}
	0< \underline{\mu} (1+ \vt) \leq \mu(\vt) \leq \overline{\mu} (1+\vt),
	\quad
	\sup_{\vt \in [0,\infty)} |\mu'(\vt)| \leq \overline{m},
	\eF
	\bFormula{eta_1}
	 0<\underline{\zeta}(1+\vt) \leq \zeta(\vt) \leq \overline{\zeta} (1+\vt).
	\eF
The heat coefficient $\kappa$ can be decomposed into two parts
	\bFormula{kappa_1}
\kappa(\vt) = \kappa_M (\vt) + \kappa_R (\vt)
	\eF
where $\kappa_M, \ \kappa_R \in C^1[0,\infty)$ and
	\bFormula{kappa_2}
	0< \underline{\kappa_R} (1+ \vt^3) \leq \kappa_R (\vt) \leq \overline{\kappa_R}(1+ \vt^3),
	\eF
	\bFormula{kappa_3}
	0< \underline{\kappa_M}   (1+ \vt) \leq \kappa_M (\vt) \leq \overline{\kappa_M}  (1+ \vt).
	\eF
In the above formulas $\underline{\mu}$, $\overline{\mu}$, $\overline{m}$, $\overline{\eta}$, $\underline{\kappa_R}$, $\overline{\kappa_R}$, $\underline{\kappa_M} $, $\overline{\kappa_M}$ are positive  constants.

The quantities $p$, $e$, and $s$ are continuously differentiable functions for positive values of $\vr$, $\vt$ and satisfy Gibbs' equation
	\bFormula{gibs}
	\vt D s(\vr,\vt) = D e(\vr,\vt) + p(\vr,\vt) D\left( \frac{1}{\vr}  \right) \mbox{ for all } \vr, \ \vt > 0.
	\eF
Further, we assume the following state equation for the pressure and the internal energy
\begin{equation}\label{p1p}
p(\vr,\vt) = p_M(\vr,\vt)+ p_R(\vt), \quad p_R(\vt) = \frac a3 \vt^4,\ a>0,
\end{equation}
\begin{equation}\label{e1e}
e(\vr,\vt) = e_M(\vr,\vt) + e_R(\vr,\vt), \quad \vr e_R(\vr,\vt) = a\vt^4,
\end{equation}
and
\begin{equation}\label{s1s}
s(\vr,\vt) = s_M(\vr,\vt) + s_R(\vr,\vt), \quad \vr s_R(\vr,\vt) = \frac 43 a\vt^3.
\end{equation}
According to the hypothesis of thermodynamic  stability the molecular components satisfy
	\bFormula{pm_1}
	\frac{\partial p_M}{\partial \vr} >0 \mbox{ for all } \vr, \ \vt >0
	\eF
and
	\bFormula{em1}
	0<\frac{\partial e_M}{\partial \vt} \leq c \mbox{ for all }  \vr, \ \vt >0.
	\eF
Moreover
	\bFormula{em2}
	\lim_{\vt\to 0^+} e_M(\vr,\vt) = \underline{e}_M(\vr) > 0 \mbox{ for any fixed } \vr >0,
	\eF	
and
	\bFormula{em3}
	\left| \vr \frac{\partial e_M(\vr,\vt)}{\partial \vr} \right| \leq c e_M(\vr, \vt) \mbox{ for all } \vr, \ \vt >0.
	\eF
We suppose also that there is a function $P$ satisfying
	\bFormula{p2p}
	P \in C^1[0,\infty), \ P(0)=0,\  P'(0)>0,
	\eF
and two positive constants $0< \underline{Z} < \overline{Z}$ such that
	\bFormula{pm2}
	p_M(\vr,\vt) =  \vt^{\frac{5}{2}} P\left( \frac{\vr}{\vt^{\frac{3}{2}}} \right)
	\mbox{ whenever } 0< \vr \leq \underline{Z} \vt^{\frac{3}{2}},\mbox{ or, } \vr > \overline{Z} \vt^{\frac{3}{2}}
	\eF
and
	\bFormula{pm3}
	p_M(\vr,\vt) = \frac{2}{3} \vr e_M (\vr,\vt) \mbox{ for } \vr > \overline{Z}\vt^{\frac{3}{2}}.
	\eF

\begin{rem}
A prototype of the above pressure law reads
\begin{eqnarray*}
    p(\rho,\vartheta)=c_1 \rho^{\frac{5}{3}}+c_2 \rho \vartheta + \frac{a}{3} \vartheta^4,
\end{eqnarray*}
with the corresponding internal energy and specific entropy of the form
\begin{eqnarray*}
    &&e(\rho,\vartheta)=\frac{3}{2}c_1 \rho^{\frac{2}{3}}+c_v  \vartheta + \frac{a}{\rho} \vartheta^4, \\
    &&s(\rho,\vartheta)=c_v \ln \vartheta- c_2 \ln \rho + \frac{4a}{3\rho} \vartheta^3,
\end{eqnarray*}
where $a,c_1,c_2,c_v>0$.
\end{rem}
We will denote system \eqref{shelleq}-\eqref{enteq}, \eqref{kinc}-\eqref{initialdata1} together with the described constitutive relations \eqref{mu_1}-\eqref{pm3} by FSI-HEAT.
The main result of the paper is:
\begin{thm}\label{mainth}
Let initial data $(\rho_0,(\rho\bb{u})_0,\vartheta_0,(\rho s)_0, w_0,v_0,\theta_0)$ satisfy assumptions \eqref{initial1}-\eqref{ICComp} and $\alpha_1+\alpha_2>0$. Moreover, assume that the hypotheses \eqref{mu_1}-\eqref{pm3} are satisfied. Then there exists $T>0$ and a weak solution to FSI-HEAT system in the sense of Definition \ref{weaksolution} defined on $(0,T)$. Moreover, either $T=+\infty$ or the domain $\Omega^w(t)$ degenerates as $t\to T$.
\end{thm}
The main novelties of the present work are:
\begin{enumerate}
    \item We consider a model where both the fluid and the structure conduct heat and there is heat coupling given through temperature continuity and entropy flux given in $\eqref{heat}$ and $\eqref{heatf}$. While, in the context of fluid-structure interaction, heat-conducting fluids have been studied in \cite{breit2021navier} and thermoelastic structures have been studied in \cite{trwa3}, to the best of our knowledge this is the first work that takes into account heat conduction of both components, and heat exchange between the components. One nice consequence of this approach is that we were able to prove that the plate temperature is positive which does not hold if one considers just a linear plate model without coupling it to the fluid with heat exchange.
    \item The fluid model we consider is from \cite{feireislnovotny}, and the same result also holds for the model given in $\eqref{another:model}$ below. This way, a wide range of physical cases are covered. In particular, we allow pressure laws which are not uniform and can change depending on the region of the $(\rho,\vartheta)$-plane (see \eqref{pm2} and \eqref{pm3}). This is very important in the case of gases, as they become fully ionized in the degenerate region $\rho>\vartheta^{\frac32}\overline{Z}$ and change their behavior.
    \item From the methodological point of view, we introduced a new approach to construct the approximate solutions. More precisely, we introduced a construction scheme that combines three approximation methods - decoupling, penalization and domain extension. This allows to decouple the problem and directly use sophisticated results and methods that are already developed to study compressible fluid on the moving domains \cite{heatfl,commoving,feireisl2011convergence}. We emphasize that such approach significantly simplifies the proof and has potential for further generalization. Namely, since our approach is modular, it is robust and can be adapted to more general fluid and/or structure models.
  \end{enumerate}
  
Let us first point out, that due to stronger imbedding results, our result easily holds in the case when the fluid is 2D and the structure is 1D, even in the case $\alpha_1=\alpha_2=0$:
\begin{cor}\label{Cor2D}
The conclusion of Theorem \ref{mainth} holds in 2D/1D case for $\alpha_1+\alpha_2\geq 0$.
\end{cor}

Inclusion of rotational inertia or viscoelasticity is needed in our proof. However, we use this assumption only in certain parts of the proof related to the convergence properties of the approximate solutions, and do not use it in the construction. Therefore, if we consider somewhat less general constitutive assumptions, our result holds without adding an additional regularization to the plate equations (see Remark $\ref{pressureremark}$ for a detailed explanation).

\begin{cor}\label{CorSimple}
Let us assume $p(\rho,\vartheta)=\rho^\gamma+\rho\vartheta+\frac{a}{3}\vartheta^4$ and $\gamma>\frac{12}{7}$ in 3D or $\gamma>1$ in 2D. Then conclusion of Theorem \ref{mainth} holds also for $\alpha_1=\alpha_2=0$.
\end{cor}
\begin{rem}\label{prototype}
The above pressure form is a prototype of the following pressure law where $p_M$ is given by
\begin{equation}\label{another:model}
p_M(\rho, \vartheta) = \vartheta^{\frac{\gamma}{\gamma-1}} P \left( \frac{\rho}{\vartheta^{\frac{1}{\gamma-1}}} \right),
\mbox{ whenever } 0< \vr \leq \underline{Z} \vt^{\frac{3}{2}},\mbox{ or, } \vr > \overline{Z} \vt^{\frac{3}{2}},
\end{equation}	
\begin{equation*}
	p_M(\vr,\vt) = (\gamma-1) \vr e_M (\vr,\vt) \mbox{ for } \vr > \overline{Z}\vt^{\frac{3}{2}}.
\end{equation*}
instead of (\ref{pm2})--(\ref{pm3}). For more details about such model, see \cite{NP} for the $3D$ case and \cite{MR2807430} for the $2D$ case.
\end{rem}
\begin{rem}
Even though our analysis is done for the $3D$ fluid flow, we also included results for $2D$ case (Corollaries \ref{Cor2D}, \ref{CorSimple} and Remark \ref{prototype}). Namely, the analysis for the $2D$ case is analogous for the following reasons. First, all the embedding results that we use are stronger in the $2D$ case. Second, our proof relies on the theory for the Navier-Stokes-Fourier system which is analogous for $2D$ case and the main difference in comparison to $3D$ case is that one can prove results with milder restriction on exponent $\gamma$ (see also recent result in $2D$ case \cite{PS21}). 

\end{rem}
\begin{rem}
It seems that most thermoelastic plate models in literature (including our own) are based upon the assumptions that the temperature is small with respect to the reference temperature, and that the entropy depends linearly on temperature \cite[Chapter 1.5]{lagnese1989boundary}. While this makes sense for the stability analysis of plates, it is in contrast with the thermodynamical properties of our fluid, which has strictly positive and arbitrarly large temperature and has a component of the entropy which depends logarithmically on the temperature. A possible solution to this problem might be to derive a new nonlinear thermoelastic plate model specially for our interaction problem, under the assumption that the thermodynamical properties of the plate are similar to the ones of the fluid, and then study its interaction with a heat-conducting fluid. This is a topic for future research.
\end{rem}

\subsection{Outline of the proof and organization of the paper}
The proof is split into three parts that correspond to various level of approximation in the construction of approximate solutions. Each step includes limiting procedure which uses standard tool for analysis of the compressible fluids equations. 
\begin{itemize}
    \item[Step 1] \textbf{Existence of a weak solution to the extended problem.} In the first step we define the extended problem on a large domain $B$. The extension involves approximation parameters $\eta,\omega,\nu,\lambda$ and follows approach from \cite{heatfl,commoving}. We also add pressure regularization with approximation parameter $\delta$. This regularization improves integrability of the pressure and by now standard in the analysis of compressible Navier-Stokes equations.
    In this step we prove existence of a weak solution to the extended problem (see Definition \ref{weaksolutionap1}). In order to construct approximate solutions to the extended problem, we introduce time step parameter $\Delta t$ and a time marching scheme that combines a decoupling approach (of the fluid and the structure) with penalization of the kinematic coupling conditions. The discussion about ideas behind this approach is included at the end of Section \ref{ExtensionSec}. The main advantage of such approach is that the existence result for the fluid part \cite{heatfl} can be directly used.
    \item[Step 2] \textbf{Extension limit} Here we study limit as extension parameters $\eta,\omega,\nu,\lambda\to 0$. In this part we adapt ides from \cite{heatfl} to pass to the limit and obtain a solution which is defined of the physical domain $Q^w_T$.
    \item[Step 3] \textbf{Pressure regularization limit}. The last step of the proof is standard and is common in all existence proofs of a weak solution to compressible fluid equations.
\end{itemize}

The paper is organized as follows. In the Section \ref{sec2}, we introduce a concept of weak solution. The next three sections correspond to the three steps of the proof as described above: Step 1 (Existence of a weak solution to the extended problem) is explained in Section \ref{penalization}, Step 2 (limit of extension parameters) and Step 3 (pressure regularization limit) are discussed in Section \ref{penlimit} and Section \ref{sec5} respectively. Finally, we include two Appendices where some technical results are proved. 

\section{Weak solution}\label{sec2}
We will use a concept of a weak solution that corresponds to the concept used in \cite[Chapter 2]{feireislnovotny} for the Navier-Stokes-Fourier system. However, since we consider a coupled moving boundary problem, there are some significant differences which we briefly describe before introducing the formal definition. First, since the fluid domain is defined by the structure displacement, we work with function spaces defined on the non-cylindrical domains in time and space. Moreover, from the energy inequality we have $w\in H^2(\Gamma)$ which is below threshold of Lipschitz regularity that is needed for standard functional analytic results on Sobolev spaces, such as the trace Theorem and Korn's inequality. This functional framework for FSI problems is by now standard, so we just refer to \cite[Section 1.3]{CDEG05} or \cite[Section 2]{LenRuz}. In particular, we will use the Lagrangian trace operator $\gamma_{|\Gamma^w}:C(\Omega^w)\to C(\Gamma)$ is defined as
\begin{eqnarray*}
    \gamma_{|\Gamma^w} f:= f_{|\Gamma^w} \circ \Phi_w. 
\end{eqnarray*}
and then extended to a continuous and bounded operator $\gamma_{|\Gamma^w}:W^{1,p}(\Omega^w)\to W^{1-\frac{1}{r},r}(\Gamma)$, for any $1<r<p$, \cite[Corollary  2.9]{LenRuz} (see also \cite{BorisTrace}). For time dependent functions and displacements, we will usually write
\begin{eqnarray*}
    (\gamma_{|\Gamma^w} f)(t,\cdot) = \gamma_{|\Gamma^w(t)} f(t,\cdot).
\end{eqnarray*}
Moreover, we prove the following version of Korn's inequality (Lemma $\ref{ourkorn}$):
\begin{eqnarray*}
    ||\bb{u}||_{ W^{1,p}(\Omega^w)}^2 \leq C \Big[||  \nabla \bb{u} + \nabla^\tau \bb{u}||_{L^2(\Omega^w)}^2 + \int_{\Omega^w} \rho |\bb{u}|^2 \Big], \quad \text{for any } p<2,
\end{eqnarray*}
where constant $C$ blows up as $p$ goes to $2$, and therefore we will use $W^{1,p}$, $p<2$ space instead of $H^1$ in the definition of weak solution. Finally, solution and test spaces depend on solution and are not linear spaces. The kinematic coupling conditions \eqref{kinc}, \eqref{heat} are incorporated into the solution spaces, while the dynamic coupling conditions \eqref{dync} and \eqref{heatf} are implicitly prescribed via the weak formulation. Namely, $\mathbb{S}(\vartheta,\nabla\bb{u}), p(\rho,\vartheta)$ and $\frac{\bb{q}(\vartheta)}{\vartheta}$ do not have well-defined traces on the interface, and therefore \eqref{dync} and \eqref{heatf} are only formally satisfied in a weak formulation.

\begin{mydef}\label{weaksolution}(\textbf{Weak solution})
We say that $(\vartheta,\rho ,\bb{u}, w, \theta)$ is a weak solution to the FSI-HEAT problem with initial data $(\rho_0,(\rho\bb{u})_0,\vartheta_0,(\rho s)_0, w_0,v_0,\theta_0)$ satisfying the assumptions $\eqref{initial1}-\eqref{ICComp}$, if the following conditions hold: 
\begin{enumerate}
\item { $\rho \geq 0$}, $\rho \in L^\infty(0,T; L^{\frac{5}{3}}(\mathbb{R}^3))\cap L_{loc}^q([0,T]\times \Omega^w(t))$, for some\footnote{Here, the additional integrability of density can only be obtained on compact subsets, rather than on the whole domain. This is because $w$ is not regular enough to ensure Lipschitz regularity of the fluid domain $\Omega^w(t)$ (which also changes in time), so the standard improved estimates of the density based on the Bogovskii operator do not hold up to the boundary.} $q>\frac{5}{3}$;\\
$\bb{u}\in L^2(0,T; W^{1,p}(\Omega^w(t)))$ for any $p<2$, $\rho |\bb{u}|^2 \in L^\infty(0,T; L^1(\mathbb{R}^3))$;\\
$\vartheta>0$ a.e. in $Q_T^w$, $\vartheta \in L^\infty(0,T; L^4(\Omega^w(t)))$;\\
 $\vartheta, \nabla \vartheta \in L^2(Q_T^w)$, $\log \vartheta,  \nabla \log \vartheta \in L^2(Q_T^w)$;\\
$\rho s, \rho s \bb{u}, \frac{q}{\vartheta} \in L^1(Q_T^w)$; \\
$w \in L^\infty(0,T; H^2(\Gamma))\cap W^{1,\infty}(0,T; L^2(\Gamma))$, $\alpha_1 w\in H^1(0,T;H^1(\Gamma))$, $\alpha_2 w\in W^{1,\infty}(0,T;H^1(\Gamma))$;\\
$\theta>0$ a.e. on $\Gamma_T$, $\theta \in L^\infty(0,T; L^2(\Gamma))\cap L^2(0,T; H^1(\Gamma))$;\\
$\ln\theta\in L^2(0,T;H^s(\Gamma))$, for any\footnote{This comes from the fact that $\gamma_{|\Gamma^w}\ln\vartheta=\ln\theta$, see Lemma \ref{lntrace}.} $s<\frac12$. 
\item The coupling conditions $\partial_t w  \bb{n} = \gamma_{|\Gamma^w} \mathbf{u}$ and $\vartheta =  \gamma_{|\Gamma^w} \theta$ hold on $\Gamma_T$.
\item The renormalized continuity equation
\begin{eqnarray}\label{reconteqweak}
 \int_{Q_T^w} \rho B(\rho)( \partial_t \varphi +\bb{u}\cdot \nabla \varphi) =\int_{Q_T^w} b(\rho)(\nabla\cdot \bb{u}) \varphi +\int_{\Omega^{w_0}} \rho_0 B(\rho_0) \varphi(0,\cdot),
\end{eqnarray}
holds for all $\varphi \in C_c^\infty([0,T)\times \overline{\Omega^w(t)})$ and any $b\in L^\infty (0,\infty) \cap C[0,\infty)$ such that $b(0)=0$ with $B(\rho)=B(1)+\int_1^\rho \frac{b(z)}{z^2}dz$.
\item The coupled momentum equation
\begin{eqnarray}
    &&\int_{Q_T^w} \rho \bb{u} \cdot\partial_t \boldsymbol\varphi + \int_{Q_T^w}(\rho \bb{u} \otimes \bb{u}):\nabla\boldsymbol\varphi +\int_{Q_T^w} p(\rho,\vartheta) (\nabla \cdot \boldsymbol\varphi) -  \int_{Q_T^w} \mathbb{S}(\vartheta, \nabla\bb{u}): \nabla \boldsymbol\varphi\nonumber\\
    &&+\int_{\Gamma_T} \partial_t w \partial_t \psi - \int_{\Gamma_T}\Delta w \Delta \psi-\alpha_1 \bint_{\Gamma_T} \partial_t \nabla w \cdot \nabla \psi + \alpha_2 \bint_{\Gamma_T}\partial_t \nabla w\cdot \partial_t \nabla \psi  +\int_{\Gamma_T} \nabla \theta \cdot \nabla \psi  \nonumber\\
    &&= - \int_{\Omega^{w_0}}(\rho\mathbf{u})_0\cdot\boldsymbol\varphi(0,\cdot)-  \int_\Gamma v_0 \psi(0,\cdot)-\alpha_2  \int_\Gamma  \nabla  v_0 \cdot \nabla\psi(0,\cdot),  \label{momeqweak}
\end{eqnarray}
holds for all $\boldsymbol\varphi \in C_c^\infty([0,T)\times \overline{\Omega^w(t)})$ and $\psi\in C_c^\infty([0,T)\times \Gamma)$ such that $\gamma_{|\Gamma^w}\boldsymbol\varphi  =\psi \bb{n}$ on $\Gamma_T$. 
\item The coupled entropy inequality
\begin{eqnarray}
    &&\int_{Q_T^w} \rho s( \partial_t \varphi + \bb{u}\cdot \nabla \varphi) -\int_{Q_T^w}\frac{\kappa(\vartheta) \nabla \vartheta \cdot \nabla \varphi}{\vartheta}+ \int_{Q_T^w} \frac{\varphi}{\vartheta}\Big( \mathbb{S}(\vartheta, \nabla \bb{u}):\nabla \bb{u}
    +\frac{\kappa(\vartheta)|\nabla \vartheta|^2}{\vartheta}\Big) \nonumber\\
    &&+\int_{\Gamma_T} \theta \partial_t \tilde{\psi} -\int_{\Gamma_T} \nabla\theta \cdot \nabla \tilde{\psi} +\int_{\Gamma_T} \nabla w \cdot \nabla \partial_t \tilde{\psi} \nonumber\\
    &&\leq - \int_{\Omega^{w_0}} \rho_0 s(\vartheta_0,\rho_0) \varphi(0,\cdot) -  \int_{\Gamma} \theta_0 \tilde{\psi}(0,\cdot) - \int_{\Gamma} \nabla w_0 \cdot \nabla\tilde{\psi}(0,\cdot) \label{entineqweak}
\end{eqnarray}
holds for all non-negative $\varphi \in C_c^\infty([0,T)\times \overline{\Omega^w(t)})$ and $\tilde{\psi}\in C_c^\infty([0,T) \times\Gamma)$ such that $\gamma_{|\Gamma^w}\varphi =\tilde{\psi}$ on $\Gamma_T$.
\item The energy inequality
\begin{eqnarray}\label{enineq}
  &&\displaystyle{\int_{\Omega^w(t)} \Big( \frac{1}{2} \rho |\bb{u}|^2 + \rho e(\rho,\vartheta)  \Big)(t) +  \int_{\Gamma} \Big(\frac{1}{2}|\partial_t w|^2 + \frac{1}{2} |\Delta w|^2 +\frac{\alpha_2}{2}|\nabla \partial_t w|^2 +\frac{1}{2} |\theta|^2  \Big)(t)}  \nonumber \\ && + \displaystyle{\int_0^T \int_\Gamma \alpha_1 |\nabla \partial_t w|^2} + |\nabla \theta|^2 \leq \displaystyle{\int_{\Omega^{w_0}} \Big( \frac{1}{2\rho_0} |(\rho\bb{u})_0|^2 + \rho_0 e(\vartheta_0,\rho_0)  \Big)} \nonumber \\ && + \int_{\Gamma} \Big(\frac{1}{2}|v_0|^2 + \frac{1}{2} |\Delta w_0|^2+ \frac{\alpha_2}{2}|\nabla v_0|^2 + \frac{1}{2} |\theta_0|^2  \Big)
\end{eqnarray}
holds for all $t\in [0,T]$.
\end{enumerate}
\end{mydef}

\begin{rem}
(1) The derivation of coupled momentum equation for smooth solutions is standard and we refer to \cite{compressible,trwa3} for more details. However, the coupled entropy inequality appears here for the first time and it is derived in Appendix B.\\
(2) In the above definition, we have both entropy and energy balances in the form of inequalities. While entropy balance being satisfied as an inequality is standard, the energy one is different from standard theory \cite{feireislnovotny}. This is a consequence of the fact that we have additional dissipation terms on the interface due to the coupling with the thermoelastic shell, and weak solution is not regular enough to obtain compactness results needed to preserve the energy equality in the limiting procedure. Although this definition may seem restrictive, we argue that is sufficient. Namely, if we assume that a weak solution is regular enough, we can obtain that both entropy and energy inequalities hold as equalities. This is proved in Appendix B by following the ideas from \cite{poul}.
\end{rem}

\begin{rem}(Achieving the initial data). By the standard theory,  one deduces from $\eqref{reconteqweak}$ that
\begin{eqnarray*}
    \rho \in C_w(0,T; L^{\frac{5}{3}}(\mathbb{R}^3)),
\end{eqnarray*}
and since $\eqref{momeqweak}$ holds for any compactly supported function $\boldsymbol\varphi\in C_c^\infty(Q_T^w)$ (with $\psi=0$), one also has
\begin{eqnarray*}
    \rho \bb{u} \in C_w(0,T; L^{\frac{5}{4}}(\mathbb{R}^3)).
\end{eqnarray*}
Consequently, the equation $\eqref{momeqweak}$ implies
\begin{eqnarray*}
    \int_\Gamma \partial_t w(t,y) \psi(y) +\alpha_2  \int_\Gamma  \nabla \partial_t w(t,y) \cdot \nabla\psi(y) \to \int_\Gamma v_0 \psi+\alpha_2  \int_\Gamma  \nabla  v_0 \cdot \nabla\psi, \quad \text{as } t\to 0,
\end{eqnarray*}
for any $\psi \in C^\infty(\Gamma)$, which  gives by density argument 
\begin{eqnarray*}
    \lim\limits_{t\to 0}\left[\int_\Gamma \partial_t w \psi +\alpha_2  \int_\Gamma \nabla \partial_t w \cdot \nabla\psi\right](t) = \int_\Gamma v_0 \psi+\alpha_2  \int_\Gamma  \nabla  v_0 \cdot \nabla\psi, \quad \text{for any } \psi \in L^2(\Gamma), \text{ with } \alpha_2\psi \in H^1(\Gamma).
\end{eqnarray*}
However, since  there is $\psi\in H^2(\Gamma)$ such that $(\psi-\alpha_2\Delta\psi)=\phi$ for every function $\phi\in L^2(\Gamma)$, this implies 
\begin{eqnarray*}
    &&\lim\limits_{t\to 0}\left[\int_\Gamma \partial_t w \psi +\alpha_2  \int_\Gamma \nabla \partial_t w \cdot \nabla\psi\right](t) =  \lim\limits_{t\to 0}\left[\int_\Gamma \partial_t w \psi - \int_\Gamma  \partial_t w  \alpha_2\Delta\psi\right](t) \\
    &&= \lim\limits_{t\to 0}\left[\int_\Gamma \partial_t w \big(\psi -\alpha_2 \Delta\psi\big)\right](t) = \lim\limits_{t\to 0}\left[\int_\Gamma \partial_t w \phi \right](t)
    = \int_\Gamma v_0 \phi, \quad \text{for any } \phi \in L^2(\Gamma).
\end{eqnarray*}
Therefore
\begin{eqnarray*}
    &&\lim\limits_{t\to 0}\left[\int_\Gamma \partial_t w \psi \right](t) = \int_\Gamma v_0 \psi,\quad \text{for any } \psi \in L^2(\Gamma),\\
    && \lim\limits_{t\to 0}\left[ \alpha_2\int_\Gamma \nabla \partial_t w \cdot \nabla\psi\right](t)  =\alpha_2\int_\Gamma  \nabla  v_0 \cdot \nabla\psi,\quad \text{for any } \psi \in  H^1(\Gamma).
\end{eqnarray*}
Finally, the entropy inequality $\eqref{entineqweak}$ yields
\begin{eqnarray*}
    \lim\limits_{t\to 0}\left[\int_{\Omega^w} \big(\rho s(\rho,\vartheta)\big) + \int_{\Gamma} \theta\right](t) \geq  \int_{\Omega^{w_0}} \rho_0 s(\vartheta_0,\rho_0)+  \int_{\Gamma} \theta_0.
\end{eqnarray*}
\end{rem}

\section{Step 1 - Extended problem}\label{penalization}
The first step in the construction of an approximate solution is to extend the problem to a large fixed domain which contains the physical domain $\Omega^w(t)$ for every $t\in [0,T]$. Here we follow the approach from \cite{heatfl}. Let $R>0$ be large enough so that $\Omega^w(t) \subset B:= \{|X|<2R\}$ for all $t\in [0,T]$. Note that  the displacement $w$ is bounded due to the energy estimates and thus such $R$ exists. Our constructed  solution will satisfy the energy estimates, so this assumption is justified.
The idea is to extend the problem onto $B$ (see figure $\ref{penalpic}$) by extending the initial data, the viscosity coefficients and the heat conductivity coefficients in the following way. 

\begin{figure}[h!]
\centering\includegraphics[scale=0.2]{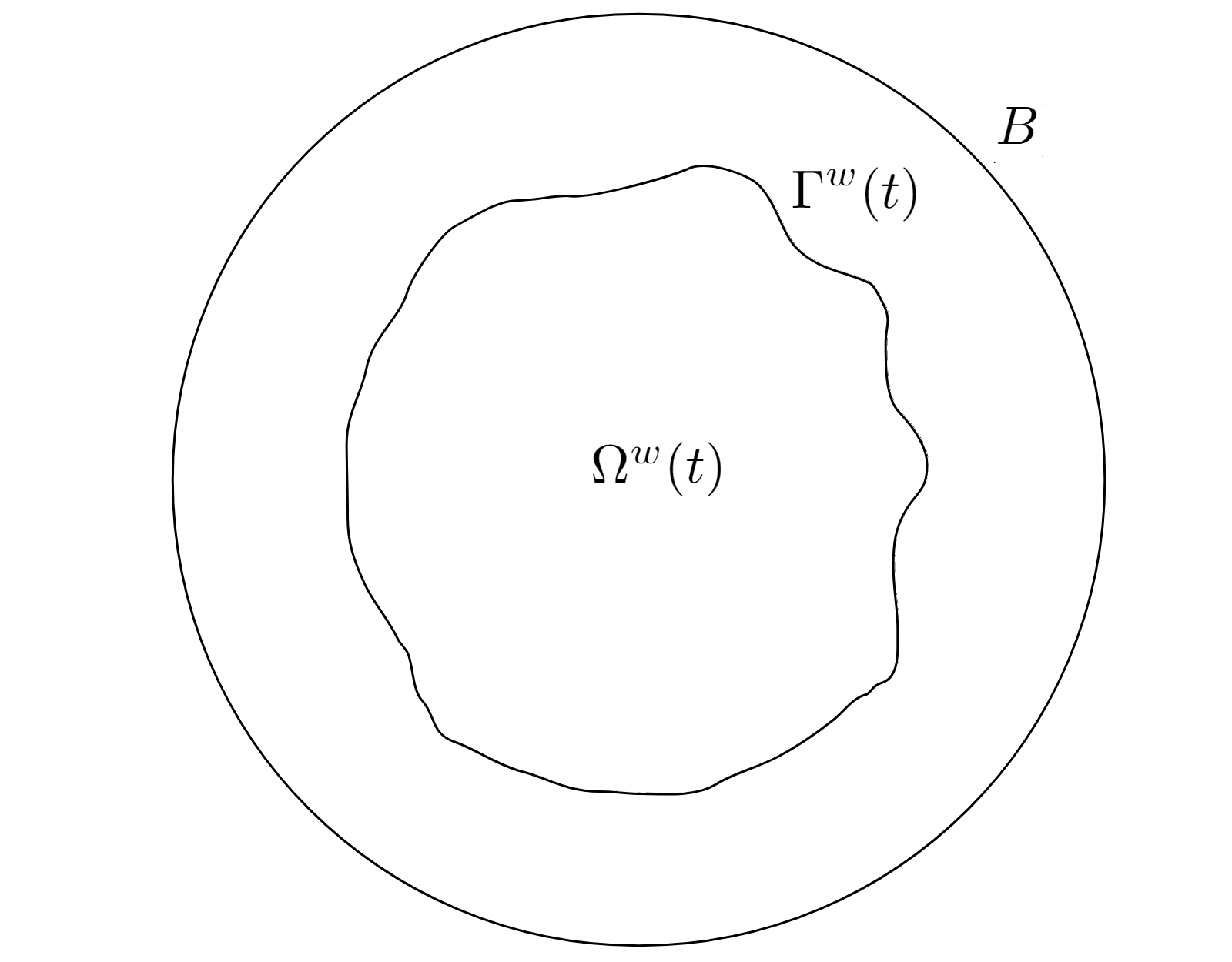}
\caption{The fluid domain extended from $\Omega^w(t)$ to $B$.}
\label{penalpic}
\end{figure}

\subsection{Extension of data and coefficients}\label{ExtensionSec}
For a given $\omega>0$ and the structure displacement $w$, the shear viscosity coefficients $\mu$ and $\zeta$ (see \eqref{mu_1}--\eqref{eta_1}) are approximated as
\begin{eqnarray}\label{ext-mu}
    \mu_{\omega}^w:=f_\omega^w \mu, \quad \zeta_{\omega}^w:=f_\omega^w \zeta
\end{eqnarray}
where $f_\omega^w$ Lipschitz continuously depends on $w$ and
\begin{eqnarray*}
&&f_{\omega}^w \in C_c^\infty([0,T]\times \mathbb{R}^3), \quad 0<\omega \leq f_{\omega}^w(t,x) \leq 1, \text{ in } [0,T] \times B,\nonumber\\
&& f_{\omega}^w(t,\cdot)_{|\Omega^w(t)} = 1,\text{ for all } t \in [0,T],\quad ||f_{\omega}^{w}||_{L^p(((0,T)\times B)\setminus Q_T^w)}\leq C \omega \quad \text{for some } p\geq \frac{5}{3}. \label{fomega}
\end{eqnarray*}

\noindent Next, for a given $\nu>0$, the heat conductivity coefficient $\kappa$ (see $ \eqref{kappa_1}-\eqref{kappa_3}$) is approximated by
\begin{eqnarray}\label{ext-k}
\kappa_{\nu}^w(\vartheta,t,x)= \chi_{\nu}^w(t,x) \kappa(\vartheta), \text{ where } \chi_\nu^w =1 \text{ in } Q_T^w \text{ and } \chi_\nu^w = \nu \text{ in } ((0,T)\times B) \setminus Q_T^w,
\end{eqnarray}
and similarly, the coefficient $a$ corresponding to the radiative part of the pressure, internal energy and energy (see \eqref{p1p}-\eqref{s1s}) is approximated by
\begin{eqnarray}\label{ext-a}
    a_\eta^w: = \chi_\eta^w a, \text{ where } \chi_\eta^w =1 \text{ in } Q_T^w \text{ and } \chi_\eta^w = \eta \text{ in } ((0,T)\times B) \setminus Q_T^w.
\end{eqnarray}
The pressure is approximated as follows:
\begin{eqnarray}\label{ext-p}
p_{\eta,\delta}^w:=p_M(\rho,\vartheta)+\frac{a_\eta^w}{3}\vartheta^4+\delta \rho^\beta,\quad \beta \geq 4, \quad \delta>0,
\end{eqnarray}
and the internal energy and specific entropy are approximated accordingly
\begin{eqnarray}\label{ext-e}
e_\eta^w:=e_M(\rho,\vartheta)+a_\eta^w\frac{\vartheta^4}{\rho}, \quad s_\eta^w(\rho,\vartheta):=s_M(\rho,\vartheta)+\frac{4}{3}a_\eta^w \frac{\vartheta^3}{\rho}.
\end{eqnarray}

The initial data $\rho_0,(\rho \bb{u})_0,\vartheta_0$ defined on $\Omega^{w_0}$ are extended and approximated by $\rho_{0,\delta},(\rho \bb{u})_{0,\delta},\vartheta_{0,\delta}$ as follows:
\begin{eqnarray}\label{init-ap}
&&\rho_{0,\delta}\geq 0,~~ \rho_{0,\delta}\not\equiv 0,~~ \rho_{0,\delta}|_{\mathbb{R}^3\setminus \Omega^{w_{0}}}=0, ~~\int_B \big(\rho_{0,\delta}^{\frac{5}{3}}+\delta\rho_{0,\delta}^\beta\big) \leq c,\\ &&\rho_{0,\delta}\to \rho_0 \text{ in }L^{\frac{5}{3}}(B), ~~ |\{\rho_{0,\delta}<\rho_0 \}|\to 0, \text{ as } \delta \to 0,\\
&&(\rho\bb{u})_{0,\delta}=\begin{cases} (\rho\bb{u})_0, &\text{ if }\rho_{0,\delta}\geq \rho_0, \\
0, &\text{ otherwise}\end{cases},\quad \int_{B} \frac{1}{\rho_{0,\delta}}|(\rho\bb{u})_{0,\delta}|^2 \leq c,\\
&& \vartheta_{0,\delta} \geq \underline{\vartheta} \geq 0 \text{ and } \vartheta_{0,\delta} \in L^\infty(B)\cap C^{2+\nu}(B),
\end{eqnarray}
and $\rho_{0,\delta},\vartheta_{0,\delta}$ are such that
\begin{eqnarray}\label{init-ap1}
\begin{aligned}
&\int_{\Omega^{w_{0}}} \rho_{0,\delta} e(\rho_{0,\delta},\vartheta_{0,\delta}) \to \int_{\Omega^{w_0}} \rho_0 e(\rho_{0},\vartheta_{0}),\\[2mm]
&\rho_{0,\delta} s(\rho_{0,\delta},\vartheta_{0,\delta}) \rightharpoonup \rho_0 s(\rho_{0},\vartheta_{0}), \text{ in }L^1(\Omega^{w_0}),
\end{aligned}
\end{eqnarray}
as $\delta \to 0$.

Now we define a weak solution to the extended FSI-HEAT problem on $B$ in the following way:
\begin{mydef}\label{weaksolutionap1}(\textbf{Weak solution to extended problem})
We say  that  $ (\vartheta,\rho ,\bb{u}, w, \theta)$ is a weak solution to the extended FSI-HEAT problem if it satisfies the initial conditions \eqref{initialdata}, \eqref{initialdata1} and

\begin{enumerate}
\item {$\rho \geq 0$}, $\rho \in L^\infty(0,T; L^{\beta}(\mathbb{R}^3))\cap L_{loc}^q([0,T]\times( \overline{B}\setminus\Gamma^w(t)))$, for some $q>\frac{5}{3}$;\\
Other regularity assumptions are the same as in Definition \ref{weaksolution}, but with the fluid quantities defined on $(0,T)\times B$ instead of $Q^w_T$. 
\item The coupling conditions $\partial_t w  \bb{n} = \gamma_{|\Gamma^w} \mathbf{u}$ and $\vartheta =  \gamma_{|\Gamma^w} \theta$ hold on $\Gamma_T$.
\item The renormalized continuity equation
\begin{eqnarray}\label{RCE1}
    \int_0^T \int_B \rho B(\rho)( \partial_t \varphi +\bb{u}\cdot \nabla \varphi) =\int_0^T \int_B b(\rho)(\nabla\cdot \bb{u}) \varphi +\int_{B} \rho_{0,\delta} B(\rho_{0,\delta}) \varphi(0,\cdot),
\end{eqnarray}
holds for all $\varphi \in C_c^\infty([0,T)\times B)$ and any $b\in L^\infty (0,\infty)\cap C[0,\infty)$ such that $b(0)=0$ with $B(\rho)=B(1)+\int_1^\rho \frac{b(z)}{z^2}dz$.
\item The coupled momentum equation of the form
\begin{align}
    &\int_0^T \bint_B \rho \bb{u} \cdot\partial_t \boldsymbol\varphi + \int_0^T \bint_B (\rho \bb{u} \otimes \bb{u}):\nabla\boldsymbol\varphi +\int_0^T \bint_B  p_{\eta,\delta}^w(\rho,\vartheta) (\nabla \cdot \boldsymbol\varphi) -  \int_0^T \bint_B \mathbb{S}_\omega^w(\vartheta, \nabla\bb{u}): \nabla \boldsymbol\varphi\nonumber\\
    &+\bint_{\Gamma_T}\partial_t w \partial_t \psi -\bint_{\Gamma_T}\Delta w \Delta \psi -\alpha_1 \bint_{\Gamma_T} \partial_t \nabla w \cdot \nabla \psi + \alpha_2 \bint_{\Gamma_T}\partial_t \nabla w\cdot \partial_t \nabla \psi +\bint_{\Gamma_T}\nabla \theta \cdot \nabla \psi \nonumber\\
    &= - \int_{B}(\rho\mathbf{u})_0\cdot\boldsymbol\varphi(0,\cdot)- \int_\Gamma v_0 \psi(0,\cdot)-\alpha_2  \int_\Gamma  \nabla  v_0 \cdot \nabla\psi(0,\cdot) , \label{coupledmomeq1}
\end{align}
holds for all $\boldsymbol\varphi \in C_c^\infty([0,T)\times B)$ and $\psi \in C^\infty([0,T)\times \Gamma)$ such that $\psi \bb{n}=\gamma_{|\Gamma^{w}}\boldsymbol\varphi$ on $\Gamma_T$. 
\item The coupled entropy balance of the form
\begin{eqnarray}
    &&\int_0^T\int_{B} \rho s_\eta^w( \partial_t \varphi + \bb{u}\cdot \nabla \varphi) -\int_0^T\int_{B} \frac{\kappa_\nu^w(\vartheta) \nabla \vartheta \cdot \nabla \varphi}{\vartheta}+\langle \sigma_{\omega,\nu}^w;\varphi\rangle_{[\mathcal{M},C]([0,T]\times \overline{B})}\nonumber\\
    &&+\lambda \int_0^T\int_{B}   \vartheta^4 \varphi
    +\int_0^T\int_{\Gamma} \theta \partial_t \tilde{\psi} -\int_0^T\int_{\Gamma}  \nabla\theta \cdot \nabla \tilde{\psi} +\int_0^T\int_{\Gamma} \nabla w \cdot \nabla \partial_t \tilde{\psi}  \nonumber\\
    &&=-\int_{B} \rho s_\eta^w(\vartheta_{0,\delta},\rho_{0,\delta}) \varphi(0,\cdot) - \int_{\Gamma} \theta_0 \tilde{\psi}(0,\cdot) - \int_{\Gamma} \nabla w_0 \cdot \nabla\tilde{\psi}(0,\cdot)  \label{coupledenteq1}
\end{eqnarray}
 holds for all non-negative $\varphi \in C_c^\infty([0,T)\times B)$ and $\tilde{\psi} \in C_c^\infty([0,T)\times\Gamma)$ such that $\tilde{\psi} =\gamma_{|\Gamma^{w}}\varphi$ on $\Gamma_T$, where
 \begin{eqnarray*}
   \sigma_{\omega,\nu}^w\geq \frac{1}{\vartheta}\Big( \mathbb{S}_\omega^w(\vartheta, \nabla \bb{u}):\nabla \bb{u}+\frac{\kappa_\nu^w(\vartheta)|\nabla \vartheta|^2}{\vartheta}\Big).
\end{eqnarray*}
\item The following energy inequality
\begin{eqnarray}\label{ExtenededEiIn}
    &&\displaystyle{\int_{B} \Big( \frac{1}{2} \rho |\bb{u}|^2 + \rho e_\eta^w(\rho,\vartheta)+\frac{\delta}{\beta-1}\rho^{\beta} \Big)(t)} + \lambda\int_{0}^t  \int_{B} \vartheta^5 +\frac{1-\delta}{2}|| \partial_t w(t)||_{L^2(\Gamma)}^2 +\frac{1}{2}|| \Delta w(t)||_{L^2(\Gamma)}^2 \nonumber \\
    && \quad+\frac{\alpha_2}{2}||\nabla \partial_t w(t)||_{L^2(\Gamma)}^2 + \frac{1-\delta}{2}||\theta(t)||_{L^2(\Gamma)}^2+ \int_{0}^t \int_{\Gamma}\big(\alpha_1 |\nabla \partial_t w|^2+ | \nabla \theta|^2 \big)\nonumber \\
    &&  \leq \displaystyle{\int_{B} \Big( \frac{1}{2 \rho_{0,\delta }} |(\rho\bb{u})_{0,\delta }|^2 + \rho_{0,\delta } e_\eta^w(\vartheta_{0,\delta },\rho_{0,\delta })}+\frac{\delta}{\beta-1}\rho_{0,\delta }^{\beta} \Big) \nonumber\\
    &&\quad+ \frac{1}{2}|| v_{0,\delta }||_{L^2(\Gamma)}^2+\frac{1}{2}|| \Delta w_{0,\delta }||_{L^2(\Gamma)}^2 + \frac{\alpha_2}{2}||\nabla v_{0,\delta}||^2+ \frac{1}{2}||\theta_{0,\delta }||_{L^2(\Gamma)}^2,
\end{eqnarray}
holds for all $t\in(0,T]$.
\end{enumerate}
\end{mydef}
\begin{rem}
The solution defined in the above definition depends on the parameters $\omega, \eta,\nu,\lambda,\delta$. However, in order to simplify the notation, we will not write this explicitly. Throughout the rest of the paper, we adapt the convention that we do write this explicit dependence on parameter only in the limiting procedure related to that parameter. When there is no possibility of confusion we will omit the parameters at all.
\end{rem}

The advantage of this formulation is that the fluid equations are given on a time-independent domain $B$ and therefore, following ideas from \cite{heatfl}, we can use the theory and the ideas developed for the Navier-Stokes-Fourier system. However, note the system in Definition \ref{weaksolutionap1} is still coupled and depends on the geometry through condition 2 and conditions on the test functions. Therefore, it is far from straightforward how to decouple the system and solve the fluid part separately. Here we use the decoupling method based on operator splitting from \cite{trwa3} (see also \cite{SubBorARMA} where the splitting method in the context of FSI was introduced) which penalizes the fluid velocity and temperature to ensure the kinematic coupling conditions. More precisely, we split the time interval $(0,T)$ into subintervals of length $\Delta t$. The approximate solution is constructed via time-marching procedure where in each time sub-interval we solve separately the fluid and the structure sub-problems (which are continuous in time). The decoupling is achieved thanks to the relaxation of the kinematic boundary conditions which are satisfied only approximately via penalization with parameter $\frac{1}{\Delta t}$. Moreover, the sub-problems ``communicate'' with each other via the penalization terms. Physically, in the approximate problems we make the interface transparent so the fluid can pass through it (see figure \ref{penalpic2}). In the limit $\Delta t\to 0$, the kinematic coupling conditions are satisfied and the interface $\Gamma^w(t)$ becomes impermeable, and therefore we obtain a weak solution to the extended problem from Definition \ref{weaksolutionap1}.

\begin{figure}[h!]
\centering\includegraphics[scale=0.2]{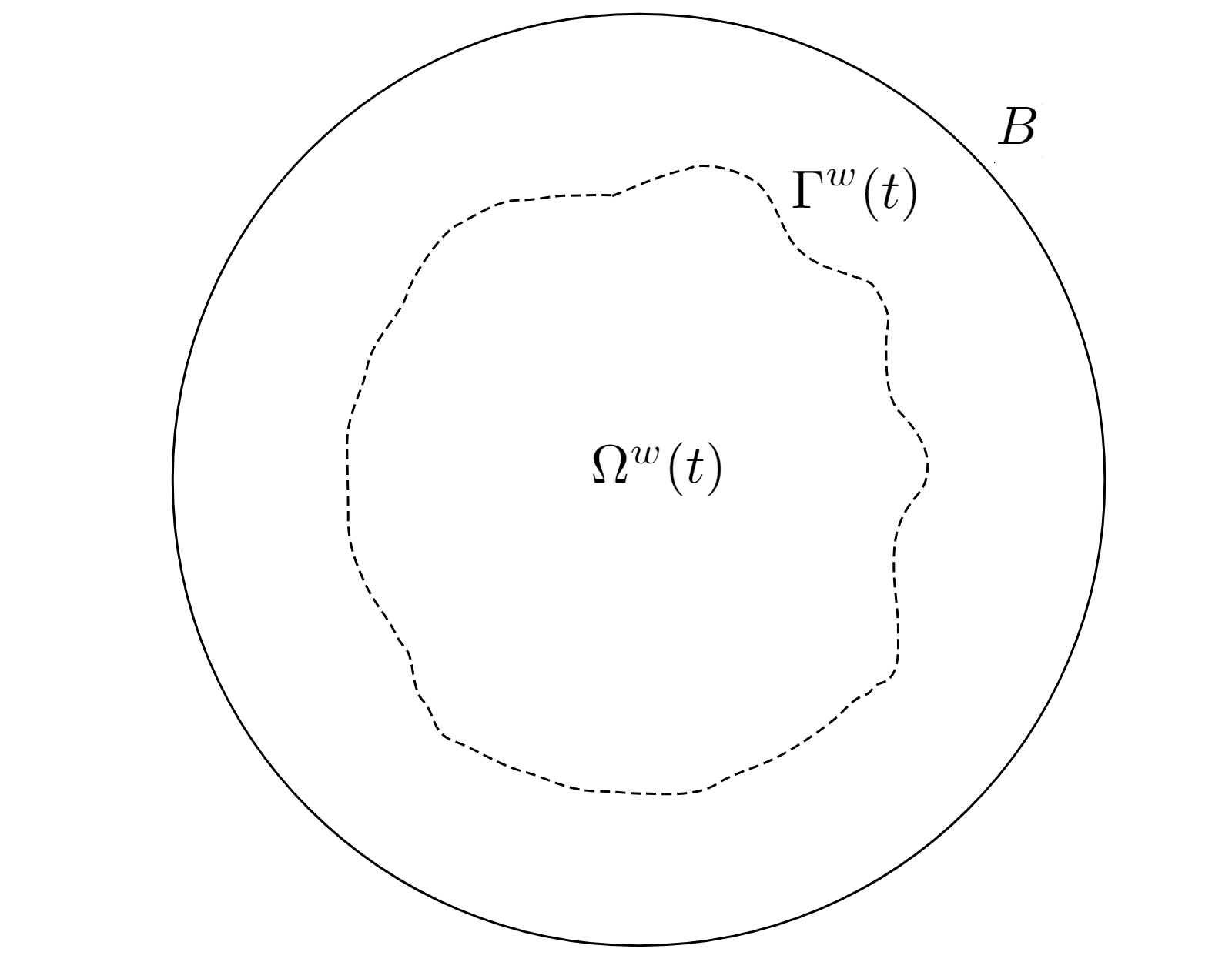}
\caption{The penalized fluid problem on an extended domain $B$. Here, the boundary $\Gamma^w(t)$ is dashed in order to emphasize that the fluid can pass through it.}
\label{penalpic2}
\end{figure}
\newpage

\subsection{Vanishing density outside of the physical domain $\Omega^w$}
Before constructing a solution to the extended problem, we show that density vanishes outside of the physical domain:

\begin{lem}\label{pressvanish}
Let $\rho \in L^\infty(0,T;L^3(B)), \bb{u} \in L^2(0,T; H_0^1(B))$ satisfy the renormalized continuity equation $\eqref{RCE1}$ on $(0,\tau)$ for any $\tau\in(0,T)$ and let $w\in L^\infty(0,T; H^2(\Gamma))\cap W^{1,\infty}(0,T; L^2(\Gamma))\cap H^1(\Gamma_T)$. If
\begin{eqnarray*}
    \rho(0,\cdot)_{|B\setminus \Omega^{w_0}} \equiv 0
\end{eqnarray*}
then
\begin{eqnarray*}
    \rho_{|B\setminus \Omega^w(t)}(t,\cdot) \equiv 0, \quad \text{for a.a. }t \in (0,T).
\end{eqnarray*}
\end{lem}
\begin{proof}
The proof follows \cite[Lemma 4.1]{commoving}, which is based on level set approach. Since our flow function is quite different and less regular to the one in the mentioned lemma, we provide a full proof. \\

Let $g_0 \in C^3(B)$ be such that
\begin{eqnarray*}
g_0= \begin{cases}
0, \quad &\text{on }\Gamma\cup \partial B, \\
>0, \quad &\text{in } B\setminus \Omega,\\
<0,\quad  &\text{elsewhere in } \mathbb{R}^3.
\end{cases}
\end{eqnarray*}
and
\begin{eqnarray}
\nabla g_0(X) =h(d(X)) \bb{n}(\pi(X)),\quad 
h(x)\geq c>0 \label{d0assumption1}
\end{eqnarray}
in a small neighbourhood of $\Gamma$ denoted by $S_\Gamma$. Signed distance function $d$ and projection $\pi$ are defined by \eqref{d} and \eqref{p} in Appendix A.  Denote
\begin{eqnarray*}
    \Phi := \tilde{\Phi}_w^B, \quad \bb{V}:= \partial_t \tilde{\Phi}_w^B\circ (\tilde{\Phi}_w^B)^{-1}.
\end{eqnarray*}
where the flow function $\tilde{\Phi}_w^B:[0,T]\times B\to B$ is defined precisely in $\eqref{domainflow}$.  
We introduce the function $g:[0,T]\times B \to \mathbb{R}$ defined as
\begin{eqnarray*}
g(t,X) := g_0(\Phi^{-1}(t,X))
\end{eqnarray*}
which is a solution to the following transport equation
\begin{eqnarray*}
\partial_t g + \nabla g \cdot  \bb{V}=0, \quad
g(0, \cdot  ) = g_0.
 \end{eqnarray*}

Before we proceed, let us calculate on $[0,T]\times\Phi(t,S_\Gamma)$
\begin{eqnarray*}
&&\nabla g (t,X) =  \nabla \big(g_0(\Phi^{-1}(t,X)) \big)=  \nabla g_0\big(\Phi^{-1}(t,X)\big) \cdot \nabla \Phi^{-1}(t,X) \nonumber \\[2mm]
&&= \partial_{\bb{n}(\pi(X))} g_0(\Phi^{-1}(t,X)) \bb{n}(\pi(X))\cdot\nabla \Phi^{-1}(t,X) =  \partial_{\bb{n}(\pi(X))} g_0(\Phi^{-1}(t,X))  \partial_{\bb{n}(\pi(X))} \Phi^{-1}(t,X),
\end{eqnarray*}
and since
\begin{eqnarray*}
    \partial_{\bb{n}(\pi(X))} \Phi^{-1}(t,X) =\frac{1}{1+ f'(d(X)) w(t,\pi(X)) }\bb{n}(\pi(X)),\quad\text{ in } [0,T]\times\Phi(t,S_\Gamma),
\end{eqnarray*}
one obtains by $\eqref{normal}$, $\eqref{jacobianphi1},\eqref{jacobianphi2}$ and $\eqref{d0assumption1}$ 
\begin{eqnarray}\label{normalpositive}
    \nabla g(t,X) = \tilde{h}(t,X) \bb{n}(\pi(X)), \quad 0< c \leq  \tilde{h}(t,X) \leq C, \quad \text{for all }(t,X)\in[0,T]\times\Phi(t,S_\Gamma).
\end{eqnarray}
${}$\\

Next, fix $\xi>0$ and let us test $\eqref{RCE1}$ by
\begin{eqnarray*}
    \varphi=\Big[ \min \Big\{ \frac{1}{\xi}g;1 \Big\}  \Big]^+,
\end{eqnarray*}
to obtain
\begin{eqnarray}\label{passlimithere}
    \int_{B\setminus \Omega^w(t)} (\rho \varphi)(\tau) = \frac{1}{\xi}\int_0^\tau \int_{\{ 0\leq g(t,X)< \xi \}} \Big(\rho\partial_t g+\rho\bb{u} \cdot \nabla g \Big).
\end{eqnarray}
We can now calculate
\begin{eqnarray}\label{anotherone}
    \rho \partial_t g + \rho \bb{u}\cdot \nabla g= \rho \Big( \partial_t g + \bb{u}\cdot \nabla g \Big) = \rho (\bb{V} - \bb{u}) \cdot \nabla g,
\end{eqnarray}
where by $\eqref{phireg}$, $\eqref{d0assumption1}$ and $\eqref{normalpositive}$
\begin{eqnarray}\label{hardy}
    (\bb{V} - \bb{u}) \cdot \nabla g \in L^2(0,T; W_0^{1,\frac{3}{2}}(B\setminus \Omega^w(t))).
\end{eqnarray}
Denoting
\begin{eqnarray*}
    \delta(t,X) := \text{dist}(X,\Gamma^w(t)\cup \partial B), \quad \text{for } (t,X)\in[0,\tau]\times   (B\setminus \Omega^w(t)),
\end{eqnarray*}
$\eqref{hardy}$ and Hardy's inequality give us
\begin{eqnarray}\label{alsohardy}
\frac{1}{\delta}(\bb{V} - \bb{u}) \cdot \nabla g \in L^2(0,\tau;L^{\frac{3}{2}}(B\setminus \Omega^w(t))),
\end{eqnarray}
which by $\eqref{anotherone}$ and $\eqref{passlimithere}$ imply
\begin{eqnarray*}
     \int_{B\setminus \Omega^w(t)} (\rho \varphi)(\tau) = \frac{1}{\xi}\int_0^\tau \int_{\{ 0\leq g(t,X)< \xi \}} \delta\underbrace{\rho\frac{1}{\delta}(\bb{V} - \bb{u}) \cdot \nabla g }_{\in L^2(0,T; L^1(B))}.
\end{eqnarray*}
Thus, taking into consideration
\begin{eqnarray*}
    \mathcal{M}(\{ 0\leq g(t,X)< \xi \}) \to 0, \quad \text{as }\xi\to 0,
\end{eqnarray*}
which follows from the uniform boundedness of $\Phi$ in $C^{0,\alpha}([0,T];C^{0,1-2\alpha}(B))$, $0<\alpha<1/2$, the proof of this lemma will follow by passing to the limit $\xi\to 0$ if we show
\begin{eqnarray*}
(t,X)\in[0,T]\times (B\setminus \Omega^w(t)), \quad  0\leq g(t,X)< \xi \implies   \frac{\delta(t,X)}{\xi}\leq C.
\end{eqnarray*}
 We split the set $B\setminus \Omega^w(t)=S_1(t)\cup S_2(t)$, where
\begin{eqnarray*}
    &&S_1(t):=\{X\in B\setminus\Omega^w(t): \text{dist}(X,\Gamma^w(t))>\text{dist}(X,\partial B)\},\\
    &&S_2(t):=\{X\in B\setminus\Omega^w(t): \text{dist}(X,\Gamma^w(t))\leq\text{dist}(X,\partial B)\}
\end{eqnarray*}
and fix a small enough $\xi_0>0$. First, it is easy to conclude that
\begin{eqnarray*}
    (t,X)\in [0,T]\times S_2(t) , \quad 0\leq g(t,X)=g_0(X)< \xi\leq \xi_0 \implies   \frac{\delta(t,X)}{\xi}\leq C,
\end{eqnarray*}
by the regularity of $g_0$ and $\partial B$, since $\Phi = id$ near $\partial B$. Next, since $\xi_0$ was chosen to be small, one has
\begin{eqnarray*}
    (t,X)\in [0,T]\times S_1(t), \quad 0\leq g(t,X)=g_0(\Phi^{-1}(t,X)) <\xi \leq \xi_0 \implies \Phi^{-1}(t,X) \in S_\Gamma.
\end{eqnarray*}
so
\begin{eqnarray*}
    &&\delta(t,X) \leq |X - w(t,\pi(X)) | \\
    &&\leq \frac{1}{\Big|\min\limits_{[0,T]\times \Phi(t.S_\Gamma)} \nabla g(t,X) \cdot \bb{n}
    (\pi(X))\Big|} (g(t,X) - \underbrace{g(w(t,\pi(X)))}_{=0}) \leq C\xi,
\end{eqnarray*}
from $\eqref{normalpositive}$. Thus, we can pass to the limit $\xi_0\geq\xi\to 0$, which concludes the proof.
\end{proof}

\begin{cor}
Under the assumptions of previous lemma, we have:
\begin{eqnarray*}
    p_{M}(\rho,\vartheta)=0, \quad \text{on } (0,T)\times (B\setminus \Omega^w(t)).
\end{eqnarray*}
\end{cor}
\begin{proof}
This directly follows from the previous lemma and the estimate (see \cite[Page 54, Section 3.2]{feireislnovotny})
\begin{eqnarray*}
    0\leq p_{M}(\rho,\vartheta)\leq c(\rho^{\frac{5}{3}}+\rho\vartheta).
\end{eqnarray*}

\end{proof}

\subsection{The splitting}\label{splitting}

We split the time interval to $N \in \mathbb{N}$ sub-intervals of length $\Delta t = T/N$ (the time derivatives are not discretized). We split the extended problem into two sub-problems, the fluid sub-problem (FSP) and the structure sub-problem (SSP). The splitting is done in the coupled momentum equation $\eqref{coupledmomeq1}$ and the coupled entropy inequality $\eqref{coupledenteq1}$, and the kinematic coupling conditions $\eqref{kinc}$ are not preserved after the splitting. More precisely, we introduce two auxiliary unknowns $\bb{v}$ and $\tau$ representing the traces of the fluid velocity and the fluid temperature on the interface, respectively:
\begin{eqnarray*}
\bb{v}:=\gamma_{|\Gamma^w}\bb{u}   \quad \tau:=\gamma_{|\Gamma^w}\vartheta.
\end{eqnarray*}
Note that the kinematic coupling conditions are not satisfied on the level of approximate solutions, i.e. in general $\bb{v}\neq \partial_t w\bb{n}$ and $\tau\neq\vartheta$. However, penalty terms will be included in the decoupled equations, which will ensure that kinematic coupling conditions are satisfied in the limit $\Delta t \to 0.$
The fluid and the structure sub-problems are solved one at the time through a time-marching scheme as it is represented on the figure $\ref{timemarching}$.

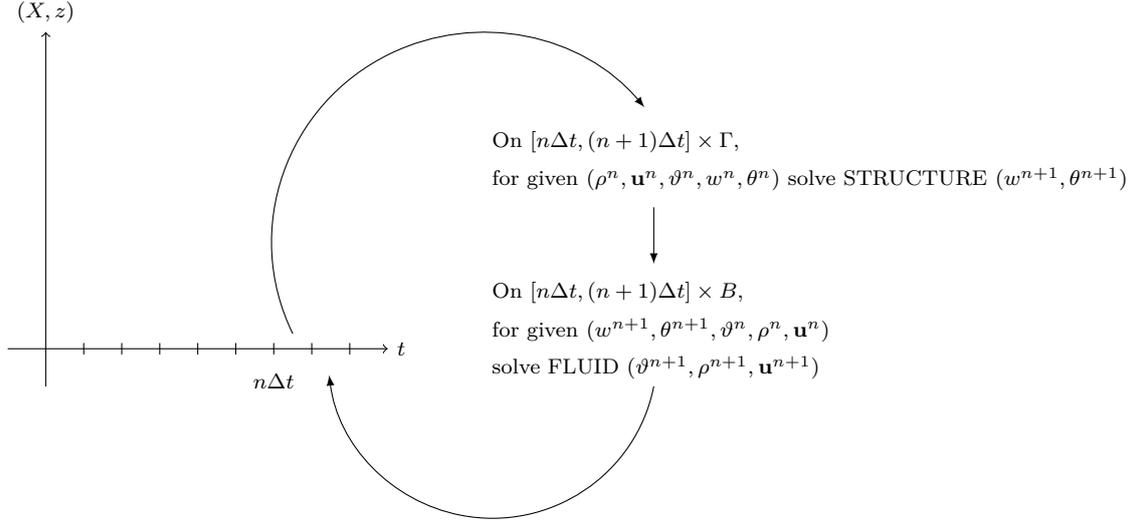
\begin{figure}[!h]
\begin{center}
\begin{tikzpicture}[domain=0:3]

\draw[->] (-0.5,0) -- (4.5,0) node[right] {{\footnotesize $t$}};
\draw[->] (0,-0.5) -- (0,4.2) node[above] {{\footnotesize$ (X,z)$}};
\node (v1) at (0.5,0.2) {};
\node (v2) at (0.5,-0.2) {};
\draw (v1) edge (v2);

\node (v4) at (1,0.2) {};
\node (v3) at (1,-0.2) {};
\node (v5) at (1.5,0.2) {};
\node (v6) at (1.5,-0.2) {};
\node (v8) at (2,0.2) {};
\node (v7) at (2,-0.2) {};
\node (v10) at (2.5,0.2) {};
\node (v9) at (2.5,-0.2) {};
\node (v12) at (3,0.2) {};
\node (v11) at (3,-0.2) {};
\node (v14) at (3.5,0.2) {};
\node (v13) at (3.5,-0.2) {};
\node (v16) at (4,0.2) {};
\node (v15) at (4,-0.2) {};
\draw (v3) edge (v4);
\draw (v5) edge (v6);
\draw (v7) edge (v8);
\draw (v9) edge (v10);
\draw (v11) edge (v12);
\draw (v13) edge (v14);
\draw (v15) edge (v16);

\node [anchor=north] (v20) at (3,-0.2) {{\footnotesize $n \Delta t$}};

\node [anchor=west] at (5.75,2.75) {{\footnotesize On $[n\Delta t, (n+1)\Delta t]\times \Gamma$,}};  
\node [anchor=west] at (5.75,2.25) {{\footnotesize  for given  ($\rho ^n, \bb{u}^{n},\vartheta^n,{w}^{n},\theta^n)$  solve STRUCTURE $({w}^{n+1}, \theta^{n+1})$}};
\draw[-latex] (3.25,0.2) arc (-154.0988:-320:2.7792);

\node (v17) at (8,2) {};
\node  [anchor=west] at (5.75,0.75) {\footnotesize On $[n\Delta t, (n+1)\Delta t]\times B$,};

\node [anchor=west] (v19) at (5.75,0.25) {{\footnotesize for given (${w}^{n+1}, \theta^{n+1}, \vartheta^{n}, \rho^n, \mathbf{u}^{n})$ }};
\node [anchor=west] (v20) at (5.75,-0.25) {{\footnotesize solve
FLUID $(\vartheta^{n+1}, \rho^{n+1}, \bb{u}^{n+1})$}};
\node (v18) at (8,1) {};
\draw[-latex] (v17) edge (v18);
\draw[-latex] (8,-0.5) arc (-11.0789:-173:2.1636);
\end{tikzpicture}
\caption{The diagram of the solving procedure. Here, the STRUCTURE and FLUID solvers corespond to $(SSP)$ and $(FSP)$ systems given in the next section.}
\label{timemarching}
\end{center}
\end{figure}

For $1 \leq n \leq N-1$, the sub-problems consist of the following equations, corresponding to the equations of the weak formulation of the extended problem  in the sense of Definition $\ref{weaksolutionap1}$:\\

\noindent
\textbf{The structure sub-problem} on $[n\Delta t, (n+1)\Delta t] \times \Gamma$: 
\begin{enumerate}
\item The structure part of the coupled momentum equation $\eqref{coupledmomeq1}$;
\item The structure part of the coupled entropy equation $\eqref{coupledenteq1}$;
\item The structure part of the energy inequality. $\eqref{ExtenededEiIn}$
\end{enumerate}

\noindent
\textbf{The fluid sub-problem} on $[n\Delta t, (n+1)\Delta t]\times B$: 
\begin{enumerate}
\item The renormalized continuity equation $\eqref{RCE1}$;
\item The fluid part of the coupled momentum equation $\eqref{coupledmomeq1}$;
\item The fluid part of the entropy inequality $\eqref{coupledenteq1}$;
\item The fluid part of the energy inequality $\eqref{ExtenededEiIn}$.\\
\end{enumerate}

\noindent
We now go on to define the sub-problems precisely.

\subsection{The sub-problems}
Denote the translation in time by $-\Delta t$ as
\begin{eqnarray*}
    T_{\Delta t}f(t):=\begin{cases}
    f(t-\Delta t),& \quad t\in [(n-1)\Delta t, n\Delta t], n\geq 1, \\
    f(0),& \quad t\in [0, \Delta t].
    \end{cases}
\end{eqnarray*}

${}$\\

We now introduce the approximation scheme: \\

\noindent
\underline{\textbf{The structure sub-problem (SSP)}:} \\
By induction on $n\geq 0$, assume that:
\begin{itemize}
    \item[] \textbf{Case} $n=0$: $w^0(0,\cdot):=w_0$,~ $\partial_t w^0(0,\cdot): = v_0$,~ $\theta^0(0,\cdot) := \theta_0$ and
    \begin{eqnarray*}
        \bb{v}^0(t,\cdot): = v_0\bb{n},\quad  \tau^0(t,\cdot):= \theta_0, \quad \text{for }t\in[-\Delta t,0];
    \end{eqnarray*}
    \item[] \textbf{Case} $n\geq 1$: the solution $(w^n,\theta^n)$ of $(SSP)$ and the solution $(\vartheta^n,\rho^n,\bb{u}^n)$ of $(FSP)$ (defined below) are already obtained.
\end{itemize}

Find $(w^{n+1},\theta^{n+1})$ so that:
\begin{enumerate}
    \item $w^{n+1} \in W^{1,\infty}(n \Delta t, (n+1)\Delta t; L^2(\Gamma))\cap  L^\infty(n \Delta t, (n+1)\Delta t;H^2(\Gamma))$,\\ $\alpha_1 w^{n+1}\in H^1(n \Delta t, (n+1)\Delta t;H^1(\Gamma))$, $\alpha_2 w^{n+1}\in W^{1,\infty}(n \Delta t, (n+1)\Delta t;H^1(\Gamma))$,\\
    $\theta^{n+1} \in L^\infty(n \Delta t, (n+1)\Delta t; L^2(\Gamma))\cap L^2(n \Delta t, (n+1)\Delta t; H^1(\Gamma))$.
    \item $w^{n+1}(n \Delta t, \cdot) =w^{n}(n \Delta t,  \cdot), ~~ \partial_t w^{n+1}(n \Delta t, \cdot) = \partial_t w^{n}(n \Delta t, \cdot)$ and $  \theta^{n+1}(n \Delta t, \cdot) = \theta^{n}(n \Delta t, \cdot)$ in the weakly continuous sense in time.
    \item The following structure heat equation
    \begin{eqnarray}
        &&(1-\delta)\bint_{n\Delta t}^{(n+1)\Delta t}\bint_{\Gamma}  \theta^{n+1} \partial_t\tilde{\psi}-\delta\bint_{n\Delta t}^{(n+1)\Delta t}\bint_{\Gamma} \frac{ \theta^{n+1} - T_{\Delta t} \tau^{n}}{\Delta t}\tilde{\psi} \nonumber \\[2mm]
        &&\quad - \bint_{n\Delta t}^{(n+1)\Delta t}\bint_{\Gamma}\nabla \theta^{n+1}\cdot \nabla \tilde{\psi} -\bint_{n\Delta t}^{(n+1)\Delta t}\bint_{\Gamma}\nabla \partial_t w^{n+1}\cdot \nabla \tilde{\psi} = \bint_{n\Delta t}^{(n+1)\Delta t}\frac{d}{dt}\int_\Gamma \theta^{n+1}\tilde{\psi} \nonumber \\&& \label{heateqSSP}
    \end{eqnarray}
holds for all $\tilde{\psi}\in C^\infty([n\Delta t,(n+1)\Delta t]\times \Gamma)$.
\item The following plate equation
\begin{eqnarray}
&&(1-\delta)\bint_{n\Delta t}^{(n+1)\Delta t}\bint_{\Gamma}\partial_t w^{n+1} \partial_t \psi-\delta\bint_{n\Delta t}^{(n+1)\Delta t}\bint_{\Gamma}\ddfrac{\partial_t w^{n+1} - T_{\Delta t} \bb{v}^{n}\cdot \bb{n}}{\Delta t}\psi\nonumber \\ && -\bint_{n\Delta t}^{(n+1)\Delta t}\bint_{\Gamma}\Delta w^{n+1} \Delta \psi-\alpha_1\bint_{n\Delta t}^{(n+1)\Delta t}\int \nabla \partial_t w^{n+1} \cdot \nabla \psi\nonumber \\ &&
+ \alpha_2 \bint_{n\Delta t}^{(n+1)\Delta t}\bint_\Gamma \nabla \partial_t w^{n+1}\cdot \partial_t\nabla\psi  +\bint_{n\Delta t}^{(n+1)\Delta t}\bint_{\Gamma}\nabla \theta^{n+1} \cdot \nabla \psi\nonumber \\ &&=(1-\delta)\bint_{n\Delta t}^{(n+1)\Delta t}\frac{d}{dt}\int_\Gamma \partial_t w^{n+1}\psi + \alpha_2\bint_{n\Delta t}^{(n+1)\Delta t}\frac{d}{dt}\int_\Gamma \partial_t \nabla  w^{n+1}\cdot \nabla \psi \label{plateeqSSP}
\end{eqnarray}
holds for all ${\psi}\in C^\infty([n\Delta t,(n+1)\Delta t]\times \Gamma)$.
\item The following energy inequality
\begin{eqnarray}
&&\frac{\delta}{2\Delta t}\int_{n \Delta t}^{t}\big(||\partial_t
w^{n+1}-T_{\Delta t}\bb{v}^{n}\cdot \bb{n}||_{L^2(\Gamma)}^2 +||\partial_t
w^{n+1} ||_{L^2(\Gamma)}^2\big)+\frac{1-\delta}{2}|| \partial_t w^{n+1}(t)||_{L^2(\Gamma)}^2\nonumber \\
&&\quad+\frac{\delta}{2\Delta t}\int_{n \Delta t}^{t}\big(||\theta^{n+1}-T_{\Delta t}\tau^{n+1}||_{L^2(\Gamma)}^2 +||\theta^{n+1} ||_{L^2(\Gamma)}^2\big)+\frac{1-\delta}{2}|| \theta^{n+1}(t)||_{L^2(\Gamma)}^2 \nonumber\\
&&\quad +\frac{1}{2}|| \Delta w^{n+1}(t)||_{L^2(\Gamma)}^2 +\frac{\alpha_2}{2}||\nabla\partial_t w^{n+1}(t)||_{L^2(\Gamma)}^2+ \int_{n\Delta t}^t \int_{\Gamma}\big(\alpha_1|\nabla \partial_t w^{n+1}|^2+ | \nabla \theta^{n+1}|^2 \big) \nonumber \\
&&\leq  \frac{1-\delta}{2}|| \partial_t w^{n+1}(n\Delta t)||_{L^2(\Gamma)}^2+\frac{1}{2}|| \Delta w^{n+1}(n\Delta t)||_{L^2(\Gamma)}^2 +\frac{\alpha_2}{2}||\nabla\partial_t w(n\Delta t)||_{L^2(\Gamma)}^2\nonumber \\
&&\quad + \frac{1-\delta}{2}||\theta^{n+1}(n\Delta  t)||_{L^2(\Gamma)}^2+\frac{\delta}{2\Delta t}\int_{n \Delta t}^{t}||T_{\Delta t}\bb{v}^{n+1}||_{L^2(\Gamma)}^2+\frac{\delta}{2\Delta t}\int_{n \Delta t}^{t}||T_{\Delta t}\tau^{n+1}||_{L^2(\Gamma)}^2.\nonumber\\
&& \label{SSPeneq}
\end{eqnarray}
holds for all $t\in (n\Delta t,(n+1)\Delta t]$.
\end{enumerate}
\underline{\textbf{The fluid sub-problem (FSP)}:} \\
By induction on $n\geq 0$, assume that:
\begin{itemize}
    \item[] \textbf{Case} $n=0$: $\rho^0(0,\cdot):=\rho_{0,\delta}(\cdot)$,~ $(\rho\bb{u})^0(0,\cdot): = (\rho\bb{u})_{0,\delta}(\cdot)$,~ $(\rho s(\rho,\vartheta))^0(0,\cdot) := (\rho_{0,\delta} s(\rho_{0,\delta},\vartheta_{0,\delta}))$;
    \item[] \textbf{Case} $n\geq 1$: the solution $(\vartheta^n,\rho^n,\bb{u}^n)$ of $(FSP)$ and the solution $(w^{n+1},\theta^{n+1})$ of $(SSP)$ are already obtained.
\end{itemize}

Find $(\vartheta^{n+1}, \rho^{n+1}, \bb{u}^{n+1})$ so that:

\begin{enumerate}
\item $\rho ^{n+1} \geq 0$, $\rho^{n+1} \in L^\infty(n\Delta t, (n+1)\Delta t;L^{\frac{5}{3}}(\mathbb{R}^3))\cap L_{loc}^q([n\Delta t, (n+1)\Delta t]; \overline{B}\setminus \Gamma^w(t)), \text{ for some }q>\frac{5}{3}$,\\
$\rho^{n+1}\in L^\infty(n\Delta t, (n+1)\Delta t;L^{\beta}(\mathbb{R}^3))$, \\
$\bb{u}^{n+1},\ \nabla \bb{u}^{n+1} \in L^2((n\Delta t, (n+1)\Delta t) \times B)$,~~ $(\rho \bb{u})^{n+1} \in L^\infty(n\Delta t, (n+1)\Delta t; L^1(\mathbb{R}^3))$,\\
$\vartheta^{n+1}>0$ a.e. in $(n\Delta t, (n+1)\Delta t) \times B$,~~ $\vartheta^{n+1} \in L^\infty(n\Delta t, (n+1)\Delta t; L^4(B))$,\\
$\vartheta^{n+1},\ \nabla \vartheta^{n+1} \in L^2((n\Delta t, (n+1)\Delta t) \times B)$, ~~$\log \vartheta^{n+1},\ \log \nabla \vartheta^{n+1} \in L^2((n\Delta t, (n+1)\Delta t) \times B)$,\\
$(\rho s)^{n+1}, (\rho s \bb{u})^{n+1},~~ \frac{q^{n+1}}{\vartheta^{n+1}} \in L^1((n\Delta t, (n+1)\Delta t) \times B)$.
\item $\rho^{n+1}(n\Delta t)= \rho^{n}(n\Delta t)$, $(\rho \bb{u})^{n+1}(n\Delta t)=(\rho \bb{u})^{n}(n\Delta t)$ in weakly continuous sense in time.
\item The renormalized continuity equation
\begin{align*}
&\int_{n\Delta t}^{(n+1)\Delta t} \bint_B \rho^{n+1} B(\rho^{n+1})( \partial_t \varphi +\bb{u}^{n+1}\cdot \nabla \varphi) \nonumber \\
&= \int_{n\Delta t}^{(n+1)\Delta t} \bint_B b(\rho^{n+1})(\nabla\cdot \bb{u}^{n+1}) \varphi - \int_{n\Delta t}^{(n+1)\Delta t}\frac{d}{dt}\int_{B} \rho^{n+1} B(\rho^{n+1}) \varphi , 
\end{align*}
holds for all $\varphi \in C^\infty([n\Delta t,(n+1)\Delta t]\times \mathbb{R}^3)$ and any $b\in L^{\infty}(0,\infty) \cap C[0,\infty)$ such that $b(0)=0$ with $B(x)=B(1)+\int_1^x \frac{b(z)}{z^2}dz$
\item The following momentum equation
\begin{align}
    &\int_{n\Delta t}^{(n+1)\Delta t} \bint_B \rho^{n+1} \bb{u}^{n+1} \cdot\partial_t \boldsymbol\varphi + \int_{n\Delta t}^{(n+1)\Delta t} \bint_B (\rho^{n+1} \bb{u}^{n+1} \otimes \bb{u}^{n+1}):\nabla\boldsymbol\varphi\nonumber\\
    &  +\int_{n\Delta t}^{(n+1)\Delta t} \bint_B  p_{\eta,\delta}^{n+1}(\vartheta^{n+1},\rho^{n+1}) (\nabla \cdot \boldsymbol\varphi)-  \int_{n\Delta t}^{(n+1)\Delta t} \bint_B   \mathbb{S}_\omega^{n+1}(\vartheta^{n+1}, \nabla\bb{u}^{n+1}): \nabla \boldsymbol\varphi\nonumber\\
    & -\delta
    \int_{n\Delta t}^{(n+1)\Delta t}  \bint_\Gamma\ddfrac{\bb{v}^{n+1}-\partial_t w^{n+1}\bb{n}}{\Delta t} \cdot\boldsymbol\psi   = \int_{n\Delta t}^{(n+1)\Delta t}  \frac{d}{dt} \int_{B}\rho^{n+1}\mathbf{u}^{n+1}\cdot\boldsymbol\varphi.\label{momeqFSP}
\end{align}
holds for all $\boldsymbol\varphi \in C_c^\infty([n\Delta t,(n+1)\Delta t]\times B)$ and $\boldsymbol\psi\in C^\infty(\Gamma_T)$ such that $\gamma_{|\Gamma^{w^{n+1}}}\boldsymbol\varphi = \boldsymbol\psi$ on $\Gamma_T$, where
\begin{eqnarray*}
    &&\bb{v}^{n+1}=\gamma_{|\Gamma^{w^{n+1}}}\bb{u}^{n+1}  ,\\[2mm]
    &&\mathbb{S}_\omega^{n+1}(\vartheta, \nabla \bb{u}):=\mu_\omega^{n+1}(\vartheta) \big( \nabla \bb{u} + \nabla^\tau \bb{u}-\frac{2}{3} \nabla \cdot \bb{u} \big) + \zeta_\omega^{n+1}(\vartheta) \nabla \cdot \bb{u},\\[2mm]
    &&\mu_\omega^{n+1}:=\mu_\omega^{w^{n+1}}, \quad \zeta_\omega^{n+1}:=\zeta_\omega^{w^{n+1}}, \quad p_{\eta,\delta}^{n+1}:=p_{\eta,\delta}^{w^{n+1}},
\end{eqnarray*}
with, $\mu_\omega^{w}$, $\zeta_\omega^{w}$ and $p_{\eta,\delta}^{w}$ being defined in Section $\ref{penalization}$.
\item The following entropy balance 
\begin{eqnarray}
&&\int_{n \Delta t}^{(n+1)\Delta t} \int_B \rho^{n+1} s_\eta^{n+1}( \partial_t \varphi + \bb{u}^{n+1}\cdot \nabla \varphi) - \int_{n \Delta t}^{(n+1)\Delta t} \int_B\frac{\kappa_{\nu}^{n+1}(\vartheta^{n+1}) \nabla \vartheta^{n+1} \cdot \nabla \varphi}{\vartheta^{n+1}} \nonumber\\
&&+ \langle \sigma_{\omega,\nu}^{n+1}; \varphi \rangle_{[\mathcal{M},C]([n\Delta t,(n+1)\Delta t]\times \overline{B})} +\lambda \int_{n \Delta t}^{(n+1)\Delta t} \int_B  \vartheta^4  \varphi\nonumber\\
&&-\delta\bint_{n \Delta t}^{t} \int_\Gamma \frac{\tau^{n+1} - \theta^{n+1}}{\Delta t}\tilde{\psi} = \Big( \int_{B} \rho^{n+1} s_\eta^{n+1} \varphi\Big)\Big|_{n\Delta t}^{(n+1)\Delta t} \label{entineqFSP}
\end{eqnarray}
holds for any $\varphi \in C^\infty([n\Delta t,(n+1)\Delta t]\times B)$ and $\tilde{\psi}\in C^\infty([n\Delta t,(n+1)\Delta t]\times\Gamma)$ such that $\gamma_{|\Gamma^{w^{n+1}}}\varphi = \psi$, where
\begin{eqnarray*}
  \sigma_{\omega,\nu}^{n+1}\geq \frac{1}{\vartheta^{n+1}}\Big( \mathbb{S}_\omega(\vartheta^{n+1}, \nabla \bb{u}^{n+1}):\nabla \bb{u}^{n+1}
+\frac{\kappa_{\nu}^{n+1}(\vartheta^{n+1})|\nabla \vartheta^{n+1}|^2}{\vartheta^{n+1}}\Big) 
\end{eqnarray*}
and
\begin{eqnarray*}
    \tau^{n+1}=\gamma_{|\Gamma^{w^{n+1}}}\vartheta^{n+1}, \quad s_\eta^{n+1}:=s_\eta^{w^{n+1}}, \quad \kappa_\nu^{n+1}:=\kappa_{\nu}^{w^{n+1}},
\end{eqnarray*}
with $s_\eta^{w}$ and $\kappa_{\nu}^{w}$ being defined in Section $\ref{penalization}$.
\item The following energy inequality
\begin{eqnarray}
    &&\displaystyle{\int_{B} \Big( \frac{1}{2} \rho^{n+1} |\bb{u}^{n+1}|^2 + \rho^{n+1} e_\eta^{n+1}(\vartheta^{n+1},\rho^{n+1})+\frac{\delta}{\beta-1}(\rho^{n+1})^{\beta}\Big)(t)} +\displaystyle{\lambda \int_{n \Delta t}^{t}\int_{B} (\vartheta^{n+1})^5 } \nonumber \\
    &&+\displaystyle{\frac{\delta}{2\Delta t}  \int_{n \Delta t}^{t}\int_{\Gamma} \Big(|\bb{v}^{n+1} - \partial_t w^{n+1}\bb{n}|^2 + |\bb{v}^{n+1}|^2 \Big)}+\displaystyle{\frac{\delta}{2\Delta t}  \int_{n \Delta t}^{t}\int_{\Gamma} \Big(|\tau^{n+1} - \theta^{n+1}|^2 + |\tau^{n+1}|^2 \Big)}\nonumber \\
    &&\leq \displaystyle{\int_{B} \Big( \frac{1}{2} \rho^n|\bb{u}^n|^2 + \rho^n e_\eta^n(\vartheta^n,\rho^n)  \Big)(n\Delta t) }+\frac{\delta}{2\Delta t}\int_{n \Delta t}^{t}\int_{\Gamma} \Big( | \partial_t w^{n+1}|^2+ | \theta^{n+1}|^2\Big), \label{fspen}
\end{eqnarray}
holds for any $t \in [n\Delta t, (n+1)\Delta t]$, where $e_\eta^{n+1}:=e_\eta^{w^{n+1}}$, with $e_\eta^{w}$ being defined in Section $\ref{penalization}$.
\item The total dissipation inequality
\begin{eqnarray}
    &&\displaystyle{\int_{B} \Big( \frac{1}{2} \rho^{n+1} |\bb{u}^{n+1}|^2+H_{1,\eta}^{n+1}(\rho^{n+1},\vartheta^{n+1})-(\rho^{n+1}-\overline{\rho})\frac{\partial H_{1,\eta}^{n+1}(\overline{\rho},1)}{\partial \rho}-H_{1,\eta}^{n+1}(\overline{\rho},1)+\frac{\delta}{\beta-1}(\rho^{n+1})^{\beta} \Big)(t) }\nonumber \\
    &&+   \lambda\int_{n\Delta t}^t \int_{B}  (\vt ^{n+1})^5+ \int_{n\Delta t}^t \int_{B} \frac{1}{\vartheta^{n+1}}\Big( \mathbb{S}_{\omega}^{n+1}(\vartheta^{n+1},\nabla \bb{u}^{n+1}):\nabla \bb{u}^{n+1} + \frac{\kappa_{\nu}^{n+1}(\vartheta^{n+1})|\nabla \vartheta^{n+1}|^2}{\vartheta^{n+1}}\Big)\nonumber  \\ 
    && +\frac{\delta}{2\Delta t}  \int_{n \Delta t}^{t}\int_{\Gamma} \Big( |\bb{v}^{n+1} - \partial_t w^{n+1}\bb{n}|^2 + |v^{n+1}|^2 \Big) + \displaystyle{\frac{\delta}{2\Delta t}  \int_{n \Delta t}^{t}\int_{\Gamma} \Big(|\tau^{n+1} - \theta^{n+1}|^2 + |\tau^{n+1}|^2 \Big)}\nonumber \\\nonumber \\
    &&\leq \displaystyle{\int_{B} \Big( \frac{1}{2} \rho^n|\bb{u}^n|^2 +H_{1,\eta}^{n}(\rho^n,\vartheta^n) -(\rho^{n}-\overline{\rho})\frac{\partial H_{1,\eta}^{n}(\overline{\rho},1)}{\partial \rho}-H_{1,\eta}^{n}(\overline{\rho},1)+\rho^n e(\vartheta^n,\rho^n) \Big)(n\Delta t) }\nonumber\\
    &&+\frac{\delta}{2\Delta t}\int_{n \Delta t}^{t} \int_{\Gamma} | \partial_t w^{n+1}|^2 -\delta\int_{n\Delta t}^t\int_\Gamma \frac{\tau^{n+1}-\theta^{n+1}}{\Delta t} , \label{dissineq}
\end{eqnarray}
holds for any $t \in [n\Delta t, (n+1)\Delta t]$, where the approximate Helmholtz function is defined as
\begin{eqnarray*}
H_{\overline{\vartheta},\eta}^{n+1}(\rho,\vartheta):=\rho(e_\eta^{n+1}(\rho,\vartheta)-\overline{\vartheta}s_\eta^{n+1}(\rho,\vartheta))\mbox{ for some positive constant } \overline{\vartheta}>0.
\end{eqnarray*}
and $\overline{\rho}$ is chosen so that $\int_{B} \rho - \overline{\rho}=0$ for a.a. $t\in(0,T)$.
\end{enumerate}

\subsection{Solving the sub-problems}
Here,  we give a brief explanation how the each of the sub-problems is solved and how the estimates they satisfy are obtained.\\

First, in order to solve $(SSP)$, one can span $w^{n+1},\theta^{n+1}$ in finite Galerkin bases, then solve the problem by the standard ODE theory (see for example \cite[Lemma 3.1]{trwa3}), while the inequality $\eqref{SSPeneq}$ (in finite bases it is actually an equality) is obtained by choosing $\tilde{\psi}=\theta^{n+1}$ in $\eqref{heateqSSP}$ and $\psi=\partial_t w^{n+1}$ in $\eqref{plateeqSSP}$ and summing up these two identities, where   the identity $2(a-b)b = a^2 - b^2 - (a-b)^2$ is also used. Then, it is a routine matter to pass to the limit in the number of basis functions to prove the desired result.

Next, to solve $(FSP)$, notice that this system is almost the same as in \cite[Definition 3.1]{heatfl}, where $\varepsilon$ is replaced by $\frac{\Delta t}{\delta}$ and there is an additional penalization term in the entropy equation $\eqref{entineqFSP}$
\begin{eqnarray}\label{thatterm}
 \delta\int_{n \Delta t}^{(n+1)\Delta t} \int_\Gamma \frac{\tau^{n+1} - \theta^{n+1}}{\Delta t}\psi.
\end{eqnarray}
To obtain a solution of $(FSP)$, one can use \cite[Theorem 3.1]{heatfl} (which is proved as \cite[Theorem 3.1]{feireislnovotny}), where the above penalization term $\eqref{thatterm}$ can be dealt with as follows. At the highest level of approximation (see \cite[Section 3.4]{feireislnovotny}), where fluid velocity is spanned in a finite Galerkin basis and the continuity equation is damped, the entropy inequality is replaced by the internal energy equation \cite[(3.55)]{feireislnovotny} with the modified penalization term defined on entire $B$ as
\begin{eqnarray}\label{thatterm2}
    \delta \varphi_\xi\vartheta^{n+1}\frac{\vartheta^{n+1} - \theta_\xi^{n+1}}{\Delta t},\qquad \xi>0.
\end{eqnarray}
Here, $\varphi_\xi\in C_0^\infty([0,T]\times B)$ is a non-negative function such that $\varphi_\xi\to \delta_{\Gamma^{w^{n+1}}}$ as $\xi\to 0$,
\begin{eqnarray}
    \theta_\xi^{n+1}:=\int_0^T \int_B \iota_\xi(t-\tau,X-Y)(\theta^{n+1}\circ \Phi_{w^{n+1}}^{-1})\delta_{\Gamma^{w^{n+1}}} ~d\tau dY
\end{eqnarray}
where $\iota$ is a standard time-space mollifying kernel, and $\Phi_{w^{n+1}}$ is the boundary flow function defined in $\eqref{ab}$. Note that $\theta_\xi^{n+1}\in C^\infty([0,T]\times B)$ and for $\xi$ small enough vanishes on $\partial B$. Moreover, the term given in $\eqref{thatterm2}$ is regular enough and compatible with the comparison principle (see \cite[Lemma 3.2]{feireislnovotny}), which ensures the positivity of the approximate fluid temperature at this approximation level. Now, we can solve this penalized internal energy equation by treating the term $\eqref{thatterm2}$ as a compact perturbation, while the continuity equation and penalized momentum equation can be solved in the same way, so the approximate solution of the entire penalized fluid system follows by the fixed-point argument. The internal energy equation is then divided by $\vartheta^{n+1}$, and we can pass the Galerkin limit together with $\xi\to 0$, by which the penalization term $\eqref{thatterm2}$ becomes $\eqref{thatterm}$. Afterwards, in the limiting system, the vanishing density damping limit is done and we obtain our solution of $(FSP)$.\\

Next, in order to obtain the solutions of $(SSP)$ and $(FSP)$ inductively on the whole time interval $[0,T]$, it is enough to prove the uniform estimates on the time interval $[0,n\Delta t]$. This will be the subject in the following section.
\subsection{The uniform bounds of the approximate solutions}
Assume that, inductively, we have solved the structure sub-problem and the fluid sub-problem on $[0,m\Delta t]$, for some $1\leq m \leq N-1$. Let us denote, for simplicity,
\begin{eqnarray*}
    g(t):= g^{n+1}(t), \quad \text{for } t\in [n\Delta t, (n+1)\Delta t],
\end{eqnarray*}
where $g$ is one of the functions $\vartheta, \rho, \bb{u},w, \theta$, and
\begin{eqnarray*}
    \mathbb{S}_\omega^{w} := \mathbb{S}_\omega^{n+1}, \quad \mu_\omega^{w}:=\mu_\omega^{n+1}, \quad \zeta_\omega^{w}:=\zeta_\omega^{n+1},\quad \sigma_{\omega,\nu}^w:=\sigma_{\omega,\nu}^{n+1}, \quad p_{\eta,\delta}^{w}:=p_{\eta,\delta}^{n+1},\quad H_{1,\eta}^{n+1}:=H_{1,\eta}^{w}, 
\end{eqnarray*}
for $t\in [n\Delta t, (n+1)\Delta t]$, $0\leq n\leq N-1$.

Now, for $n\leq m-2$, sum $\eqref{SSPeneq}$ and $\eqref{fspen}$ at times $(n+1)\Delta t$, sum over $n=1,...,m-2$, then add $\eqref{SSPeneq}$ and $\eqref{fspen}$ at time $t \in [(m-1)\Delta t, m \Delta t]$, and by telescoping, one obtains
\begin{eqnarray}
&&\displaystyle{\int_{B} \Big( \frac{1}{2} \rho |\bb{u}|^2 + \rho e_\eta^w(\rho,\vartheta)+\frac{\delta}{\beta-1}\rho^{\beta} \Big)(t)} + \lambda\int_{0}^t  \int_{B} \vartheta^5 +\frac{1-\delta}{2}|| \partial_t w(t)||_{L^2(\Gamma)}^2 +\frac{1}{2}|| \Delta w(t)||_{L^2(\Gamma)}^2 \nonumber \\
&& \quad+\frac{\alpha_2}{2}||\nabla \partial_t w(t)||_{L^2(\Gamma)}^2 + \frac{1-\delta}{2}||\theta(t)||_{L^2(\Gamma)}^2+ \int_{0}^t \int_{\Gamma}\big(\alpha_1 |\nabla \partial_t w|^2+ | \nabla \theta|^2 \big) + \frac{\delta}{2\Delta t}\int_{n\Delta t}^t ||T_{\Delta t} v||_{L^2(\Gamma)}^2\nonumber \\
&&\quad+\displaystyle{\frac{\delta}{2\Delta t}  \int_{0}^{t} \Big(||\bb{v} - \partial_t w\bb{n}||_{L^2(\Gamma)}^2 +||\partial_t w-T_{\Delta t}\bb{v}\cdot \bb{n}||_{L^2(\Gamma)}^2 \Big)}  + \frac{\delta}{2\Delta t}\int_{n\Delta t}^t ||T_{\Delta t} \tau||_{L^2(\Gamma)}^2\nonumber \\
&&
\quad+\displaystyle{\frac{\delta}{2\Delta t}  \int_{0}^{t} \Big(||\tau - \theta||_{L^2(\Gamma)}^2 +||\theta-T_{\Delta t}\tau||_{L^2(\Gamma)}^2 \Big)}\nonumber \\
&&  \leq \displaystyle{\int_{B} \Big( \frac{1}{2 \rho_{0,\delta }} |(\rho\bb{u})_{0,\delta }|^2 + \rho_{0,\delta } e_\eta^w(\vartheta_{0,\delta },\rho_{0,\delta })}+\frac{\delta}{\beta-1}\rho_{0,\delta }^{\beta} \Big) \nonumber\\
&&\quad+ \frac{1}{2}|| v_{0,\delta }||_{L^2(\Gamma)}^2+\frac{1}{2}|| \Delta w_{0,\delta }||_{L^2(\Gamma)}^2 + \frac{\alpha_2}{2}||\nabla v_{0,\delta}||^2+ \frac{1}{2}||\theta_{0,\delta }||_{L^2(\Gamma)}^2, \label{decenineqtotal}
\end{eqnarray}
and similarly for the total dissipation inequality $\eqref{dissineq}$
\begin{eqnarray}
    &&\displaystyle{\int_{B} \Big( \frac{1}{2} \rho |\bb{u}|^2 +
    H_{1,\eta}^w(\rho,\vartheta)-(\rho-\overline{\rho})\frac{\partial H_{1,\eta}^w(\overline{\rho},1)}{\partial \rho}-H_{1,\eta}^w(\overline{\rho},1)+ \rho e_\eta^w(\rho,\vartheta)+\frac{\delta}{\beta-1}\rho^{\beta} \Big)(t)} \nonumber \\
    &&\quad+\lambda \int_{0}^t  \int_{B} \vartheta^5+\int_{0}^t \int_{B} \frac{1}{\vartheta}\Big( \mathbb{S}_{\omega}^w(\vartheta,\nabla \bb{u}):\nabla \bb{u} + \frac{\kappa_{\nu}^w(\vartheta)|\nabla     \vartheta|^2}{\vartheta}\Big) \nonumber+\frac{1-\delta}{2}|| \partial_t w(t)||_{L^2(\Gamma)}^2 +\frac{1}{2}|| \Delta w(t)||_{L^2(\Gamma)}^2 \\
    &&\quad +\frac{\alpha_2}{2}||\nabla \partial_t w(t)||_{L^2(\Gamma)}^2 + \frac{1-\delta}{2}||\theta(t)||_{L^2(\Gamma)}^2+ \int_{0}^t \int_{\Gamma}\big[\alpha_1 |\nabla \partial_t w|^2+ | \nabla \theta|^2\big] \nonumber \\
    &&\quad+ \frac{\delta}{2\Delta t}\int_{n\Delta t}^t ||T_{\Delta t} v||_{L^2(\Gamma)}^2+\displaystyle{\frac{\delta}{2\Delta t}  \int_{0}^{t} \Big(||v - \partial_t w||_{L^2(\Gamma)}^2 +||\partial_t w-T_{\Delta t}\bb{v}\cdot \bb{n}||_{L^2(\Gamma)}^2 \Big)}  \nonumber\\
    &&\quad+ \frac{\delta}{2\Delta t}\int_{n\Delta t}^t ||T_{\Delta t} \tau||_{L^2(\Gamma)}^2+\displaystyle{\frac{\delta}{2\Delta t}  \int_{0}^{t} \Big(||\tau - \theta||_{L^2(\Gamma)}^2 +||\theta-T_{\Delta t}\tau||_{L^2(\Gamma)}^2 \Big)}  \nonumber\\
    &&\leq \displaystyle{\int_{B} \Big( \frac{1}{2 \rho_{0,\delta }} |(\rho\bb{u})_{0,\delta }|^2 +H_{1,\eta}^w( \rho_{0,\delta }, \vartheta_{0,\delta})-( \rho_{0,\delta }-\overline{\rho})\frac{\partial H_{1,\eta}^w(\overline{\rho},1)}{\partial \rho}-H_{1,\eta}^w(\overline{\rho},1)+\frac{\delta}{\beta-1}\rho^{\beta}_{0,\delta}+ \rho_{0,\delta } e_\eta^w(\vartheta_{0,\delta },\rho_{0,\delta })}\Big) \nonumber \\
    &&+ \frac{1}{4}|| v_{0,\delta }||_{L^2(\Gamma)}^2+\frac{1}{2}|| \Delta w_{0,\delta }||_{L^2(\Gamma)}^2 + \frac{1}{2}||\theta_{0,\delta }||_{L^2(\Gamma)}^2 +(1-\delta)\Big(\int_{\Gamma} \theta \Big)\Big|_{0}^{t}+\delta \int_0^t\int_\Gamma  \ddfrac{\tau-T_{\Delta t}\tau}{\Delta t}. \label{decdissineqtotal}
\end{eqnarray}
Note that the last two terms can be controlled as
\begin{eqnarray*}
(1-\delta)\Big(\int_{\Gamma} \theta \Big)\Big|_{0}^{t} \leq \frac{1-\delta}{4}||\theta(t)||_{ L^2(\Gamma)}^2 + 2(1-\delta) || \theta_{0,\delta}||_{L^2(\Gamma)}^2+C(\Gamma),
\end{eqnarray*}
and
\begin{eqnarray*}
&&\delta \int_0^t\int_\Gamma  \ddfrac{\tau-T_{\Delta t}\tau}{\Delta t} =- \delta \int_0^{\Delta t} \int_\Gamma \frac{\theta_{0,\delta}}{\Delta t} + \delta \int_{t-\Delta t}^{t} \int_\Gamma \frac{\tau}{\Delta t} \leq \delta||\theta_{0,\delta}||_{L^2(\Gamma)}^2 + \frac{\delta}{4\Delta t} \int_{t-\Delta t}^t||\tau||_{L^2(\Omega)}^2+ \frac{\delta}{\Delta t}\int_{t-\Delta t}^t 1 \\
&&\leq \delta||\theta_{0,\delta}||_{L^2(\Gamma)}^2+ \frac{\delta}{4\Delta t} \int_{t-\Delta t}^t ||\theta||_{L^2(\Gamma)}^2 + \frac{\delta}{4 \Delta t} \int_{t-\Delta t}^t||\tau - \theta||_{L^2(\Gamma)}^2 +\delta \\
&&\leq \delta||\theta_{0,\delta}||_{L^2(\Gamma)}^2 +\frac{\delta}{4} ||\theta||^{2}_{L^\infty(0,t;L^2(\Gamma))} + \frac{\delta}{4\Delta t} \int_{0}^t||\tau - \theta||_{L^2(\Gamma)}^2 + \delta
\end{eqnarray*}
where $||\theta||_{L^\infty(0,t;L^2(\Gamma))}$ is uniformly bounded from $\eqref{decenineqtotal}$.

By means of the uniform bounds $\eqref{decenineqtotal}$ and $\eqref{decdissineqtotal}$, we can inductively obtain the approximate solutions $( \vartheta, \rho, \bb{u},w, \theta)$  on the whole time interval $[0,T]$ which satisfy the following approximation problem and the corresponding uniform estimates. Notice that the approximate solution $( \vartheta, \rho, \bb{u},w, \theta)$  depends on parameters $\Delta t, \eta,\omega, \nu,\lambda$ and $\delta$. However, to avoid cumbersome notation we will not write these parameters explicitly as the indices. If it is not otherwise stated, all estimates are uniform in all parameters and generic constant $C$ does not depend on the parameters.

In order to pass to the limit in $\Delta t$ we first write the weak forms of renormalized continuity equation, coupled momentum equation and coupled  entropy balance which is satisfied by the approximate solutions. The renormalized equation is obtained directly from the definition of the approximate solutions since it involves only the quantities from the fluid sub-problem. For the other two, let us fix the admissible pair of test functions $(\boldsymbol\varphi, \psi)$. To obtain the momentum equation, sum $\eqref{momeqFSP}$ tested by $(\boldsymbol\varphi_{|[n\Delta t, (n+1)\Delta t]}, \psi_{|[n\Delta t, (n+1)\Delta t]}\mathbf{n})$ and $\eqref{plateeqSSP}$ tested by $\psi_{|[n\Delta t, (n+1)\Delta t]}$ and finally sum over $n=1,...,N-1$. Similarly, the entropy balance can be obtained by summing $\eqref{entineqFSP}$ tested by $(\varphi_{|[n\Delta t, (n+1)\Delta t]},\tilde{\psi}_{|[n\Delta t, (n+1)\Delta t]})$ and $\eqref{heateqSSP}$ tested by $\tilde{\psi}_{|[n\Delta t, (n+1)\Delta t]}$, and then summed over $n=1,...,N-1$. To summarize, the approximate solution satisfy the following equalities:
\begin{enumerate}
    \item The renormalized continuity equation
\begin{eqnarray}\label{RCETM-Ap}
    \int_0^T \int_B \rho B(\rho)( \partial_t \varphi +\bb{u}\cdot \nabla \varphi) =\int_0^T \int_B b(\rho)(\nabla\cdot \bb{u}) \varphi +\int_{B} \rho_{0,\delta} B(\rho_{0,\delta}) \varphi(0,\cdot),
\end{eqnarray}
for all $\varphi \in C_c^\infty([0,T)\times B)$ and any $b\in L^\infty (0,\infty)\cap C[0,\infty)$ such that $b(0)=0$ with $B(\rho)=B(1)+\int_1^\rho \frac{b(z)}{z^2}dz$.
\item The momentum equation
\begin{align}
    &\int_0^T \bint_B \rho \bb{u} \cdot\partial_t \boldsymbol\varphi + \int_0^T \bint_B (\rho \bb{u} \otimes \bb{u}):\nabla\boldsymbol\varphi +\int_0^T \bint_B  p_{\eta,\delta}^w(\rho,\vartheta) (\nabla \cdot \boldsymbol\varphi) -  \int_0^T \bint_B \mathbb{S}_\omega^w(\vartheta, \nabla\bb{u}): \nabla \boldsymbol\varphi\nonumber\\
    &-\delta \bint_{\Gamma_T}\ddfrac{\bb{v}\cdot\bb{n}-T_{\Delta t}\bb{v}\cdot \bb{n}}{\Delta t} \psi+
    (1-\delta)\bint_{\Gamma_T}\partial_t w \partial_t \psi -\bint_{\Gamma_T}\Delta w \Delta \psi -\alpha_1 \bint_{\Gamma_T} \partial_t \nabla w \cdot \nabla \psi \nonumber\\
    &+ \alpha_2 \bint_{\Gamma_T}\partial_t \nabla w\cdot \partial_t \nabla \psi +\bint_{\Gamma_T}\nabla \theta \cdot \nabla \psi=- \int_{B}(\rho\mathbf{u})_{0,\delta}\cdot\boldsymbol\varphi(0,\cdot)-(1-\delta) \int_\Gamma v_0\psi(0,\cdot)- \int_{\Gamma} \nabla v_0 \cdot \nabla \psi(0,\cdot) , \label{coupledmomeq-Ap}
\end{align}
for all $\boldsymbol\varphi \in C_c^\infty([0,T)\times B)$ and $\psi \in C_c^\infty([0,T)\times \Gamma)$ such that $\psi \bb{n}=\gamma_{|\Gamma^{w}}\boldsymbol\varphi$ on $\Gamma_T$. Recall that here, $\bb{v}=\gamma_{|\Gamma^{w^{n+1}}}\bb{u}$.
\item The entropy balance 
\begin{eqnarray}
    &&\int_0^T\int_{B} \rho s_\eta^w( \partial_t \varphi + \bb{u}\cdot \nabla \varphi) -\int_0^T\int_{B} \frac{\kappa_\nu^w(\vartheta) \nabla \vartheta \cdot \nabla \varphi}{\vartheta}+\langle \sigma_{\omega,\nu}^w;\varphi\rangle_{[\mathcal{M},C]([0,T]\times \overline{B})}\nonumber\\
    &&+\lambda \int_0^T\int_{B}   \vartheta^4 \varphi
    +(1-\delta)\int_0^T\int_{\Gamma} \theta \partial_t \tilde{\psi}-  \delta\int_0^T\int_{\Gamma} \ddfrac{\tau-T_{\Delta t}\tau}{\Delta t} \tilde{\psi} -\int_0^T\int_{\Gamma}  \nabla\theta \cdot \nabla \tilde{\psi} +\int_0^T\int_{\Gamma} \nabla w \cdot \nabla \partial_t \tilde{\psi}  \nonumber\\
    &&= \int_{B} \rho_{0,\delta} s_\eta^w(\vartheta_{0,\delta},\rho_{0,\delta}) \varphi(0,\cdot) +  (1-\delta)\int_{\Gamma} \theta_0 \tilde{\psi}(0,\cdot) + \int_{\Gamma} \nabla w_0 \cdot \nabla\tilde{\psi}(0,\cdot) \label{coupledenteq-Ap}
\end{eqnarray}
 for all $\varphi \in C_c^\infty([0,T)\times B)$ and $\tilde{\psi} \in C_c^\infty([0,T)\times \Gamma)$ such that $\tilde{\psi} =\gamma_{|\Gamma^{w}}\varphi$ on $\Gamma_T$, where
 \begin{eqnarray*}
   \sigma_{\omega,\nu}^w\geq \frac{1}{\vartheta}\Big( \mathbb{S}_\omega^w(\vartheta, \nabla \bb{u}):\nabla \bb{u}+\frac{\kappa_\nu^w(\vartheta)|\nabla \vartheta|^2}{\vartheta}\Big).
\end{eqnarray*}
Recall that here, $\tau=\gamma_{|\Gamma^{w}}\vartheta$.
\end{enumerate}
We finish this subsection by summarizing uniform estimate in the following lemma which is a direct consequence of $\eqref{decdissineqtotal}$ and the properties of constitutive relations (see \cite{feireislnovotny,heatfl} for more details)
\begin{lem}
The approximate solutions constructed via the splitting scheme satisfy the following uniform (in all approximation parameters) estimates \begingroup \allowdisplaybreaks
\begin{eqnarray}
 \delta \int_{0}^{t} \Big(||\tau - \theta||_{L^2(\Gamma)}^2 + ||\bb{v} - \partial_t w\bb{n}||_{L^2(\Gamma)}^2  \Big)&\leq& C \Delta t, \label{splittingbound} \\[2mm]
  \text{ess} \sup\limits_{t\in (0,T)}\int_\Gamma\Big[ |\partial_t w|^2 + |\Delta w|^2 + \alpha_2 |\partial_t \nabla w|^2+|\theta|^2\Big] &\leq & C \label{linftyplate}\\[2mm]
   \int_0^T\int_\Gamma \alpha_1 |\partial_t \nabla w|^2 + |\nabla \theta|^2 &\leq& C,\label{dissplate}\\[2mm]
  \text{ess} \sup\limits_{t\in (0,T)}\int_{B}\Big[ \delta \rho^\beta(t) +  \rho|\bb{u}  |^2  \Big] &\leq& C, \label{linftybound}\\[2mm]
  \lambda \int_0^T\int_B \vartheta^5 &\leq&C \label{vartheta5}\\[2mm]
 \int_0^T \int_B \chi_\nu^w \Big[ |\nabla \log (\vartheta)|^2 + |\nabla \vartheta^{\frac{3}{2}}|^2 \Big] &\leq&C, \nonumber \\[2mm]
 \text{ess} \sup\limits_{t\in(0,T)}\int_B \Big[ a_\eta^w \vartheta^4 + \rho^{\frac{5}{3}} \Big]\leq C\text{ ess} \sup\limits_{t\in(0,T)}\int_B \rho e_\eta^w(\rho,\vartheta)  &\leq& C, \label{varthetarho}\\[2mm]
 || \vartheta^\gamma||_{L^2(0,T; H^1(B ))} +  ||\log \vartheta||_{L^2(0,T; H^1(B ))} &\leq& C, \quad \text{for } \gamma \in [1,3/2], \label{varthetah1}\\[2mm]
 \int_0^T\int_B\frac{1}{\vartheta}\Big( \mathbb{S}_\omega^w(\vartheta, \nabla \bb{u}):\nabla \bb{u}+\frac{\kappa_\nu^w(\vartheta)|\nabla \vartheta|^2}{\vartheta}\Big) &\leq& C, \label{fluidentropy} \\[2mm]
||\bb{u}||_{L^2(0,T; L^6(B))} \leq C  ||\bb{u}||_{L^2(0,T; H^1(B))} &\leq& \frac{C}{\omega}, \label{nablau}\\[2mm]
 ||\rho \bb{u}||_{L^\infty(0,T; L^\frac{5}{4}(B))} &\leq& C, \label{rhou54} \\[2mm]
  ||\rho \bb{u}\otimes \bb{u}||_{L^2(0,T; L^\frac{30}{29}(B))} &\leq& \frac{C}{\omega}, \label{rhouu}
\end{eqnarray}
\endgroup and
\begin{eqnarray}\label{divcurlest1}
||\rho s_\eta^w(\rho,\eta)||_{L^p((0,T)\times B)}+  ||\rho s_\eta^w(\rho,\eta)\bb{u}||_{L^q((0,T)\times B)}+ \Big|\Big|\frac{\kappa_\nu^w(\vartheta)\nabla \vartheta}{\vartheta}\Big|\Big|_{L^r((0,T)\times B)}\leq C(\nu,\omega,\lambda), 
\end{eqnarray}
for some $p,q,r>1$.\\
\end{lem}

Further, improved pressure estimates based on the Bogovskii operator
\begin{eqnarray}\label{pressure1}
\int_O p_{\eta,\delta}^w(\rho,\vartheta)\rho^\epsilon \leq C(O)
\end{eqnarray}
holds for any $O \Subset [0,T]\times (\overline{B}\setminus \Gamma^w(t))$ which is Lipschitz in space.  Here, $\epsilon\in(0,1/9)$ and does not depend on any of the approximation parameters.

Finally, based on the uniform boundedness of $w$ in $ W^{1,\infty}(0,T; L^2(\Gamma))\cap L^\infty(0,T;H^2(\Gamma))$ and the embedding (see \cite[Lemma 2.2]{JoSt}):
\begin{equation*}
W^{1,\infty}(0,T; L^2(\Gamma))\cap L^\infty(0,T;H^2(\Gamma))\hookrightarrow  C^{0,\alpha}(0,T; C^{0,1-2\alpha}(\Gamma))\quad \mbox{for}\quad 0<\alpha<1/2,
\end{equation*}
 the lifespan of the solution $T>0$ can be chosen small enough so that
\begin{eqnarray}
    &&|w(t,x)|\geq |w_0(x)| -  |w(t,x)-w_0(x)| \geq \min\limits_{x\in\Gamma} w_0-CT^{\frac{1}{4}}=:m >a_{\partial\Omega}, \nonumber\\
    &&|w(t,x)|\leq |w_0(x)| +  |w(t,x)-w_0(x)| \leq \max\limits_{x\in\Gamma} w_0+CT^{\frac{1}{4}}=:M <b_{\partial\Omega}, \label{lifespan}
\end{eqnarray}
for all $(t,x)\in \Gamma_T$, where $a_{\partial\Omega},b_{\partial\Omega}$ are defined in $\eqref{ab}$.

In the remainder of this paper, the goal is to remove all the approximation layers in the system defined in the previous section, in order to prove the Theorem $\ref{mainth}$. This will be done in the following order: 
\begin{eqnarray*}
    &&\Delta t \to 0, \quad \eta, \omega, \nu,\lambda \to 0,\quad \delta \to 0.
\end{eqnarray*}
First, we will analyse the penalization/splitting limit $\Delta t \to 0$.
Thus we denote the approximate solutions from the previous section as $( \vartheta_{\Delta t}, \rho_{\Delta t}, \bb{u}_{\Delta t},w_{\Delta t}, \theta_{\Delta t})$. The goal is to pass to the limit  the equations $\eqref{RCETM-Ap}, \eqref{coupledmomeq-Ap}$ and $\eqref{coupledenteq-Ap}$.

\subsection{Convergences based on uniform estimates}
First, $\eqref{splittingbound}$ gives us that $\partial_t w_{\Delta t} \rightharpoonup \partial_t w$, $ \bb{v}_{\Delta t} \rightharpoonup \partial_t w\bb{n}$ and $\theta_{\Delta t}, \tau_{\Delta t} \rightharpoonup \theta$ in $L^2(\Gamma_T)$, which implies
\begin{eqnarray*}
&&\int_0^T \int_\Gamma \ddfrac{\bb{v}_{\Delta t}\cdot\bb{n}-T_{\Delta t} \bb{v}_{\Delta t}\cdot\bb{n}}{\Delta t} \psi\\
&&=- \int_{\Delta t}^{T-\Delta t} \int_\Gamma \bb{v}_{\Delta t}\cdot\bb{n}\frac{ T_{-\Delta t}\psi- \psi}{\Delta t} + \ddfrac{1}{\Delta t}\int_{T- \Delta t}^T \int_\Gamma \bb{v}_{\Delta t}\cdot\bb{n} \psi- \ddfrac{1}{\Delta t}\int_{0}^{\Delta t} \int_\Gamma v_{0,\delta}\psi \\
&&\to - \int_0^T \int_\Gamma \partial_t w \partial_t \psi - \int_\Gamma v_0 \psi(0,\cdot),\quad \text{as }\Delta t \to 0,\quad \forall \psi \in C_c^\infty([0,T)\times \Gamma)
\end{eqnarray*}
and similarly,
\begin{eqnarray*}
-\int_0^T  \bint_\Gamma\ddfrac{\tau_{\Delta t}-T_{\Delta t}\tau_{\Delta t}}{\Delta t} \tilde{\psi} \to -\int_0^T  \bint_\Gamma  \theta \partial_t\tilde{\psi}- \int_\Gamma \theta_0 \tilde{\psi}(0,\cdot) , \quad \text{as }\Delta t \to 0,\quad \forall \tilde{\psi} \in C_c^\infty([0,T)\times \Gamma).
\end{eqnarray*}
Therefore we have:
\begin{align*}
-\delta \bint_{\Gamma_T}\ddfrac{\bb{v}\cdot\bb{n}-T_{\Delta t}\bb{v}\cdot\bb{n}}{\Delta t} \psi+
    (1-\delta)\bint_{\Gamma_T}\partial_t w \partial_t \psi -(1-\delta) \int_0^T \frac{d}{dt} \int_\Gamma \partial_t w \psi    
    \to \int_0^T \int_\Gamma \partial_t w \partial_t \psi +\int_\Gamma  w_0 \psi(0,\cdot),
\end{align*}
\begin{align*}
-  \delta\int_0^T\int_{\Gamma} \ddfrac{\tau-T_{\Delta t}\tau}{\Delta t} \tilde{\psi}- (1-\delta)\int_0^T\int_{\Gamma} \theta \partial_t \tilde{\psi}+(1-\delta)\int_{\Gamma} \theta \tilde{\psi}\Big|_0^T 
    \to\int_0^T  \bint_\Gamma  \theta \partial_t\tilde{\psi}+ \int_\Gamma \theta_0 \tilde{\psi}(0,\cdot)  .
\end{align*}

Now, from $\eqref{linftyplate}$, $w_{\Delta t}\in W^{1,\infty}(0,T; L^2(\Gamma))\cap L^\infty(0,T; H^2(\Gamma))\hookrightarrow C^{0,\alpha}(0,T; C^{0,1-2\alpha}(\Gamma))$, $0<\alpha<\frac{1}{2}$, so
\begin{eqnarray}\label{wcconv}
w_{\Delta t} \to w, \quad \text{in } C^{\frac{1}{4}}([0,T]\times \Gamma),
\end{eqnarray}
which implies
\begin{eqnarray}\label{domainconv}
\mathcal{M}\big((\Omega^{w_{\Delta t}} \setminus\Omega^w)\cup (\Omega^w\setminus\Omega^{w_{\Delta t}}) \big) \to 0, \quad \text{in } C^{\frac{1}{4}}([0,T]),
\end{eqnarray}
and 
\begin{eqnarray}\label{chiconv}
\begin{aligned}
&f_\omega^{w_{\Delta t}}\to f_\omega^{w}, \quad \text{in } C^{\frac{1}{4}}([0,T]\times B),\\
&\chi_\nu^{w_{\Delta t}}\to \chi_\nu^w,\quad\chi_\eta^{w_{\Delta t}}\to \chi_\eta^w, \quad \text{in } L^{\infty}([0,T]\times B).
\end{aligned} 
\end{eqnarray}

The uniform bounds $\eqref{varthetarho}$ and $\eqref{varthetah1}$ directly give us
\begin{eqnarray*}
&&\rho_{\Delta t} \rightharpoonup \rho, \quad \text{weakly}^* \text{ in } L^\infty(0,T; L^{\frac{5}{3}}(B)),\\
&&\vartheta_{\Delta t} \rightharpoonup \vartheta, \quad \text{weakly}^* \text{ in } L^\infty(0,T; L^{4}(B)),\\
&&\vartheta_{\Delta t} \rightharpoonup \vartheta, \quad \text{weakly in } L^2(0,T; H^{1}(B)).
\end{eqnarray*}

Now, the continuity equation $\eqref{RCE1}$ together with $\eqref{linftybound}$ yield
\begin{eqnarray*}
\rho_{\Delta t} \to \rho,\quad  \text{ in } C_w([0,T]; L^{\frac{5}{3}}(B)),
\end{eqnarray*}
while $\eqref{nablau}$ implies
\begin{eqnarray}
\bb{u}_{\Delta t} \rightharpoonup \bb{u}, \text{weakly in } L^2(0,T; H_0^{1}(B)), \label{weaku}
\end{eqnarray}
which together give rise to
\begin{eqnarray*}
\rho_{\Delta t}\bb{u}_{\Delta t} \rightharpoonup \rho \bb{u},  \text{weakly}^* \text{ in } L^\infty(0,T; L^{\frac{5}{4}}(B))
\end{eqnarray*}
by the compact imbedding of $L^\frac{5}{3}(B)$ into $H^{-1}(B)$ and the uniform bounds $\eqref{linftybound},\eqref{nablau}$. Thus, we infer that
\begin{eqnarray*}
    \rho_{\Delta t}\bb{u}_{\Delta t} \rightharpoonup \rho \bb{u},  \text{weakly}^* \text{ in } C_w([0,T]; L^{\frac{5}{4}}(B)).
\end{eqnarray*}
Finally, by the momentum equation $\eqref{coupledmomeq1}$, $\eqref{rhou54}$ and the compact imbedding of $L^{\frac{5}{4}}(B)$ into $H^{-1}(B)$, one obtains
\begin{eqnarray*}
    \rho_{\Delta t} \bb{u}_{\Delta t} \to \rho \bb{u}, \quad \text{in } L^2(0,T; H^{-1}(B)),
\end{eqnarray*}
so by $\eqref{weaku}$ and $\eqref{rhouu}$ one has
\begin{eqnarray*}
\rho_{\Delta t} \bb{u}_{\Delta t}\otimes\bb{u}_{\Delta t} \rightharpoonup \rho \bb{u}\otimes\bb{u}, \quad \text{in } L^2(0,T; L^{\frac{30}{29}}(B)).
\end{eqnarray*}
\subsection{Weak convergence of the pressure}\label{pressuresec}
Using the estimates $\eqref{pressure1}$, $\eqref{domainconv}$ and $\eqref{chiconv}$, one has
\begin{eqnarray*}
p_{\eta,\delta}^{w_{\Delta t}}(\rho_{\Delta t},\vartheta_{\Delta t}) = p_M(\rho_{\Delta t},\vartheta_{\Delta t}) + \frac{a_\eta }{3}\vartheta_{\Delta t}^4 + \delta \rho_{\Delta t}^\beta \rightharpoonup \overline{p_M(\rho,\vartheta)} + \frac{a_\eta }{3}\overline{\vartheta^4} + \delta \overline{\rho^\beta}, \quad \text{weakly in } L^1(O)
\end{eqnarray*}
for any $O \Subset [0,T]\times \overline{B}$ such that $O\cap \big([0,T]\times\Gamma^w(t)) = \emptyset$, where bar notation from now on denotes the weak limit. The idea is to take a sequence of compact sets $\{O_i\}_{i\in \mathbb{N}}$ such that
\begin{eqnarray*}
 O_i\cap \big([0,T]\times \Gamma^w(t)\big)= \emptyset, \quad O_i \subset O_{i+1}, \quad O_i \to [0,T] \times B \text{ as }i\to \infty,
\end{eqnarray*}
in order to obtain a weak limit of $p_{\eta,\delta}^{w_{\Delta t}}(\rho_{\Delta t},\vartheta_{\Delta t})$ on the entire set $[0,T] \times B$. The key issue is that this convergence doesn't exclude a possible concentration of $L^1$ norm of $p_{\eta,\delta}^{w_{\Delta t}}(\rho_{\Delta t},\vartheta_{\Delta t})$ as $\Delta t \to 0$ that may result in a measure appearing at the moving boundary $\Gamma^w$, which is then felt by test functions that do not vanish at the moving boundary, unlike the case in \cite{heatfl} (comprehensible description of the idea may be found in  \cite[Section 2.2.6]{feireislnovotny}). To deal with this issue, we can use the approach developed in \cite{kukucka} which was later adapted to the fluid-structure interaction framework in \cite{compressible}:

\begin{lem}\label{pressconc}
For any given $\kappa>0$, there exists a $(\Delta t)_0>0$ and $A_{\kappa}\Subset (0,T)\times B$ such that for all $\Delta t<(\Delta t)_0$
\begin{eqnarray*}
A_\kappa \cap \big([0,T]\times\Gamma^{w_{\Delta t}}\big)= \emptyset, \quad 
\int_{((0,T)\times B)\setminus A_\kappa} p_{\eta,\delta}^{w_{\Delta t}}(\rho_{\Delta t},\vartheta_{\Delta t}) \leq \kappa.
\end{eqnarray*}
\end{lem}
\begin{proof}
The proof can be carried out in the same way as in \cite[Lemma 7.4]{compressible}, since all the fluid terms in the coupled momentum equation (except the pressure) are the same and the corresponding integrants are bounded in at least $L^1(0,T; L^p(B))$ for some $p>1$. For reader's convenience, we present the proof here.

The key idea is to construct a test function which has an arbitrarily large and positive divergence near the boundary $\Gamma^{w_{\Delta t}}$. First, let $f\in C_0^\infty(\mathbb{R})$ be a function satisfying
\begin{eqnarray*}
    0 \leq f\leq 1, \quad \text{supp} f\Subset (a_{\partial\Omega},b_{\partial\Omega}), \quad f=1 \text{ in a neighbourhood of } [m,M],
\end{eqnarray*}
where $m,M$ are defined in $\eqref{lifespan}$, and fix $K>0$ small enough such that that 
\begin{equation}\label{choiceK}
f\equiv 1\mbox{ on  }\ [m-\frac{1}{K},M+\frac{1}{K}]. 
\end{equation}
We define
\begin{eqnarray*}
    \boldsymbol\varphi_{\Delta t}^K(t,X):= f(d(X)) \times \begin{cases}\min \big\{ K (d(X)- w_{\Delta t}(t,\pi(X))),1\big\} \bb{n}(\pi(X)),\quad &\text{ on } [0,T]\times N_a^b, \\
    0, \quad &\text{ on } [0,T]\times (B\setminus N_a^b),
    \end{cases}
\end{eqnarray*}
where $d,\pi$ and $N_a^b$ are defined in $\eqref{d},\eqref{p}, \eqref{nab}$, and
\begin{eqnarray*}
    S_K(t):=\{ X:  |d(X)-w(t,\pi(X))  | <1/K \}.
\end{eqnarray*}
Note that using the property \eqref{choiceK} and the definition of the space $S_K(t)$, we have  $f(d(X))=1$ on $[0,T]\times S_K(t)$. Moreover, we have $\boldsymbol\varphi_{\Delta t}^K=0$ on $[0,T]\times \Gamma^w(t)$.

Now, denote by $\boldsymbol\tau_1(X)$ and $\boldsymbol\tau_2(X)$ the orthogonal tangential vectors at point $ X\in \Gamma$. We have on $[0,T]\times S_K(t)$:
\begin{eqnarray*}
    &&\nabla \cdot \boldsymbol\varphi_{\Delta t}^K (t,X)\\
    &&= \partial_{\bb{n}(\pi(X))} \boldsymbol\varphi_{\Delta t}^K(t,X) \cdot \bb{n}(\pi(X)) + \partial_{\boldsymbol\tau_1(\pi(X))} \boldsymbol\varphi_{\Delta t}^K(t,X) \cdot\boldsymbol\tau_1(\pi(X))+\partial_{\boldsymbol\tau_2(\pi(X))} \boldsymbol\varphi_{\Delta t}^K(t,X) \cdot\boldsymbol\tau_2(\pi(X)) \\
    &&=:I_1+I_2+I_3.
\end{eqnarray*}
The first term can be estimated in the following way
\begin{eqnarray*}
    I_1&&= \underbrace{\Big[\partial_{\bb{n}(\pi(X))}f(d(X)) \Big]}_{=0,~\text{ on }[0,T]\times S_K(t)}\min \big\{ K (d(X)- w_{\Delta t}(t,\pi(X))),1\big\} \\
    &&\quad +\underbrace{f(d(X))}_{=1,~\text{ on }[0,T]\times S_K(t)} \underbrace{\Big[\partial_{\bb{n}(\pi(X))}\min \big\{ K (d(X)- w_{\Delta t}(t,\pi(X))),1\big\}}_{=K,~\text{ on }[0,T]\times S_K(t)} \Big]\\
    &&=K,
\end{eqnarray*}
where we have used the relations
\begin{equation*}
\nabla d(X)=\bb{n}(\pi(X)),\quad \partial_{\bb{n}(\pi(X))} d(X)=1,\quad \partial_{\bb{n}(\pi(X))} \pi(X)=0.
\end{equation*}
We estimate the second and third terms by
\begin{eqnarray*}
    |I_{i+1}| &&  \leq   \big|\partial_{\boldsymbol\tau_i(\pi(X))}\big[ f(d(X)) \min \big\{ K (d(X)- w_{\Delta t}(t,\pi(X))),1\big\}\big] \underbrace{\bb{n}(\pi(X)) \cdot \boldsymbol\tau_i(\pi(X))}_{=0}\big|\\
    &&\quad +\big|f(d(X)) \min \big\{ K (d(X)- w_{\Delta t}(t,\pi(X))),1\big\}\partial_{\boldsymbol\tau_i(\pi(X))} \bb{n}(\pi(X))\cdot \boldsymbol\tau_i(\pi(X))\big|\\
    &&\leq c
\end{eqnarray*}
for $i=1,2$. Thus,
\begin{eqnarray*}
    \nabla \cdot \boldsymbol\varphi_{\Delta t}^K \geq K-c, \quad \text{ on } [0,T]\times S_K(t).
\end{eqnarray*}
Next, one has the following for all $p<\infty$
\begin{eqnarray}
    &&||\nabla \boldsymbol\varphi_{\Delta t}^K||_{L^\infty(0,T; L^{p}( S_K(t)))} \leq C\big(|| \nabla w_{\Delta t}||_{L^\infty(0,T; L^{p}(\Gamma)} + 1\big) \leq  C(K+1),\label{varphidtk1}\\
    &&||\nabla \boldsymbol\varphi_{\Delta t}^K||_{L^\infty(0,T; L^{p}( B\setminus S_K(t)))} \leq C, \label{varphidtk2}\\
    &&|| \boldsymbol\varphi_{\Delta t}^K ||_{L^\infty((0,T)\times B)}\leq C, \nonumber\\
     &&\partial_t \boldsymbol\varphi_{\Delta t}^K=0,\quad \text{on } [0,T]\times (B\setminus S_K(t)) ,\nonumber\\
     && \nabla \cdot \boldsymbol\varphi_{\Delta t}^K \leq C, \quad \text{ on } [0,T]\times (B\setminus S_K(t)), \nonumber\\
     &&||\partial_t \boldsymbol\varphi_{\Delta t}^K||_{L^2(0,T; L^{p}(S_K(t)))} \leq C(K+1)|| \partial_t w_{\Delta t}||_{L^\infty(0,T; L^{p}(\Gamma))} \leq  C(K+1), \label{timederest}
\end{eqnarray}
where $C$ only depend on $f, p, \Gamma$ and the initial energy. Now, choosing $(\boldsymbol\varphi,\psi)=(\boldsymbol\varphi_{\Delta t}^K,0)$ in $\eqref{coupledmomeq1}$ gives us
\begin{align}
&\int_0^T  \int_{S_K(t)} p_{\eta,\delta}^{w_{\Delta t}}(\rho_{\Delta t},\vartheta_{\Delta t})(K-c)  \leq \int_0^T  \int_{S_K(t)} p_{\eta,\delta}^{w_{\Delta t}}(\rho_{\Delta t},\vartheta_{\Delta t})(\nabla \cdot \boldsymbol\varphi_{\Delta t}^K)  \nonumber \\
&=-\int_0^T  \int_{B\setminus S_K(t)} p_{\eta,\delta}^{w_{\Delta t}}(\rho_{\Delta t},\vartheta_{\Delta t})(\nabla \cdot \boldsymbol\varphi_{\Delta t}^K) \nonumber\\
&\quad -\int_0^T \bint_B \rho_{\Delta t} \bb{u}_{\Delta t} \cdot\partial_t \boldsymbol\varphi_{\Delta t}^K - \int_0^T \bint_B (\rho_{\Delta t} \bb{u}_{\Delta t} \otimes \bb{u}_{\Delta t}):\nabla\boldsymbol\varphi_{\Delta t}^K\nonumber\\
&\quad+ \int_0^T \bint_B  \nabla \mathbb{S}_\omega^w(\vartheta_{\Delta t}, \nabla\bb{u}_{\Delta t}): \nabla \boldsymbol\varphi_{\Delta t}^K + \int_0^T\frac{d}{dt} \bint_B \rho_{\Delta t} \bb{u}_{\Delta t} \cdot\partial_t \boldsymbol\varphi_{\Delta t}^K \nonumber
\end{align}
The critical terms are the second and the third term on the right-hand side. We can bound the second term as:
\begin{eqnarray}
    &&\int_0^T \int_B \rho_{\Delta t} \bb{u}_{\Delta t} \cdot\partial_t \boldsymbol\varphi_{\Delta t}^K= \int_0^T \int_{S_K(t)} \rho_{\Delta t} \bb{u}_{\Delta t} \cdot\partial_t \boldsymbol\varphi_{\Delta t}^K\nonumber\\[2mm]
    &&\leq C||\rho_{\Delta t} \bb{u}_{\Delta t}||_{L^\infty(0,T; L^{\frac{5}{4}}(B))} || 1 ||_{L^1(0,T; L^6(S_K(t)))} (K+1) \leq C K^{-\frac{1}{6}}(K+1),\label{importantestimate}
\end{eqnarray}
where we have used the estimates \eqref{rhou54}, $\eqref{timederest}$ and the fact that $\mathcal{M}(S_K(t))\leq CK^{-1}$. Next, from $\eqref{varphidtk1},\eqref{varphidtk2},\eqref{rhouu}$, one has
\begin{eqnarray*}
    &&\int_0^T \bint_B (\rho_{\Delta t} \bb{u}_{\Delta t} \otimes \bb{u}_{\Delta t}):\nabla\boldsymbol\varphi_{\Delta t}^K\\
    &&= \int_0^T \bint_{S_K(t)} (\rho_{\Delta t} \bb{u}_{\Delta t} \otimes \bb{u}_{\Delta t}):\nabla\boldsymbol\varphi_{\Delta t}^K + \int_0^T \bint_{B\setminus S_K(t)} (\rho_{\Delta t} \bb{u}_{\Delta t} \otimes \bb{u}_{\Delta t}):\nabla\boldsymbol\varphi_{\Delta t}^K\\[2mm]
    && \leq C|| \rho_{\Delta t} \bb{u}_{\Delta t} \otimes \bb{u}_{\Delta t}||_{L^2(0,T; L^{\frac{30}{29}}(B))}\Big[ (K+1)||1||_{L^2(0,T; L^{30}(S_K(t)))}+1  \Big] \leq C(K^{-\frac{1}{30}}(K+1)+1).
\end{eqnarray*}
Noticing that the remaining terms can be bounded in a similar fashion, we conclude
\begin{eqnarray*}
    \int_0^T  \int_{S_K(t)} p_{\eta,\delta}^{w_{\Delta t}}(\rho_{\Delta t},\vartheta_{\Delta t}) \leq C\frac{    K^{-j}(K+1)+1}{K-c},
\end{eqnarray*}
for a $j>0$, so we can always choose $K>0$ large enough such that
\begin{eqnarray*}
    C\frac{K^{-j}(K+1)+1}{K-c}\leq \kappa,
\end{eqnarray*}
for any $\kappa>0$. Now, the proof follows for $A_\kappa=(0,T)\times (B\setminus S_K(t))$.
\end{proof}

\begin{cor}\label{pressconvALT}
The above lemma also holds with assumptions $\alpha_1=\alpha_2=0$ and $p(\rho,\vartheta)=\rho^\gamma+\rho\vartheta+ \frac{a}{3}\vartheta^4$, for some $\gamma>12/7$.
\end{cor}
\begin{proof}
In this case, the only difference is in closing the estimate in $\eqref{importantestimate}$. With the information we have here, we can conclude (see \cite[Corollary  2.9]{LenRuz})
\begin{eqnarray*}
    \int_0^T||\partial_t w_{\Delta t}||_{H^s(\Gamma)}^2\leq C(s,||w_{\Delta t}||_{L^\infty(0,T;H^2(\Gamma))}) \int_0^T ||\bb{u}_{\Delta t}||_{H^1(\Omega^{w_{\Delta t}}(t))}^2\leq C,
\end{eqnarray*}
for any $s<\frac{1}{2}$, so 
\begin{eqnarray*}
    ||\partial_t w_{\Delta t}||_{L^2(0,T; L^r(\Gamma))} \leq C(r),
\end{eqnarray*}
for any $r<4$, by the imbedding of Sobolev spaces. This finally gives us
\begin{eqnarray*}
   ||\partial_t \boldsymbol\varphi_{\Delta t}^K||_{L^2(0,T; L^{r}(S_K(t)))} \leq C(K+1)|| \partial_t w_{\Delta t}||_{L^\infty(0,T; L^{r}(\Gamma))} \leq  C(K+1) 
\end{eqnarray*}
so we can bound
\begin{eqnarray*}
    &&\int_0^T \int_B \rho_{\Delta t} \bb{u}_{\Delta t} \cdot\partial_t \boldsymbol\varphi_{\Delta t}^K= \int_0^T \int_{S_K(t)} \rho_{\Delta t} \bb{u}_{\Delta t} \cdot\partial_t \boldsymbol\varphi_{\Delta t}^K\\[2mm]
    &&\leq C||\rho_{\Delta t} \bb{u}_{\Delta t}||_{L^\infty(0,T; L^{\frac{2\gamma}{\gamma+1}}(B))} || 1 ||_{L^2(0,T; L^q(S_K(t)))}||\partial_t \boldsymbol\varphi_{\Delta t}^K||_{L^2(0,T; L^{r}(S_K(t)))}\\
    && \leq C K^{-\frac{1}{q}}(K+1).
\end{eqnarray*}
for some $q<\infty$ and $r<4$ such that 
\begin{eqnarray*}
    \frac{\gamma+1}{2\gamma}+\frac{1}{q}+\frac{1}{r}=1.
\end{eqnarray*}
Note that such $q$ always exists due to condition $\gamma>12/7$.
\end{proof}

\begin{rem}\label{pressureremark}
(1) The above lemma and remark reflect the only difference in the analysis for these two fluid models. The only reason we need $\alpha_1+\alpha_2>0$ is to close the estimate $\eqref{importantestimate}$, since in our case we only have $\rho \in L^\infty(0,T;L^{\frac{5}{3}}(\mathbb{R}^3))$, so $ \partial_t w_{\Delta t} \in L^2(0,T; L^r(\Gamma))$, $r<4$, is not enough (in 2D case, $\alpha_1+\alpha_2=0$ is enough, due to stronger imbedding results).\\
(2) A careful reader would have noticed that we also need $\alpha_1+\alpha_2>0$ in the proof of Lemma $\ref{pressvanish}$, since  the regularity $\partial_t \nabla w \in L^2(\Gamma_T)$ is also required there. However, once we prove that $\rho_{|B\setminus\Omega^w(t)}(t,\cdot)=0$ for a.a. $t\in (0,T)$, it will hold throughout the later convergences. This means that in the case $\alpha_1=\alpha_2=0$ and $p(\rho,\vartheta)=\rho^\gamma+\rho\vartheta+ \frac{a}{3}\vartheta^4$, for some $\gamma>12/7$, we can deal with this issue by adding, say, a term $-\delta\partial_t \nabla w\cdot \nabla \psi$ to the coupled momentum equation $\eqref{coupledmomeq1}$.
\end{rem}

\subsection{Pointwise convergence of the fluid density and the fluid temperature}\label{sec33}
In order to identify the limits of the remaining terms in $\eqref{RCE1}, \eqref{coupledmomeq1}$ and $\eqref{coupledenteq1}$,
it is enough the prove a.e. convergence of $\vartheta_{\Delta t}$ and $\rho_{\Delta t}$. \\

First, let
\begin{eqnarray*}
&&\bb{U}_{\Delta t} := \Bigg[ \rho_{\Delta t} s_\eta^{w_{\Delta t}}(\rho_{\Delta t},\vartheta_{\Delta t}), ~\rho_{\Delta t} s_\eta^{w_{\Delta t}}(\rho_{\Delta t},\vartheta_{\Delta t}) \bb{u}_{\Delta t} - \frac{\kappa_\nu^{w_{\Delta t}}(\vartheta_{\Delta t})\nabla \vartheta_{\Delta t}}{\vartheta_{\Delta t}}\Bigg],\\
&&\bb{W}_{\Delta t}:= [G(\vartheta_{\Delta t}),0,0,0],
\end{eqnarray*}
where $G$ is a bounded and Lipschitz function on $[0,\infty)$. In order to apply div-curl lemma, one should note that we cannot deal with the boundary terms in the coupled entropy balance $\eqref{coupledenteq1}$, so we need to use this lemma locally. Fix $O\Subset(0,T)\times B$ such that $O\cap \big((0,T)\times\Gamma^{w_{\Delta t}}(t)\big)=\emptyset$, for all $\Delta t\leq (\Delta t)_O$. From $\eqref{divcurlest1}$, $\bb{U}_{\Delta t}\in L^p(O)$ for some $p>1$, while one easily has that $\bb{W}_{\Delta t}$ is bounded in $L^\infty(O)$ and that $\text{Curl}_{t,x} \bb{W}_{\Delta t}$ is precompact in $W^{-1,s}(O)$, for some $s>1$. Now, in order to prove the precompactness of $\text{Div}_{t,x}  \bb{U}_{\Delta t}$, let $\varphi \in C_c^\infty(O)$. One has from $\eqref{coupledenteq1}$
\begin{eqnarray*}
    &&\int_O \text{Div}_{t,x}  \bb{U}_{\Delta t} \varphi =  -\langle \sigma_{\omega,\nu};\varphi \rangle_{[\mathcal{M},C]([0,T]\times O)}+\int_O \lambda  \vartheta^4 \varphi \leq C|| \varphi||_{C(O)} \leq C ||\varphi||_{W^{m,p}(O)}, \quad
\end{eqnarray*}
 for some $m<1$ and $p>4$.
Thus, $\text{Div}_{t,x}  \bb{U}_{\Delta t}$ is bounded in $W^{-m,p^*}(O)$ and therefore precompact in $W^{-1,s}(O)$, for $s\in[1,4/3)$. Since any $O'\Subset(0,T)\times B$ such that $O'\cap \big((0,T)\times\Gamma^{w}(t)\big)=\emptyset$ is a subset of some $O$ defined above (existence of the corresponding $(\Delta t)_O$ follows by $\eqref{wcconv}$), we can conclude by the div-curl lemma
\begin{eqnarray}\label{divcurl}
\overline{\rho s_\eta^w(\rho, \vartheta) G(\vartheta)} = \overline{\rho s_\eta^w(\rho, \vartheta)}~\overline{ G(\vartheta)}, \quad \text{a.e. in } (0,T)\times B.
\end{eqnarray}
The next step is to show
\begin{eqnarray}\label{young}
\overline{ \rho s_M(\rho, \vartheta) G(\vartheta)} \geq \overline{ \rho s_M(\rho, \vartheta)}~\overline{ G(\vartheta)} , \quad\overline{ \vartheta^3 G(\vartheta)} \geq  \overline{ \vartheta^3}~ \overline{ G(\vartheta)}
\end{eqnarray}
for any continuous and increasing function $G$. This part can be done by application of theory of parametrized Young measures (see \cite[Section 3.6.2]{feireislnovotny}). Combining $\eqref{divcurl}$ and $\eqref{young}$
\begin{eqnarray*}
\overline{\vartheta^4} = \overline{\vartheta^3} \vartheta
\end{eqnarray*}
which then by monotonicity of $x\mapsto x^4$ gives us
\begin{eqnarray}\label{aeconv1}
\vartheta_{\Delta t} \to \vartheta, \quad \text{a.e. in } (0,T)\times B.
\end{eqnarray}

In order to prove the pointwise convergence of the density, we use the convergence of the effective viscous flux, in particular:
\begin{lem}\label{evf.for.delta.t}
The equality
\begin{eqnarray*}
&&\lim_{\Delta t \to 0}\int_{Q^w_T} \varphi \left(p_M(\rho_{\Delta t},\vartheta_{\Delta t}) - \left(\frac 43 \mu(\vartheta_{\Delta t}) + \eta(\vartheta_{\Delta t})\right)\diver \bb{u}_{\Delta t}\right) \rho_{\Delta t} {\rm d}x{\rm d}t \\
&&= \int_{Q_T^w} \varphi\left(\overline{p_M(\rho,\vartheta)} - \overline{\left(\frac 43 \mu(\vartheta) + \eta(\vartheta)\right) \diver \vu}\right) \rho {\rm d}x{\rm d}t
\end{eqnarray*}
holds for all $\varphi\in C^\infty_c(Q_T^w)$.
\end{lem}
\begin{proof} The proof of this lemma does not differ from the one presented in \cite[Section 3.6.5]{feireislnovotny} and therefore we do not provide it here. Note that for every $\varphi\in C^\infty_c(Q_T^w)$, there is a small enough $(\Delta_t)_0$ such that for all $\Delta t\leq (\Delta t)_0$ (at least for a subsequence) one has $\varphi\in C^\infty_c(Q_T^{w_{\Delta t}})$, which is a direct consequence of $\eqref{domainconv}$.
\end{proof}
Now it is sufficient to take $b(\rho) = \rho\log\rho$ in the renormalized continuity equation
$$
\partial_t b(\rho) + \diver (b(\rho)\vu) + (b'(\rho)\rho - b(\rho))\diver \vu = 0
$$
(note that this equation holds true for both $\rho,\vu$ and $\rho_{\Delta t},\vu_{\Delta t}$)
and we infer that
\begin{equation}\label{rho.logrho.delta}
\int_0^\tau\int_B \rho_{\Delta t} \diver \vu_{\Delta t} \ {\rm d}x{\rm d}t = \int_B \rho_0\log \rho_{0,\delta} {\rm d}x - \int_B \rho_{\Delta t}(\tau,\cdot)\log \rho_{\Delta t}(\tau,\cdot)\ {\rm d}x
\end{equation}
and
\begin{equation}\label{rho.logrho}
\int_0^\tau\int_B \rho \diver \vu \ {\rm d}x{\rm d}t = \int_B \rho_{0,\delta}\log \rho_{0,\delta}\ {\rm d}x - \int_B \rho(\tau,\cdot)\log \rho(\tau,\cdot)\ {\rm d}x
\end{equation}
for every $\tau \in [0,T]$. Lemma $\ref{evf.for.delta.t}$ and the monotonicity of $p_M$ yield $\overline{\rho\diver\vu}\geq \rho\diver\vu$ and therefore, with help of $\eqref{rho.logrho.delta}$ and $\eqref{rho.logrho}$
$$
\overline{\rho\log(\rho)} = \rho\log(\rho)
$$
Due to the convexity of $\rho\mapsto \rho\log\rho$ we have just deduced that
\begin{eqnarray}\label{aeconv2}
\rho_{\Delta t}\to \rho \quad \text{a.e. in } (0,T)\times B,
\end{eqnarray}
so $\eqref{aeconv1},\eqref{aeconv2}$ and $\eqref{wcconv}$ combined with the uniform estimates give rise to the following weak and strong convergences in $L^1((0,T)\times B)$
\begin{eqnarray*}
&&\vartheta_{\Delta t}^4 \to \vartheta^4,\quad \frac{\kappa_\eta^{w_{\Delta t}}(\vartheta_{\Delta t})\nabla \vartheta_{\Delta t}}{\vartheta_{\Delta t}} \rightharpoonup \frac{\kappa_\eta^{w}(\vartheta)\nabla \vartheta}{\vartheta}   \\[2mm]
&&p_{\eta,\delta}^{w_{\Delta t}}(\rho_{\Delta t},\vartheta_{\Delta t})\rightharpoonup p_{\eta,\delta}^{w}(\rho,\vartheta), \quad \mathbb{S}_\omega^{w_{\Delta t}}(\vartheta_{\Delta t},\nabla \bb{u}_{\Delta t}) \rightharpoonup \mathbb{S}_\omega^{w}(\vartheta,\nabla \bb{u})\\[2mm]
&&\rho_{\Delta t} s_\eta^{w_{\Delta t}}(\rho_{\Delta t}, \vartheta_{\Delta t})\to \rho s_\eta^{w}(\rho, \vartheta), \quad
\rho_{\Delta t} s_\eta^{w_{\Delta t}}(\rho_{\Delta t}, \vartheta_{\Delta t}) \bb{u}_{\Delta t}\rightharpoonup \rho s_\eta^w(\rho, \vartheta) \bb{u}.
\end{eqnarray*}

\subsection{Construction and convergence of approximate test functions}\label{secapptest}
Here, we construct an appropriate sequence of regular test functions that converge alongside the approximate solutions. Namely, test functions in Definition \ref{weaksolutionap1} (equation \eqref{coupledmomeq1})  depend on solution since they must satisfy condition $\psi\bb{n}=\gamma_{|\Gamma^w}\varphi$. Test functions for the approximate equations satisfy analogous condition on $\Gamma^{w_{\Delta t}}$, and therefore $(\varphi,\psi)$ is not admissible test function for the approximate equations. The solution is to construct a sequence of admissible test functions for the approximate equations, $(\varphi_{\Delta t},\psi_{\Delta t})$, that converges to the test for the extended problem, $(\varphi,\psi)$. Note that just simply taking $\psi_{\Delta t}(t,y) :=\varphi(t,w_{\Delta t}(t,y))\cdot\bb{n}$ will not working for the following reasons. First,  restriction of $\varphi$ on $\Gamma^{w_{\Delta t}}$ is not necessary just in $\bb{n}$ direction. Moreover, regularity properties of the restriction are inherited from $w_{\Delta t}$ which is not enough. Therefore, we must do a more subtle construction which uses the properties of the tubular neighbourhood.

We start with introducing the sequences of functions $\{a_{\Delta t}^1\}_{\Delta t}$, $\{a_{\Delta t}^2\}_{\Delta t}$, $\{b_{\Delta t}^1\}_{\Delta t}$ and $\{b_{\Delta t}^2\}_{\Delta t}$ in $C^\infty([0,T]\times\Gamma)$ such that
\begin{eqnarray*}
    m< a_{\Delta t}^1(t,y) <a_{\Delta t}^2(t,y)< w_{\Delta t}(t,y)<b_{\Delta t}^1(t,y)<b_{\Delta t}^2(t,y)<M, \quad \text{for all } (t,X)\in \Gamma_T,
\end{eqnarray*}
where $m,M$ are given in $\eqref{lifespan}$, and\footnote{Note that this is possible due to the strong convergence of $w_{\Delta t}$ in $C^{\frac{1}{4}}(\Gamma_T)$, at least for a suitable subsequence.}
\begin{eqnarray*}
    a_{\Delta t}^2(t,y)<w(t,y)<b_{\Delta t}^1(t,y), \quad  \text{for all } (t,y)\in \Gamma_T.
\end{eqnarray*}
Moreover,
\begin{eqnarray}
    a_{\Delta t}^1,~a_{\Delta t}^2,~ b_{\Delta t}^1, ~b_{\Delta t}^2 \to w \quad \text{ as }\Delta t\to 0, \quad \text{on } [0,T]\times \Gamma\label{convergences}
\end{eqnarray}
and
\begin{eqnarray}
      (a_{\Delta t}^2-a_{\Delta t}^1) \sim (w_{\Delta t}-a_{\Delta t}^1)  \sim (b_{\Delta t}^2-b_{\Delta t}^1) \sim (b_{\Delta t}^1-w_{\Delta t}),\quad  \label{distances}
\end{eqnarray}
where the notation 
\begin{eqnarray*}
    f\sim g \iff c_1g(t,X) \leq f(t,X)\leq c_2 g(t,X) \text{ for all } (t,X)\in [0,T]\times \Gamma, 
\end{eqnarray*}
for some given uniform constants $0<c_1<c_2$. Next, let $\{\phi_{\Delta t}\}_{\Delta t}$ be a sequence of cut-off function in $C_c^\infty([0,T]\times B)$ such that $0\leq \phi_{\Delta t} \leq 1$,
\begin{eqnarray*}
    &&\phi_{\Delta t}(t,X)=1, \quad \text{for all } X\in \Omega^{b_{\Delta t}^1}(t)\setminus \Omega^{a_{\Delta t}^2}(t) \text{ and } t\in [0,T]\\
    &&\phi_{\Delta t}(t,X)=0, \quad  \text{for all } X\in \big[B\setminus \Omega^{b_{\Delta t}^2}(t)\big]\cup \Omega^{a_{\Delta t}^2}(t) \text{ and } t\in [0,T], 
\end{eqnarray*}
where $\Omega^f(t)$ is the domain determined by function $f(t,\cdot)$ defined on $\Gamma$, and
\begin{eqnarray*}
    \max|\nabla^i  \phi_{\Delta t}| \leq \frac{C}{\min \big[ \min\limits_{(t,X)\in[0,T]\times \Gamma}( a_{\Delta t}^2-a_{\Delta t}^1)^i, \min\limits_{(t,X)\in[0,T]\times \Gamma}(b_{\Delta t}^2-b_{\Delta t}^1)^i  \big]},
\end{eqnarray*}
for $i\leq 3$. Note that this condition combined with $\eqref{distances}$ gives us
\begin{eqnarray}
    \max|\nabla^i  \phi_{\Delta t}| &\leq& \frac{C}{\min\big[ \min\limits_{(t,X)\in[0,T]\times \Gamma}(w_{\Delta t}- a_{\Delta t}^2)^i, \min\limits_{(t,X)\in[0,T]\times \Gamma}(b_{\Delta t}^1-w_{\Delta t})^i  \big]} \nonumber \\
    &\leq& \frac{C}{\max \big[ \max\limits_{(t,X)\in[0,T]\times \Gamma}(w_{\Delta t}- a_{\Delta t}^2)^i, \max\limits_{(t,X)\in[0,T]\times \Gamma}(b_{\Delta t}^1-w_{\Delta t})^i  \big]},    \label{derivativesphi}
\end{eqnarray}
for $i\leq 3$.\\ \\

\noindent The \textbf{approximate test functions} are defined as (see figure $\ref{test}$):
\begin{enumerate}
    \item \textbf{Renormalized continuity equation} - test function stays the same;
    \item \textbf{Coupled momentum equation} - $(\boldsymbol\varphi,\psi)$ are approximated with  $(\boldsymbol\varphi_{\Delta t},\psi)$, where
    \begin{eqnarray}
    \boldsymbol\varphi_{\Delta t}(t,X)=\boldsymbol\varphi(t,X)+\phi_{\Delta t}(t,X)(\psi(t,\pi(X))\bb{n}(\pi(X))- \boldsymbol\varphi(t,X) ), \label{momapptest}
\end{eqnarray}
    with $\pi$ being the argument projection onto $\Gamma$ given in $\eqref{p}$;
    \item \textbf{Coupled entropy balance} - test functions $(\varphi,\psi)$ are approximated with $(\varphi_{\Delta t},\psi)$, where
    \begin{eqnarray}
    \varphi_{\Delta t}(t,X):=\varphi(t,X)+\phi_{\Delta t}(t,X)(\psi(t,\pi(X))- \varphi(t,X) ), \label{entapptest}
\end{eqnarray}
\end{enumerate}

\begin{figure}[h!]
\centering\includegraphics[scale=0.21]{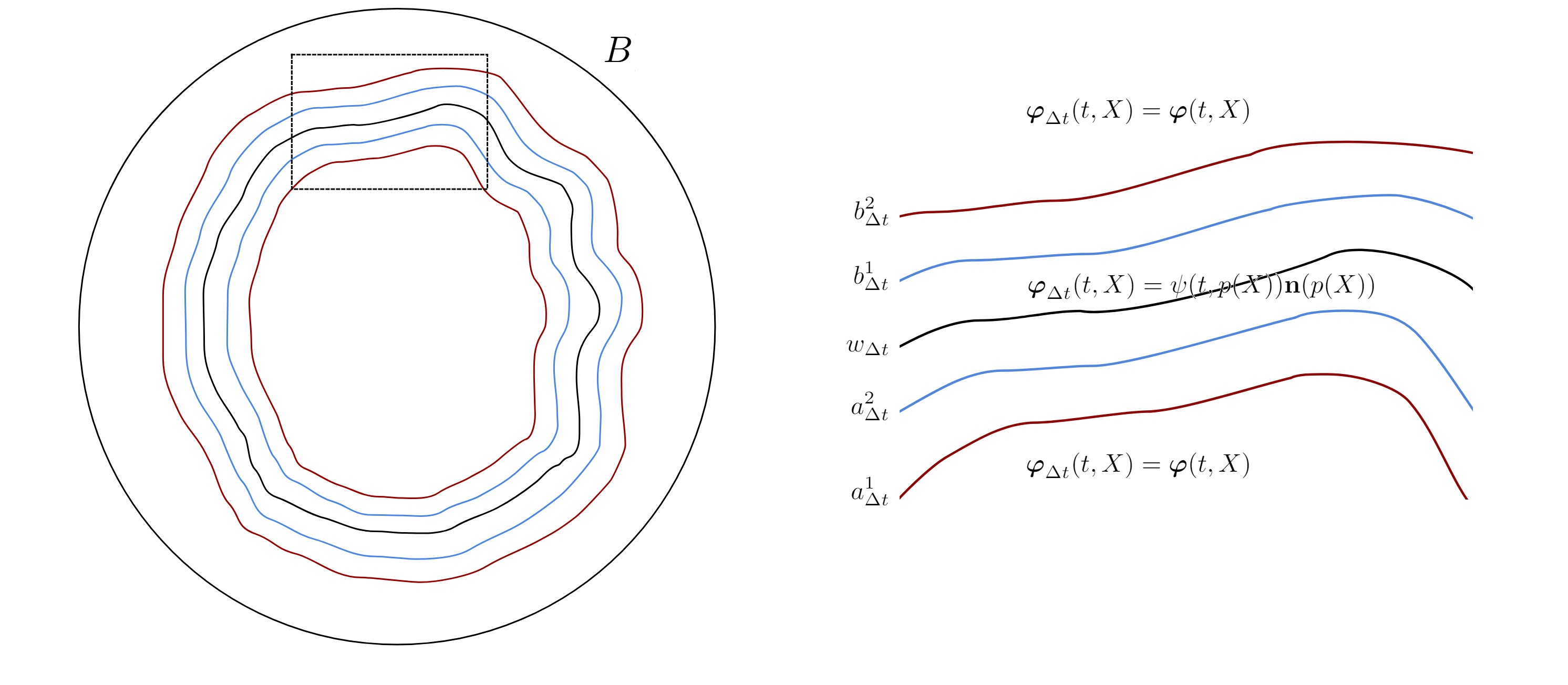}
\caption{The approximate test function $\boldsymbol\varphi_{\Delta t}$ corresponding to the coupled momentum equation. The graphs of functions on the right side are represented over $\Gamma$. Note that the strips are of proportional thickness due to $\eqref{distances}$.}
\label{test}
\end{figure}

${}$\\
\noindent One has the following:
\begin{lem}
    Let $\boldsymbol\varphi_{\Delta t}$ and $\varphi_{\Delta t}$ be defined in $\eqref{momapptest}$ and $\eqref{entapptest}$, respectively. Then
    \begin{eqnarray*}
        \boldsymbol\varphi_{\Delta t} \to \boldsymbol\varphi,\quad \quad \varphi_{\Delta t}\to \varphi, \quad  \text{ in } C^{3}([0,T]\times B)
    \end{eqnarray*}
    as $\Delta t$ tends to $0$.
\end{lem}
\begin{proof}
The proof is the same for both test functions, so we prove only for $\boldsymbol\varphi_{\Delta t}$. On $  [0,T]\times \big[\Omega^{b_{\Delta t}^2}(t)\setminus \Omega^{a_{\Delta t}^1}(t)\big]$, one has
\begin{eqnarray*}
   && |\nabla \boldsymbol\varphi_{\Delta t}|_\infty \leq  \big|\nabla \big[\phi_{\Delta t}(t,X)(\psi(t,\pi(X))\bb{n}(\pi(X))- \boldsymbol\varphi(t,X) ) \big] \big|_\infty +|\nabla\boldsymbol\varphi|_\infty\\[2mm]
   && \leq |\nabla\phi_{\Delta t}|_\infty |\psi(t,\pi(X))\bb{n}(\pi(X))- \boldsymbol\varphi(t,X) |_{\infty}+|\phi_{\Delta t}|_\infty |\nabla \big[(\psi(t,\pi(X))\bb{n}(\pi(X))-  \boldsymbol\varphi(t,X)\big] |_{\infty}+C\\[2mm]
    &&\leq \frac{C}{\max \big[ \max\limits_{(t,X)\in[0,T]\times \Gamma}(w_{\Delta t}- a_{\Delta t}^2), \max\limits_{(t,X)\in[0,T]\times \Gamma}(b_{\Delta t}^1-w_{\Delta t})  \big]} \\[2mm]
    &&\quad \times|\nabla \boldsymbol\varphi|_\infty\max \big[ \max\limits_{(t,X)\in[0,T]\times \Gamma}(w_{\Delta t}- a_{\Delta t}^2), \max\limits_{(t,X)\in[0,T]\times \Gamma}(b_{\Delta t}^2-b_{\Delta t}^1)  \big]+C \leq C, 
\end{eqnarray*}
by $\eqref{derivativesphi}$ and
\begin{eqnarray*}
    &&|\psi(t,\pi(X))\bb{n}(\pi(X))- \boldsymbol\varphi(t,X) | =|\boldsymbol\varphi(t,w(t,\pi(X)))- \boldsymbol\varphi(t,X) | \leq |\nabla \boldsymbol\varphi|_\infty |w(t,\pi(X))-X|_\infty\\
    &&\leq |\nabla \boldsymbol\varphi|_\infty \max \big[ \max\limits_{(t,X)\in[0,T]\times \Gamma}(w_{\Delta t}- a_{\Delta t}^1), \max\limits_{(t,X)\in[0,T]\times \Gamma}(b_{\Delta t}^2-w_{\Delta t})  \big]\\
    &&\leq C|\nabla \boldsymbol\varphi|_\infty \max \big[ \max\limits_{(t,X)\in[0,T]\times \Gamma}(w_{\Delta t}- a_{\Delta t}^2), \max\limits_{(t,X)\in[0,T]\times \Gamma}(b_{\Delta t}^1-w_{\Delta t})  \big],
\end{eqnarray*}
which follows by $\eqref{derivativesphi}$ the mean-value theorem. Moreover, since 
\begin{eqnarray*}
    \boldsymbol\varphi_{\Delta t} = \boldsymbol\varphi, \quad \text{on } [0,T]\times \Big(B\setminus \big[\Omega^{b_{\Delta t}^2}(t)\setminus \Omega^{a_{\Delta t}^1}(t)\big]\Big),
\end{eqnarray*}
we conclude
\begin{eqnarray*}
    |\nabla \boldsymbol\varphi_{\Delta t}|_\infty\leq C, \quad  \text{on } [0,T]\times B.
\end{eqnarray*}
Similarly, one can obtain $|\nabla^i \boldsymbol\varphi_{\Delta t}|_\infty \leq C$ for $i=2,3$, so $\boldsymbol\varphi_{\Delta t}$ is uniformly bounded in $C^3([0,T]\times B)$. Now, it is easy to see that
\begin{eqnarray*}
    \boldsymbol\varphi_{\Delta t}(t,X) \to \boldsymbol\varphi(t,X), \quad \nabla^i\boldsymbol\varphi_{\Delta t}(t,X) \to \nabla^i\boldsymbol\varphi(t,X),
\end{eqnarray*}
for all $(t,X)\in [0,T]\times B$ and $i=1,2,3$, directly by the construction of $\boldsymbol\varphi_{\Delta t}$, so by $\eqref{domainconv}$ and $\eqref{convergences}$ we conclude the desired convergences and finish the proof.
\end{proof}

Therefore, the above mentioned convergence properties of the sequence of approximate solutions $(\vartheta_{\Delta t}, \rho_{\Delta t}, \bb{u}_{\Delta t},w_{\Delta t}, \theta_{\Delta t})$ allow us to pass to the limit as $\Delta t \to 0$ in $\eqref{RCETM-Ap}, \eqref{coupledmomeq-Ap}$ and $\eqref{coupledenteq-Ap}$. Thus we obtain a solution to the extended problem and the main result of this section follows:
\begin{prop}\label{ext-thm}
Let $\beta \geq 4$, approximate parameters $\eta,\omega, \nu, \lambda,\delta>0$. The thermodynamical functions and coefficients satisfy \eqref{ext-mu}--\eqref{ext-e} with the hypotheses \eqref{mu_1}-\eqref{pm3}. Let the initial data satisfy \eqref{init-ap}--\eqref{init-ap1}. Then there exists a weak solution to the extended FSI-HEAT problem on $(0,T)\times B$ in the sense of Definition \ref{weaksolutionap1}. 
\end{prop}

\subsection{Uniform bounds on the physical domain $Q_T^w$}\label{estph}
Before we proceed to the extension limit where the integrals outside of $Q_T^w$ vanish, let us write the uniform estimates (w.r.t to the penalization parameters) on $Q_T^w$:
\begin{eqnarray*}
 \text{ess} \sup\limits_{t\in (0,T)}\int_\Gamma\Big[ |\partial_t w|^2 + |\Delta w|^2 +  \alpha_2 |\partial_t \nabla w|^2+|\theta|^2\Big]&\leq& C\\[2mm]
  \int_0^T\int_\Gamma \alpha_1 |\partial_t \nabla w|^2 + |\nabla \theta|^2 &\leq& C,\\[2mm]
  \text{ess} \sup\limits_{t\in (0,T)}\int_{\Omega^w(t)}\Big[ \delta \rho^\beta(t) +  \rho|\bb{u}  |^2  \Big] &\leq& C, \\[2mm]
 \int_{Q_T^w} \Big[ |\nabla \log (\vartheta)|^2 + |\nabla \vartheta^{\frac{3}{2}}|^2 \Big] &\leq& C, \\[2mm]
 \text{ess} \sup\limits_{t\in(0,T)}\int_{\Omega^w(t)} \Big[ a \vartheta^4 + \rho^{\frac{5}{3}} \Big]\leq C\text{ ess} \sup\limits_{t\in(0,T)}\int_{\Omega^w(t)} \rho e(\rho,\vartheta)  &\leq& C, \\[2mm]
 || \vartheta^\gamma||_{L^2(0,T; H^1(\Omega^w(t) ))} +  ||\log \vartheta||_{L^2(0,T; H^1(\Omega^w(t) ))} &\leq& C, \quad \text{for } \gamma \in [1,3/2],\\[2mm]
 \int_{Q_T^w}\frac{1}{\vartheta}\Big( \mathbb{S}(\vartheta, \nabla \bb{u}):\nabla \bb{u}+\frac{\kappa(\vartheta)|\nabla \vartheta|^2}{\vartheta}\Big) &\leq& C, \\[2mm]
 ||\rho \bb{u}||_{L^\infty(0,T; L^\frac{5}{4}(\Omega^w(t)))} &\leq& C,
\end{eqnarray*}
Before obtaining the estimates for $\nabla \bb{u}$, recall that we no longer work on a smooth domain $B$, but rather on $\Omega^w(t)$ which is in general not Lipschitz. Therefore, we need to prove a corresponding uniform Korn inequality on H\"{o}lder domains:

\begin{lem}\label{ourkorn}(\textbf{Uniform Korn's inequality}.) Let $w\in H^2(\Gamma)$ with $a_\Gamma<w<b_\Gamma$ and $\Omega^w$ be the domain corresponding to the displacement $w$. Moreover, let $M,L>0$, $\gamma\in (\frac32,\infty]$ and $q\in [1,2)$. Then, there exists a positive constant $C=C\big(q,M,L,\gamma,||w||_{H^2(\Gamma)}\big)$ such that
\begin{eqnarray}\label{korn}
    ||\bb{u}||_{ W^{1,q}(\Omega^w)}^2 \leq C \Big[||  \nabla \bb{u} + \nabla^\tau \bb{u}||_{L^2(\Omega^w)}^2 + \int_{\Omega^w} \rho |\bb{u}|^2 \Big]
\end{eqnarray}
for any $\rho,\bb{u}$ such that the right-hand side is finite and
\begin{eqnarray*}
    \rho\geq 0 \text{ a.e. on } \Omega^w, \quad ||\rho||_{L^\gamma(\Omega^w)}\leq L, \quad \int_{\Omega^w} \rho \geq M.
\end{eqnarray*}
\end{lem}
\begin{proof}
First, from \cite[Proposition 2.9]{lengeler}, we have
\begin{eqnarray}\label{lengelerkorn}
     ||\bb{u}||_{ W^{1,q}(\Omega^w)}^2 \leq C_2\big(q,||w||_{H^2(\Gamma)}\big) \Big[||  \nabla \bb{u} + \nabla^\tau \bb{u}||_{L^2(\Omega^w)}^2 + \int_{\Omega^w(t)} |\bb{u}|^2 \Big],
\end{eqnarray}
so, to prove $\eqref{korn}$, we can follow the approach from \cite[Theorem 10.17]{feireislnovotny} (note that we have to be careful with the Korn constant, i.e. we need to ensure that the constant on right-hand side of $\eqref{korn}$ is uniform with respect to $||w||_{H^2(\Gamma)}$). Namely, we assume the opposite - there exists a sequence of functions $r_n, \bb{v}_n, w_n$ such that:
\begin{eqnarray*}
    &&||r_n||_{L^1(\Omega^{w_n})} \geq M, \quad ||r_n||_{L^\gamma(\Omega^{w_n})} \leq L,\\
    &&||\bb{v}_n||_{ W^{1,q}(\Omega^{w_n})}=1, \quad ||  \nabla \bb{v}_n + \nabla^\tau \bb{v}_n||_{L^2(\Omega^{w_n})}^2 + \int_{\Omega^{w_n}} r_n |\bb{v}_n|^2 \leq \frac{1}{n},\\
    &&||w_n||_{H^2(\Gamma)}\leq C.
\end{eqnarray*}
Denote the weak limit of $w_n$ in $H^2(\Gamma)$ as $w$, and $r$ and $\bb{v}$ the weak limits in the following sense
\begin{eqnarray*}
    &&\chi_{\Omega^{w_n}}\bb{v}_n \rightharpoonup \chi_{\Omega^w}\bb{v},~ \chi_{\Omega^{w_n}}\nabla \bb{v}_n \rightharpoonup \chi_{\Omega^w}\nabla \bb{v}, \quad \text{ in } L^q(\mathbb{R}^3),\\
    &&    \chi_{\Omega^{w_n}}r_n \rightharpoonup \chi_{\Omega^w}r, \quad \text{ in } L^\gamma(\mathbb{R}^3), \\
    && \chi_{\Omega^{w_n}} \to \chi_{\Omega^w}, \quad \text{ a.e in } \Omega^w.
\end{eqnarray*}
This in particular implies
\begin{eqnarray*}
    M \leq \lim\limits_{n\to \infty}\int_{\mathbb{R}^3} r_n \chi_{\Omega^{w_n}} = \int_{\mathbb{R}^3} r \chi_{\Omega^{w}} = \int_{\Omega^{w}} r.
\end{eqnarray*}
Now, for every compact ball $O\Subset \Omega^w$, there is a $i_{0}\in \mathbb{N}$ such that $O\Subset \Omega^{w_i}$ for all $i\geq i_0$ (or at least for a subsequence), so by the compact imbedding of $L^q(O)$ into $W^{1,q}(O)$, we have that $\bb{v}_n \to \bb{v}$ in $L^q(O)$ as $i_0\leq i\to \infty$. By plugging in $(\bb{v}-\bb{v}_n)$ into the inequality $\eqref{lengelerkorn}$ on $O$, we conclude that $\bb{v}_n \to \bb{v}$ in $W^{1,q}(O)$. Since $O$ was arbitrary, this gives us
\begin{eqnarray*}
    \chi_{\Omega^{w_n}}\bb{v}_n \to \chi_{\Omega^{w}}\bb{v},~ \chi_{\Omega^{w_n}}\nabla \bb{v}_n \to \chi_{\Omega^{w}} \nabla \bb{v}, \quad \text{ in } L^q(\mathbb{R}^3),
\end{eqnarray*}
and
\begin{eqnarray*}
    ||\bb{v}||_{ W^{1,q}(\Omega^w)}=1, \quad   \nabla \bb{v} + \nabla^\tau \bb{v} =0, \quad r |\bb{v}|^2=0,
\end{eqnarray*}
so $\bb{v}$ satisfies the following elliptic equation (at least in the distributional sense)
\begin{eqnarray*}
    \nabla \cdot \big(\nabla \bb{v} + \nabla^\tau \bb{v}\big) = \Delta \bb{v} + \nabla \cdot\big[(\nabla \cdot \bb{v})I\big] = 0, \quad \text{on } \Omega^w.
\end{eqnarray*}
Thus, we conclude that $\bb{v}$ is analytic inside $\Omega^w$. Now, $r|\bb{v}|^2=0$ on the set $\{x \in\Omega^w(t) : r(x) >0 \}$ of a positive measure, so $\bb{v}\equiv 0$ due to analyticity. This is a contradiction, so the proof is now finished.
\end{proof}

Now, it follows that
\begin{eqnarray}
   &&||\bb{u}||_{L^2(0,T; W^{1,q}(\Omega^w(t)))} \leq C\big(q,E_0\big) \int_0^T\Big[||  \nabla \bb{u} + \nabla^\tau \bb{u}||_{L^2(\Omega^w(t))}^2 + \int_{\Omega^w(t)} \rho|\bb{u}|^2 \Big]\nonumber\\
   &&\leq C(q,E_0,\underline{\mu},\underline{\zeta})\int_{Q_T^w}\frac{1}{\vartheta} \mathbb{S}(\vartheta, \nabla \bb{u}):\nabla \bb{u} + C(q,E_0,T)  ~\text{ess} \sup\limits_{t\in (0,T)}\int_{\Omega^w(t)} \rho|\bb{u}|^2  \leq C(q,E_0,T,\underline{\mu},\underline{\zeta}), \nonumber\\ \label{nablau2}
\end{eqnarray}
for any $q<2$, where we used the lower bounds\footnote{This is the point where we need strict positivity of $\zeta$.} for $\mu,\zeta$ given in $\eqref{mu_1},\eqref{eta_1}$. This gives us the uniform bounds
\begin{eqnarray*}
    &&||\bb{u}||_{L^2(0,T; L^p(\Omega^w(t)))} \leq C(p),\\
    &&||\rho \bb{u}\otimes \bb{u}||_{L^2(0,T; L^\frac{31}{30}(\Omega^w(t)))} \leq C,
\end{eqnarray*}
for any $p<6$. Next,
\begin{eqnarray*}
||\rho s(\rho,\eta)||_{L^p(Q_T^w)}+  ||\rho s(\rho,\eta)\bb{u}||_{L^q(Q_T^w)}+ \Big|\Big|\frac{\kappa(\vartheta)\nabla \vartheta}{\vartheta}\Big|\Big|_{L^r(Q_T^w)}\leq C, 
\end{eqnarray*}
for some $p,q,r>1$. An improved pressure estimate based on the Bogovskii operator
\begin{eqnarray*}
\int_O p(\rho,\vartheta)\rho^\epsilon + \delta \rho^{\beta+\epsilon} \leq C(O)
\end{eqnarray*}
holds for any $O \Subset [0,T]\times \Omega^w(t)$ which is Lipschitz in space, where $\epsilon\in (0,1/9)$ and does not depend on $\delta$. Moreover,
\begin{eqnarray*}
     a_{\partial\Omega}<m\leq w(t,X) \leq M < b_{\partial\Omega}, \quad \text{for all }(t,X)\in \Gamma_T.
\end{eqnarray*}
Finally, in order to have a.e. positivity of $\theta$, let us prove the following result:
\begin{lem}\label{lntrace}
Let $w\in H^2(\Gamma)$ with $a_\Gamma<w<b_\Gamma$ and $\Omega^w$ be the domain corresponding to the displacement $w$. Moreover, let $\vartheta \in H^1(\Omega^w)$ such that $\ln \vartheta \in H^1(\Omega^w)$. Then, one has
\begin{eqnarray*}
    \gamma_{|\Gamma^w} \ln(\vartheta) = \ln(\gamma_{|\Gamma^w}\vartheta).
\end{eqnarray*}
\end{lem}
\begin{proof}
Let $\varepsilon>0$ and denote $\ln_\varepsilon(x):=\ln(x+\varepsilon)$. We will first show that $\gamma_{|\Gamma^w} \ln_\varepsilon(\vartheta) = \ln_\varepsilon(\gamma_{|\Gamma^w}\vartheta)$. Let $\vartheta_n:=\iota_{\frac1n}(\vartheta)$, where $\iota_\omega$ represents mollification for $\omega>0$. Note that $\ln_\varepsilon(\vartheta_n)\in C(\Omega^w)\cap H^1(\Omega^w)$ and it is well-defined, since $\vartheta_n>0$. Now, one has $\vartheta_n|_{\Gamma^w}\circ\Phi_w\to \gamma_{|\Gamma^w} \vartheta$, in $H^{s}(\Gamma)$, for any $s<\frac12$, so
\begin{eqnarray*}
    \lim\limits_{n\to\infty} \ln_\varepsilon(\vartheta_n|_{\Gamma^w}\circ\Phi_w) = \ln_\varepsilon(\gamma_{|\Gamma^w}\vartheta), \quad \text{in } L^2(\Gamma),
\end{eqnarray*}
by Nemytskii's theorem. Next,
\begin{eqnarray*}
    \lim\limits_{n\to\infty} \ln_\varepsilon (\vartheta_n)|_{\Gamma^w}\circ\Phi_w=\lim\limits_{n\to\infty} \gamma_{|\Gamma^w}\ln_\varepsilon (\vartheta_n)=\lim\limits_{n\to\infty} \gamma_{|\Gamma^w}\big(\ln_\varepsilon (\vartheta_n) - \ln_\varepsilon(\vartheta)\big)+\gamma_{|\Gamma^w} \ln_\varepsilon(\vartheta)=\gamma_{|\Gamma^w}\ln_\varepsilon(\vartheta)
\end{eqnarray*}
due to strong convergence of $\vartheta_n$ in $H^1(\Omega^w)$ and continuity of the trace operator from $H^1(\Omega^w)$ to $H^s(\Gamma)$, for $s<1/2$. Now, since $\ln_\varepsilon(\vartheta_n|_{\Gamma^w})=\ln_\varepsilon(\vartheta_n)|_{\Gamma^w}$ by continuity of $\vartheta_n$, one obtains $\lim\limits_{n\to\infty}\ln_\varepsilon(\vartheta_n|_{\Gamma^w})\circ\Phi_w=\lim\limits_{n\to\infty}\ln_\varepsilon(\vartheta_n)|_{\Gamma^w}\circ\Phi_w$, so $\ln_\varepsilon(\gamma_{|\Gamma^w}\vartheta)=\gamma_{|\Gamma^w}\ln_\varepsilon(\vartheta)$. Therefore, $\{\ln(\gamma_{|\Gamma^w}\vartheta+\varepsilon)\}_{\varepsilon>0}$ forms a uniformly bounded set in $H^s(\Gamma)$, so $\ln(\gamma_{|\Gamma^w}\vartheta+\varepsilon)\to \ln \gamma_{|\Gamma^w}\vartheta$ in $L^2(\Gamma)$ as $\varepsilon\to 0$. Since $\gamma_{|\Gamma^w}\ln(\vartheta+\varepsilon)\to\gamma_{|\Gamma^w}\ln(\vartheta)$ in $L^2(\Gamma)$, by letting $\varepsilon \to 0$ in the identity $\ln_\varepsilon(\gamma_{|\Gamma^w}\vartheta)=\gamma_{|\Gamma^w}\ln_\varepsilon(\vartheta)$, we conclude the desired result.
\end{proof}

Therefore, we conclude that $\ln\theta=\gamma_{|\Gamma^w}\ln\vartheta$, so $\ln\theta\in L^2(0,T;H^s(\Gamma))$, $s<\frac12$. As a direct consequence, one then has $\theta>0$ a.e. on $\Gamma_T$.

\section{Step 2 - the extension limit}\label{penlimit}
Following the ideas of \cite{heatfl}, in this section we consider limit of penalization parameters $\eta,\omega,\nu,\lambda\to 0$. In order to avoid dealing with unidentified limiting functions (note that when $\eta\to 0$, we lose the compactness of $\vartheta$ on $(0,T)\times (B\setminus\Omega^w(t))$), we will let all these parameters go to $0$ simultaneously, however not at the same rate. We set\footnote{This relation is not optimal, but it is enough to pass to the limit.}
\begin{eqnarray}\label{apppar}
    \sqrt[8]{\eta} = \sqrt{\omega}=\sqrt{\nu}=\lambda=:k
\end{eqnarray}
and denote $(\vartheta_k,\rho_k,\bb{u}_k, w_k,\theta_k)$ the solution obtained in Section $\ref{penalization}$ (see Proposition \ref{ext-thm}). Limits of some terms will be omitted and only the most important ones are proved.

\begin{lem}
Let approximation parameters satisfy $\eqref{apppar}$ and let approximate test functions be constructed as in Section $\ref{secapptest}$. We have the following as $k\to 0$:
        \begin{eqnarray*}
            &&\int_0^T \int_{B\setminus \Omega^{w_k}(t)} a\eta \vartheta_k^4 \text{div } \varphi_k \to 0, \quad \quad  \int_0^T \int_{B\setminus \Omega^{w_k}(t)} a \eta \vartheta_k^3 \bb{u}_k \cdot \nabla \varphi_k \to 0;\\
             && \int_0^T \int_{B\setminus \Omega^{w_k}(t)} \mathbb{S}_k^{w_k} (\vartheta_k, \nabla \bb{u}_k):\nabla \boldsymbol\varphi_k \to 0; \\
          &&  \int_0^T \int_{B\setminus \Omega^{w_k}(t)} \frac{\kappa_k^{w_k}(\vartheta_k) \nabla \vartheta_k}{\vartheta_k} \cdot \nabla \varphi_k \to 0; \\
      && \lambda\int_0^T \int_{B\setminus \Omega^{w_k}(t)} \vartheta_k^4 \varphi_k \to 0.
\end{eqnarray*}
    
\end{lem}

\begin{proof}
    First, one has
    \begin{eqnarray*}
        && \Big| \int_0^T \int_{B\setminus \Omega^{w_k}(t)} a\eta \vartheta_k^4 \text{div } \varphi_k \Big| \leq a \eta ||\vartheta_k^4||_{L^1((0,T)\times B)}|| \text{div} \varphi_k ||_{L^\infty((0,T)\times B)} \leq C a \eta  ||\vartheta_k||^{4}_{L^5((0,T)\times B)}\\
         &&\leq C \frac{\eta}{\lambda^{\frac{4}{5}}}= Ck^{\frac{36}5}\to 0,
\end{eqnarray*}
by $\eqref{vartheta5}$, and
\begin{eqnarray*}
        &&\Big| \int_0^T \int_{B\setminus \Omega^{w_k}(t)} a \eta \vartheta_k^3 \bb{u}_k \cdot \nabla \varphi_k \Big| \leq a \eta || \vartheta_k^3||_{L^\frac{6}{5}((0.T)\times B)} ||\bb{u}_k||_{L^2(0,T;L^6(B))}||\nabla \varphi_k||_{L^\infty((0,T)\times B)} \\
        &&\leq  C\frac{\eta}{\omega} || \vartheta_k^3||_{L^2(0,T; L^{\frac{40}{27}}(B\setminus\Omega^w(t)))} \leq C\frac{\sqrt{\eta} }{\omega} ||\vartheta^3||_{L^{\frac{5}{3}}((0,T)\times B)}^{\frac{1}{2}} \big(\eta||\vartheta^3||_{L^\infty(0,T; L^{\frac{4}{3}}(B\setminus\Omega^w(t)))} \big)^{\frac{1}{2}}\\
         &&\leq C \frac{\sqrt{\eta}}{\omega\sqrt{\lambda}}=Ck^{\frac{3}{2}} \to 0,
\end{eqnarray*}
by $\eqref{vartheta5},\eqref{varthetarho},\eqref{nablau}$ and the interpolation of Lebesgue spaces. Next
\begin{eqnarray}
    && \Big|\int_0^T \int_{(B\setminus \Omega^{w_k}(t))} \mathbb{S}_k^w (\vartheta_k, \nabla \bb{u}_k):\nabla \boldsymbol\varphi_k \Big|\leq \int_0^T \int_{(B\setminus \Omega^{w_k}(t))} \frac{1}{\sqrt{\vartheta_k}}|\mathbb{S}_k^w (\vartheta_k, \nabla \bb{u}_k):\nabla \boldsymbol\varphi_k| \sqrt{\vartheta_k}\nonumber\\
    && \leq  \Big(\int_0^T \int_{(B\setminus \Omega^{w_k}(t))} \Big|\frac{1}{\sqrt{\vartheta_k}}\mathbb{S}_k^w (\vartheta_k, \nabla \bb{u}_k)\Big|^{\frac{10}{9}}\Big)^{\frac{9}{10}} 
    ||\nabla \boldsymbol\varphi_k||_{L^\infty((0,T)\times B)}    \Big(  \int_0^T \int_{(B\setminus \Omega^{w_k}(t))}\vartheta_k^5 \Big)^{\frac{1}{10}}    \nonumber \\
    && \leq C \frac{1}{\lambda^{\frac{1}{10}}}  \Big(\int_0^T \int_{(B\setminus \Omega^{w_k}(t))} \Big|\frac{1}{\sqrt{\vartheta_k}}\sqrt{\big( \mathbb{S}_k^w (\vartheta, \nabla \bb{u}_k):\nabla \bb{u}_k\big)\big(\mu_k^w(\vartheta_k)+\zeta_k^w(\vartheta_k)}\big) \Big|^{\frac{10}{9}} \Big)^{\frac{9}{10}}\nonumber\\
    &&\leq C\frac{1}{\lambda^{\frac{1}{10}}}  \Big(\int_0^T \int_{(B\setminus \Omega^{w}(t))} \frac{1}{\vartheta_k} \mathbb{S}_k^w (\vartheta_k, \nabla \bb{u}_k):\nabla \bb{u}_k \Big)^{\frac{1}{2}} \Big(\int_0^T \int_{(B\setminus \Omega^{w_k}(t))} (\mu_k^w(\vartheta_k)+\zeta_k^w(\vartheta_k))^{\frac{5}{4}}\Big)^{\frac{2}{5}} \nonumber\\
    &&\leq C\frac{1}{\lambda^{\frac{1}{10}}}  \Big( ||\mu(\vartheta_k)||^{\frac 12}_{L^5((0,T)\times(B\setminus \Omega^{w}(t)) )}+||\zeta(\vartheta_k)||^{\frac 12}_{L^5((0,T)\times(B\setminus \Omega^{w}(t)) )}  \Big) ||f_k^{w_k}||^{\frac 12}_{L^{\frac{5}{3}}((0,T)\times(B\setminus \Omega^{w}(t)) )} \nonumber \\
    && \leq C \frac{\omega^{\frac 12}}{\lambda^{\frac{1}{5}}} = Ck^{\frac{3}{10}} \to  0, \nonumber
\end{eqnarray}
where we used $\eqref{mu_1},\eqref{eta_1},\eqref{fomega}, \eqref{vartheta5},\eqref{nablau}$ and
\begin{eqnarray*}
   |\mathbb{S}_k^w (\vartheta, \nabla \bb{u})|^2 \leq \big(\mathbb{S}_k^w (\vartheta, \nabla \bb{u}):\nabla \bb{u}\big)\big(\mu_k^w(\vartheta)+\zeta_k^w(\vartheta) \big).
\end{eqnarray*}
Now,
\begin{eqnarray*}
     &&\int_0^T \int_{B\setminus \Omega^{w_k}(t)} \frac{\kappa_k^{w_k}(\vartheta_k) \nabla \vartheta_k}{\vartheta_k} \cdot \nabla \varphi_k \\
     &&\leq \sqrt{\nu}\Bigg(\int_0^T \int_{B\setminus \Omega^{w_k}(t)} \frac{\kappa_k^{w_k}(\vartheta_k)|\nabla \vartheta_k|^2}{\vartheta_k^2}\Bigg)^{\frac{1}{2}} ||\nabla \varphi_k||_{L^\infty((0,T)\times B)}    \Bigg( \int_0^T \int_{B\setminus \Omega^{w_k}(t)}  \kappa(\vartheta_k) \Bigg)^{\frac{1}{2}} \\ 
     &&\leq C\frac{\sqrt{\nu}}{\sqrt{\lambda}} =C\sqrt{k}\to 0,
\end{eqnarray*}
by $\eqref{kappa_2},\eqref{kappa_3},\eqref{vartheta5},\eqref{fluidentropy}$ and finally,
\begin{eqnarray*}
     \int_0^T \int_{B\setminus \Omega^{w_k}(t)} \lambda \vartheta_k^4 \varphi_k \leq C\lambda^{\frac{1}{5}} \Big(\lambda \int_0^T \int_{B\setminus \Omega^{w_k}(t)}  \vartheta_k^5 \Big)^{\frac{4}{5}} \leq C \lambda^{\frac{1}{5}} =  Ck^{\frac{1}{5}} \to 0,
\end{eqnarray*}
so the proof is finished.

\end{proof}

Note that in the above limit, the convergences of all remaining terms on physical domain $Q_T^w$ can be proved in the same way as in Section $\ref{penalization}$, based on estimates given in Section $\ref{estph}$. Also, in the energy inequality, the following term
\begin{eqnarray*}
    \lambda \int_0^T \int_B \vartheta_\lambda^5  \geq 0
\end{eqnarray*}
and therefore it doesn't need to converge to $0$.

\begin{rem}
Compared to \cite{heatfl}, in our case there is less work to be done in the above limits. This is because we are working with a system with closed energy, while in \cite{heatfl} the moving domain velocity is given and acts as a source term, thus resulting in a modified energy inequality \cite[(3.15)]{heatfl}. The terms appearing on the right-hand side there need to be dealt with properly in their respective limits.\\
\end{rem}

\subsection{Limiting system and uniform bounds after the extension limit}

After passing to the limit, the limiting functions $(\vartheta,\rho,\bb{u},w,\theta)$ are a solution in the following sense:
\begin{mydef}\label{weaksolutionap3}(\textbf{Weak solution to the coupled problem with artificial pressure}.)
We say  that  $ (\vartheta,\rho ,\bb{u}, w, \theta)$ is a weak solution if the initial data satisfy the assumptions \eqref{init-ap}-\eqref{init-ap1} and

\begin{enumerate}
\item $(\vartheta,\rho,\bb{u},w,\theta)$ satisfy the same regularity properties as in Definition \ref{weaksolution}, and in addition $\rho \in L^\infty(0,T; L^\beta(\mathbb{R}^3))$.
\item The coupling conditions $\partial_t w  \bb{n} = \gamma_{|\Gamma^w} \mathbf{u}$ and $\vartheta =  \gamma_{|\Gamma^w} \theta$ hold on $\Gamma_T$.
\item The renormalized continuity equation
\begin{eqnarray}
 \int_{Q_T^w} \rho B(\rho)( \partial_t \varphi +\bb{u}\cdot \nabla \varphi) =\int_{Q_T^w} b(\rho)(\nabla\cdot \bb{u}) \varphi +\int_{\Omega^{w_0}} \rho_{0,\delta} B(\rho_{0,\delta}) \varphi(0,\cdot), \nonumber
\end{eqnarray}
holds for all $\varphi \in C_c^\infty([0,T)\times \overline{\Omega^w(t)})$ and any $b\in L^\infty\cap C[0,\infty)$ such that $b(0)=0$ with $B(\rho)=B(1)+\int_1^\rho \frac{b(z)}{z^2}dz$.
\item The coupled momentum equation
\begin{eqnarray}
    &&\int_{Q_T^w} \rho \bb{u} \cdot\partial_t \boldsymbol\varphi + \int_{Q_T^w}(\rho \bb{u} \otimes \bb{u}):\nabla\boldsymbol\varphi +\int_{Q_T^w} p_{\delta}(\rho,\vartheta) (\nabla \cdot \boldsymbol\varphi) -  \int_{Q_T^w} \mathbb{S}(\vartheta, \nabla\bb{u}): \nabla \boldsymbol\varphi\nonumber\\
    &&+\int_{\Gamma_T} \partial_t w \partial_t \psi - \int_{\Gamma_T}\Delta w \Delta \psi-\alpha_1 \bint_{\Gamma_T} \partial_t \nabla w \cdot \nabla \psi +  \alpha_2 \bint_{\Gamma_T}\partial_t \nabla w\cdot \partial_t \nabla w  +\int_{\Gamma_T} \nabla \theta \cdot \nabla \psi  \nonumber\\
    &&= - \int_{\Omega^{w_0}}(\rho\mathbf{u})_{0,\delta}\cdot\boldsymbol\varphi(0,\cdot)- \int_\Gamma v_0 \psi- \alpha_2  \int_\Gamma  \nabla v_0 \cdot \nabla\psi(0,\cdot)  \nonumber
\end{eqnarray}
holds for all $\boldsymbol\varphi \in C_c^\infty([0,T)\times B)$ and $\psi\in C_c^\infty([0,T)\times\Gamma)$ such that $\psi \bb{n} =\gamma_{|\Gamma^{w}}\boldsymbol\varphi$ on $\Gamma_T$.
\item The coupled entropy balance
\begin{eqnarray}
    &&\int_{Q_T^w} \rho s( \partial_t \varphi + \bb{u}\cdot \nabla \varphi) -\int_{Q_T^w}\frac{\kappa(\vartheta) \nabla \vartheta \cdot \nabla \varphi}{\vartheta}+ \langle \sigma; \varphi \rangle_{[\mathcal{M},C]([0,T]\times \overline{\Omega^w(t)})} \nonumber\\
    &&+\int_{\Gamma_T} \theta \partial_t \tilde{\psi} -\int_{\Gamma_T} \nabla\theta \cdot \nabla \tilde{\psi} +\int_{\Gamma_T} \nabla w \cdot \nabla \partial_t \tilde{\psi}  \nonumber\\
    &&= -\int_{\Omega^{w_0}} \rho_{0,\delta} s(\vartheta_{0,\delta},\rho_{0,\delta}) \varphi (0,\cdot)-  \int_{\Gamma} \theta_0 \tilde{\psi}(0,\cdot) - \int_{\Gamma} \nabla w_0 \cdot \nabla\tilde{\psi}(0,\cdot)  \nonumber
\end{eqnarray}
holds for all $\varphi \in C_c^\infty([0,T)\times B)$ and $\tilde{\psi}\in C_c^\infty([0,T)\times\Gamma)$ such that $\tilde{\psi} =\gamma_{|\Gamma^{w}}\varphi$ on $\Gamma_T$, where
\begin{eqnarray*}
      \sigma\geq \frac{1}{\vartheta}\Big( \mathbb{S}(\vartheta, \nabla \bb{u}):\nabla \bb{u}+\frac{\kappa(\vartheta)|\nabla \vartheta|^2}{\vartheta}\Big).
\end{eqnarray*}

\item The energy inequality 
\begin{eqnarray}
    &&\displaystyle{\int_{B} \Big( \frac{1}{2} \rho |\bb{u}|^2 + \rho e(\rho,\vartheta)+\frac{\delta}{\beta-1}\rho^{\beta} \Big)(t)}  \nonumber \\
    &&\quad+\frac{1-\delta}{2}|| \partial_t w(t)||_{L^2(\Gamma)}^2 +\frac{1}{2}|| \Delta w(t)||_{L^2(\Gamma)}^2 +\frac{ \alpha_2}{2}||\nabla \partial_t w(t)||_{L^2(\Gamma)}^2\nonumber \\
    &&\quad + \frac{1-\delta}{2}||\theta(t)||_{L^2(\Gamma)}^2+ \int_{0}^t \int_{\Gamma}\big(\alpha_1 |\nabla \partial_t w|^2+ | \nabla \theta|^2 \big) \nonumber \\
    &&\leq \displaystyle{\int_{B} \Big( \frac{1}{2 \rho_{0,\delta }} |(\rho\bb{u})_{0,\delta }|^2 + \rho_{0,\delta } e(\vartheta_{0,\delta },\rho_{0,\delta })}+\frac{\delta}{\beta-1}\rho_{0,\delta }^{\beta} \Big) \nonumber\\
    &&\quad+ \frac{1}{2}|| v_{0,\delta }||_{L^2(\Gamma)}^2+\frac{1}{2}|| \Delta w_{0,\delta }||_{L^2(\Gamma)}^2 + \frac{ \alpha_2}{2}||\nabla v_{0,\delta}||^2+ \frac{1}{2}||\theta_{0,\delta }||_{L^2(\Gamma)}^2, \nonumber
\end{eqnarray}
holds for any $t\in (0,T]$.

\end{enumerate}
\end{mydef}

Note that the limiting functions also satisfy the bounds given in Section $\ref{estph}$.


\section{Step 3 - Pressure regularization limit $\delta\to 0$}\label{sec5}
In this section, we consider the limit as pressure regularization parameter $\delta\to 0$. Throughout this section $(\vt_\delta, \rho_\delta, \vu_\delta, w_\delta, \theta_\delta)$ denotes the solution as stated in Definition \ref{weaksolutionap3} for certain $\delta>0$. The convergences here which are the same as in previous sections are omitted, and we focus on the ones that are different -- the strong convergence of fluid density and fluid temperature.

\subsection{Pointwise convergence of temperature}
The convergence of the temperature is solved similarly to Section \ref{sec33}. In particular, we use the div-curl lemma in order to deduce 
\begin{equation}\label{vartheta.konvergence}
\overline{\vartheta^4} = \overline{\vartheta^3}\vartheta.
\end{equation}
Now let $\nu_{t,x}$ be a Young measure related to $\{\vartheta_{\delta}\}_{\delta>0}$. Then \eqref{vartheta.konvergence} can be reformulated as
$$
\int_{[0,\infty)} s^4\ {\rm d}\nu_{t,x}(s) = \int_{[0,\infty)} s^3\ {\rm d}\nu_{t,x}(s)\vartheta
$$
which yields
$$
\int_{[0,\infty)} s^4 - s^3\vartheta - \vartheta^3s + \vartheta^4\ {\rm d}\nu_{t,x}(s) = 0
$$
and
$$
\int_{[0,\infty)} (s-\vartheta)(s^3 - \vartheta^3)\ {\rm d}\nu_{t,x}(s) = 0.
$$
Since the integrand is positive everywhere up to $s=\vartheta$, we get $\nu_{t,x}$ is a Dirac mass supported  in $\vartheta(t,x)$ and the point-wise convergence follows.

\subsection{Pointwise convergence of density}
We establish a family of smooth concave functions
$$
T_k(z) = T\left(\frac zk\right),\ T(z) = \left\{\begin{array}{l}z,\ \mbox{for }z\in (0,1)\\ 2,\ \mbox{for }z\geq 3,\\ \mbox{concave otherwise}\end{array}\right.
$$
Similarly to the previous section, we use the convergence of the effective viscous flux:

\begin{lem}\label{evf}
The equality
\begin{multline*}
\lim_{\delta \to 0} \int_{Q^w_T} \varphi \left(p_M(\rho_\delta,\vt_\delta) - \left(\frac 43 \mu(\vt_\delta) + \eta(\vt_\delta)\right)\diver \vu_\delta\right) T_k(\rho_\delta)\ {\rm d}x{\rm d}t\\ = \int_{Q^w_T} \varphi \left(\overline{p_M(\rho,\vt)}  - \overline{\left(\frac 43 \mu(\vt) + \eta(\vt)\right)\diver \vu}\right) \overline{T_k(\rho)}\ {\rm d}x{\rm d}t
\end{multline*}
holds for all $\varphi \in C_c^\infty(Q^w_T)$.
\end{lem}
\begin{proof}
The proof of this lemma relies on testing the momentum equation by $\varphi \nabla\Delta^{-1}(T_k(\rho_\delta))$. Since it does not differ from the proof of \cite[(3.324)]{feireislnovotny}, we omit any further details and we refer interested reader to this book.
\end{proof}

Assume that $\rho, \vu$ and $\rho_\delta, \vu_\delta$ solve the renormalized continuity equation, i.e.,
\begin{equation}\label{renorm.delta}
\partial_t b(\rho) + \diver (b(\rho)\vu) + (b'(\rho)\rho - b(\rho))\diver \vu  = 0
\end{equation}
holds in a weak sense for any $b\in C^1(\mathbb R)$ with $b'(z) = 0$ for $z$ sufficiently large. 

We  introduce functions
$$
L_k (z) = \left\{\begin{array}{l}z\log z\ \mbox{for }0\leq z \leq k,\\ z\log k + z\int_k^z\frac 1{s^2}T_k(s)\ {\rm d}s\ \mbox{for }s>k\end{array}\right.
$$
and we use them as $b$ in \eqref{renorm.delta}. We obtain
\begin{equation}\label{delta.renorm}
\partial_t L_k(\rho_\delta) + \diver (L_k(\rho_\delta)\vu_\delta) + T_k(\rho_\delta)\diver \vu_\delta = 0
\end{equation}
and
$$
\partial_t L_k(\rho) + \diver (L_k(\rho)\vu) + T_k(\rho)\diver \vu = 0.
$$
We pass to the limit in \eqref{delta.renorm}. In what follows $\Phi_n$ is a smooth function with compact support such that $\Phi_n\to 1$ pointwisely. We obtain with help of Lemma \ref{evf}
\begin{multline*}
\int_{\Omega^w_\tau} L_k(\rho(\tau,\cdot)) - L_k(\rho_\delta(\tau,\cdot))\ {\rm d}x\geq \int_{Q^w_\tau}T_k(\rho_\delta)\diver \vu_\delta - T_k(\rho)\diver \vu\ {\rm d}x{\rm d}t\\
=\int_{Q^w_\tau}\Phi_n\left(T_k(\rho_\delta)\diver \vu_\delta -\frac 1{\frac 43\mu(\theta) + \lambda(\theta)}p_M(\rho_\delta,\theta_\delta)T_k(\rho_\delta)\right.\\
 + \left.\frac 1{\frac 43\mu(\theta) + \lambda(\theta)}p_M(\rho_\delta,\theta_\delta)T_k(\rho_\delta) - T_k(\rho)\diver \vu\right) \ {\rm d}x{\rm d}t\\
 + \int_{Q^w_\tau} (1-\Phi_n)\left(T_k(\rho_\delta)\diver \vu_\delta - T_k(\rho)\diver \vu\right)\ {\rm d}x{\rm d}t \geq 0 + c(n)
\end{multline*}
where $c(n)$ tends to $0$. As a consequence, 
$$
\int_{\Omega^w_\tau}L_k(\rho(\tau,\cdot)) - \overline{L_k(\rho(\tau,\cdot))}\ {\rm d}x\geq 0
$$
for almost all $\tau\in [0,T]$. We send $k\to \infty$ to deduce $\rho\log\rho = \overline{\rho\log\rho}$ which yields $\rho_\delta\to \rho$ almost everywhere.

Now it suffices to prove that the renormalized continuity equation \eqref{renorm.delta} is true for $\rho$ and $\vu$ (recall that it holds for $\rho_\delta$ and $\vu_\delta$). According to \cite[Lemma 3.8]{feireislnovotny} it is enough to show that
\begin{equation}\label{what.we.want}
\mbox{osc}_q[\rho_\delta\to \rho](Q^w_T) <\infty
\end{equation}
for some $q>2$ where
$$
\mbox{osc}_q[\rho_\delta\to\rho](Q^w_T):= \sup_{k\geq 1} \left(\lim\sup_{\delta\to 0} \int_{Q^w_T}\left|T_k(\rho_\delta) - T_k(\rho)\right|^q\ {\rm d}x{\rm d}t\right).
$$
We have
\begin{multline}\label{the.first.one}
\int_{Q^w_T} |T_k(\rho_\delta) - T_k(\rho)|^q\ {\rm d}x{\rm d}t = \int_{Q^w_T} (1+\vartheta)^{-3q/8}(1+\vartheta)^{3q/8} |T_k(\rho_\delta) - T_k(\rho)|^q \ {\rm d}x{\rm d}t\\ 
\leq c\left(\int_{Q^w_T} (1+\vartheta)^{-1} |T_k(\rho_\delta) - T_k(\rho)|^{8/3}\ {\rm d}x{\rm d}t + \int_{Q^w_T} (1+\vartheta)^{3q/(8-3q)}\ {\rm d}x{\rm d}t\right)
\end{multline}
for some $q$ determined later. The second integral on the right hand side is finite assuming $q\leq \frac{32}{15}$. 

In order to deduce the estimate of the first term we recall that
\begin{equation*}
\int_{Q^w_T} \varphi \overline{|T_k(\rho_\delta) - T_k(\rho)|^{8/3}}\ {\rm d}x{\rm d}t\leq c\int_{Q^w_T} \varphi \left(1 + \overline{p_M(\rho,\vartheta)T_k(\rho)} - \overline{p_M(\rho,\vartheta)}\overline{T_k(\rho)}\right)\ {\rm d}x{\rm d}t
\end{equation*}
for every $\varphi \in C^\infty_c((0,T)\times \Omega),\ \varphi \geq 0$. We use Lemma \ref{evf} to obtain
\begin{multline}\label{the.second.one}
\int_{Q^w_T} (1+\vartheta)^{-1} \overline{|T_k(\rho_\delta) - T_k(\rho)|^{8/3}}\ {\rm d}x{\rm d}t\\
 \leq c\int_{Q^w_T} (1+\vartheta)^{-1}\left(1 + \left(\frac 43 \mu(\vartheta) + \eta(\vartheta)\right)(\overline{\diver \vu T_k(\rho)} - \diver \vu \overline{T_k(\rho)})\right) \ {\rm d}x{\rm d}t\\
\leq c \left(1 + \sup_{\delta>0} \|\diver \vu_\delta\|_{L^2(Q^w_T)} \lim\sup_{\delta\to0} \|T_k(\rho_\delta) - T_k(\rho)\|_{L^2(Q^w_T)}\right)\\
\leq c \left(1 + \lim\sup_{\delta\to0} \|T_k(\rho_\delta) - T_k(\rho)\|_{L^2(Q^w_T)}\right).
\end{multline}
The demanded boundedness \eqref{what.we.want} then follows from \eqref{the.first.one} and \eqref{the.second.one}.

\subsection{Maximal interval of existence}
Now, by passing to the limit $\delta\to 0$, the solutions
$(\vt_\delta, \rho_\delta, \vu_\delta, w_\delta, \theta_\delta)$ in the sense of Definition \ref{weaksolutionap3} converge to a solution of the original problem in the sense of Definition $\ref{weaksolution}$, however only on a time interval $(0,T)$ which was chosen to preserve the injectivity of the elastic structure in $\eqref{lifespan}$. We can extend the lifespan of the solution iteratively $(n-1)$ times to $(0,T_n)$ for any $n\in \mathbb{N}$. Now, by letting $n\to \infty$, this will either result in $T_n\to \infty$, which means that our solution is global, or $T_n\to T^*$, where $T^*$ is the moment when the elastic structure degenerates and loses injectivity. The proof of this claim is by now standard for FSI problems and we refer to \cite[pp. 397-398]{CDEG05} or \cite[Section 7.4]{compressible}.

\section*{Appendix A: the geometry}
In this Appendix we present some geometrical construction that are used in the proofs of several technical results in the paper. We use the notation introduced in subsection \ref{GeometryDef}
First, let
\begin{eqnarray}
    &&N_a^b:=\{ y+\bb{n}(y)z, y \in \Gamma, z \in (a_{\partial\Omega},b_{\partial\Omega})\} , \label{nab}
\end{eqnarray}
be a tubular neighbourhood of $\partial\Omega$, and let $\pi:N_a^b\to \Gamma$ be the projection onto $\Gamma$ defined as
\begin{eqnarray}\label{p}
    \pi:X\mapsto y, \quad \text{for unique }y\in \Gamma,z\in(a_{\partial\Omega},b_{\partial\Omega})  \text{ such that } X-y = \bb{n}(y)z,
\end{eqnarray}
and $d:N_a^b\to (a_{\partial\Omega},b_{\partial\Omega})$ be the signed distance function
\begin{eqnarray}\label{d}
    d:X\mapsto (X-\pi(X))\cdot \bb{n}(\pi(X)).
\end{eqnarray}

For a given displacement function 
\begin{eqnarray*}
    &&w\in H^1(0,T; H^1(\Gamma)) \cap  L^\infty(0,T; H^2(\Gamma))\cap W^{1,\infty}(0,T; L^2(\Gamma)),\\
    && a_{\partial\Omega}<m\leq w \leq M<b_{\partial\Omega}, \quad \text{on } \Gamma_T,
\end{eqnarray*}
we define the flow function  $\tilde{\Phi}_w^B:[0,T]\times B \to B$ as (here, $B$ is the extended domain given in Section $\ref{penalization}$)
\begin{eqnarray}\label{domainflow}
    \tilde{\Phi}_w^B(t,X):=X+ \begin{cases}
    f_\Gamma(d(X)) w(t,\pi(X))\bb{n}(\pi(X)), \quad \text{ in } [0,T]\times N_a^b,\\
    0, \quad \text{elsewhere in } [0,T]\times B,    \end{cases}
    \end{eqnarray}
where $f_\Gamma\in C_c^\infty(\mathbb{R})$ is defined as follows. Let $a_{\partial\Omega}<m''<m'<m<M<M'<M''<b_{\partial\Omega}$ and define $f_\Gamma:\mathbb{R}\to \mathbb{R}_0^+$ as (see figure $\ref{ffunction}$)
\begin{eqnarray*}
    f_\Gamma(X):= (f \ast g_\alpha)(X) =  \int_\mathbb{R} f(Y) g_\alpha(X-Y) dY
\end{eqnarray*}
where
\begin{eqnarray*}
    f(x):=\begin{cases}
    -\frac{x-m''+m'}{m'} +1,& \quad \text{for } m''\leq x \leq m'-m'' \\
    1,& \quad \text{for } m'-m''\leq x \leq M''-M',\\
    -\frac{x-M''+M'}{M'} +1,& \quad \text{for } M''-M'\leq x \leq M'' \\
    0, &\quad \text{for } x<m'' \text{ and } x>M''
    \end{cases}
\end{eqnarray*}
and $g_\alpha$ is a standard mollifying function with a support $(-\alpha,\alpha)$, for some $\alpha<\frac{1}{2}\min\{ m'-m'', M''-M' \}$. Note that
\begin{eqnarray}
-\frac{1}{M'} \leq f'_{\Gamma} \leq -\frac{1}{m'} \label{functionf}
\end{eqnarray}
\begin{figure}[h!]
\centering\includegraphics[scale=0.205]{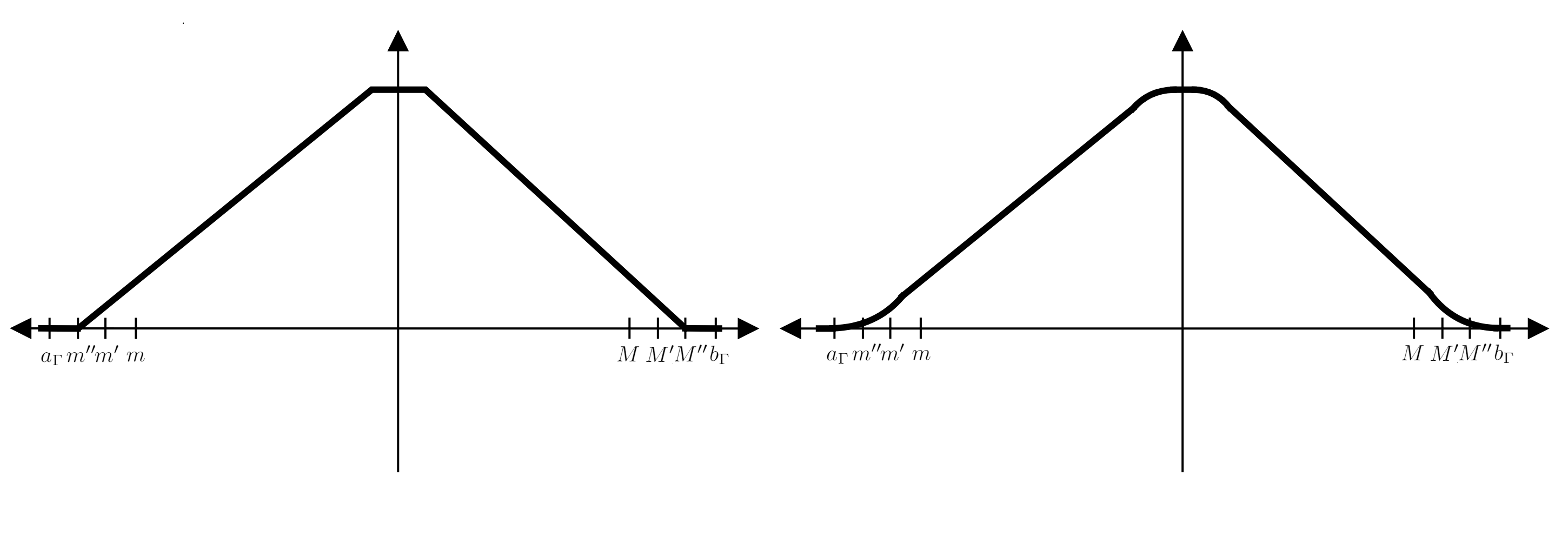}
\caption{Function $f$ (left) and function $f_\Gamma$ (right).}
\label{ffunction}
\end{figure}

\noindent Moreover, $\tilde{\Phi}_w^B|_{\Gamma} = \Phi_w$ and it inherits the regularity from $w$
\begin{eqnarray}\label{phireg}
    \tilde{\Phi}_w^B \in H^1(0,T; H_0^1(B)) \cap  L^\infty(0,T; H_0^2(B))\cap W^{1,\infty}(0,T; L^2(B)).
\end{eqnarray}
and we can calculate
\begin{eqnarray}\label{normal}
    \partial_{n(\pi(X))} \tilde{\Phi}_w^B(t,X)=\big[1+
    f'_{\Gamma}(d(X)) w(t,\pi(X))\big] \bb{n}(\pi(X)), \quad \text{on } [0,T]\times N_a^b
\end{eqnarray}
where due to $\eqref{functionf}$
\begin{align}
   &   1+ f_{\Gamma}'(d(X)) w(t,\pi(X)) \leq 1+ \max\Big\{\frac{\max\{0,M\}}{M'},\frac{\min\{0,m\}}{m'}  \Big\} \leq C,   \label{jacobianphi1}\\
     & 1+ f'_{\Gamma}(d(X)) w(t,\pi(X)) \geq 1- \max\Big\{\frac{\max\{0,M\}}{M'},\frac{\min\{0,m\}}{m'}  \Big\}\geq c>0.
    \label{jacobianphi2}
\end{align}

\section*{Appendix B: the coupled entropy inequality}
We first start with deriving the coupled entropy inequality $\eqref{entineqweak}$ for smooth solutions. For that reason, multiply the identity $\eqref{enteq}$ with a function $\varphi\in C_c^\infty([0,T)\times \overline{\Omega^w(t)})$. First two terms can be transformed as
\begin{eqnarray*}
    \int_{Q_T^w} \big[\partial_t(\rho s) \varphi + \nabla \cdot (\rho s \bb{u}) \varphi\big] =&& - \int_{Q_T^w} \big[ \rho s \partial_t \varphi +  (\rho s \bb{u}) \cdot\nabla \varphi \big] -\int_{\Omega^{w_0}} \rho_0 s(\vartheta_0,\rho_0) \varphi(0,\cdot) \\
    &&\underbrace{- \int_0^T \int_{\Gamma} (\rho s)\circ \Phi_w \partial_t w \bb{n}\cdot \bb{n}^w d\Gamma+ \int_0^T \int_{\Gamma^w} \rho s \bb{u}\cdot \bb{n}^w d\Gamma^w}_{=0, \quad \text{by } \eqref{kinc}},
\end{eqnarray*}
where we have used the Raynolds transport theorem and integration by parts. Next,
\begin{eqnarray*}
    \int_{Q_T^w} \nabla \cdot \left( \frac{\kappa(\vartheta) \nabla \vartheta}{\vartheta} \right) \varphi = - \int_{Q_T^w}\frac{\kappa(\vartheta) \nabla \vartheta \cdot \nabla \varphi}{\vartheta}+\int_0^T\int_{\Gamma^w(t)}  \frac{\kappa(\vartheta) \nabla \vartheta \cdot \bb{n}^w}{\vartheta}  \varphi d\Gamma^w
\end{eqnarray*}
so we obtain
\begin{eqnarray}
    &&\int_{Q_T^w} \rho s( \partial_t \varphi + \bb{u}\cdot \nabla \varphi) -\int_{Q_T^w}\frac{\kappa(\vartheta) \nabla \vartheta \cdot \nabla \varphi}{\vartheta}+ \int_{Q_T^w} \frac{\varphi}{\vartheta}\Big( \mathbb{S}(\vartheta, \nabla \bb{u}):\nabla \bb{u}
    +\frac{\kappa(\vartheta)|\nabla \vartheta|^2}{\vartheta}\Big)\nonumber\\
    &&+\int_0^T\int_{\Gamma^w(t)}  \frac{\kappa(\vartheta) \nabla \vartheta \cdot \bb{n}^w}{\vartheta}  \varphi d\Gamma^w =-\int_{\Omega^{w_0}} \rho_0 s(\vartheta_0,\rho_0) \varphi(0,\cdot) \label{part1}.
\end{eqnarray}
Next, multiply $\eqref{shellheateq}$ with $\tilde{\psi}\in C_c^\infty([0,T)\times \Gamma)$ to obtain
\begin{eqnarray}
    &&\int_{\Gamma_T} \theta \partial_t \tilde{\psi} -\int_{\Gamma_T} \nabla\theta \cdot \nabla \tilde{\psi} +\int_{\Gamma_T} \nabla w \cdot \nabla \partial_t \tilde{\psi} \nonumber\\
    &&=  -\int_{\Gamma} \theta_0 \tilde{\psi}(0,\cdot) - \int_{\Gamma} \nabla w_0 \cdot \nabla\tilde{\psi}(0,\cdot)  + \int_{\Gamma_T} q \tilde{\psi}.\label{part2}
\end{eqnarray}
Now, we sum up $\eqref{part1}$ and $\eqref{part2}$ for non-negative test functions $\varphi \in C_c^\infty([0,T)\times \overline{\Omega^w(t)})$ and $\tilde{\psi}\in C_c^\infty([0,T)\times\Gamma)$ such that $\gamma_{|\Gamma^w}\varphi =\tilde{\psi}$ on $\Gamma_T$, which then by the coupling condition $\eqref{heatf}$ finally give us $\eqref{entineqweak}$.\\

Note that here we have actually derived $\eqref{entineqweak}$ in a form of equality. However, in the definition of our weak solution $\ref{weaksolution}$, both entropy and energy identities are replaced with inequalities, which owes to the lower semicontinuity of norms. We argue that this is enough, i.e. if the weak solution is regular enough, then $\eqref{entineqweak}$ and $\eqref{enineq}$ are satisfied as an identity. More precisely, let smooth functions $(\vartheta,\rho,\bb{u},w,\theta)$ be a weak solution in the sense of definition. Following the ideas from \cite{poul}, we first assume that there is a smooth non-negative test function $\tilde\varphi$ such that w.l.o.g. $\tilde\varphi\leq \vartheta$ and $\eqref{entineqweak}$ holds as a strict inequality
\begin{eqnarray}
    &&\int_{Q_T^w} \rho s( \partial_t \tilde\varphi + \bb{u}\cdot \nabla \tilde\varphi) -\int_{Q_T^w}\frac{\kappa(\vartheta) \nabla \vartheta \cdot \nabla \tilde\varphi}{\vartheta}+ \int_{Q_T^w} \frac{\tilde\varphi}{\vartheta}\Big( \mathbb{S}(\vartheta, \nabla \bb{u}):\nabla \bb{u}
    +\frac{\kappa(\vartheta)|\nabla \vartheta|^2}{\vartheta}\Big) \nonumber\\
    &&+\int_{\Gamma_T} \theta \partial_t \tilde{\psi} -\int_{\Gamma_T} \nabla\theta \cdot \nabla \tilde{\psi} +\int_{\Gamma_T} \nabla w \cdot \nabla \partial_t \tilde{\psi} \nonumber\\
    &&< - \int_{\Omega^{w_0}} \rho_0 s(\vartheta_0,\rho_0) \varphi(0,\cdot) -  \int_{\Gamma} \theta_0 \tilde{\psi}(0,\cdot) - \int_{\Gamma} \nabla w_0 \cdot \nabla\tilde{\psi}(0,\cdot). \label{fakeineq}
\end{eqnarray}
Now, by choosing $(\boldsymbol\varphi,\psi)=(\bb{u},\partial_t w)$ in $\eqref{momeqweak}$, we obtain
\begin{eqnarray*}
    &&\frac{1}{2}\int_{\Omega^w(t)} \rho |\bb{u}|^2 \Big|_0^T +\frac{1}{2}\int_{\Gamma} |\partial_t w|^2\Big|_0^T + \frac{1}{2}\int_\Gamma |\Delta w|^2\Big|_0^T  - \int_{\Gamma_T} \Delta \theta \partial_t w + \alpha_1  \int_{\Gamma_T} |\partial_t \nabla w|^2 + \frac{\alpha_2}{2} \int_\Gamma |\partial_t \nabla w|^2 \Big|_0^T \\
    &&= \int_0^T \int_{\Omega^w(t)} p(\rho,\vartheta) (\nabla \cdot \bb{u})- \mathbb{S}(\vartheta,\nabla \bb{u}):\nabla \bb{u}.
\end{eqnarray*}
while the energy inequality gives us
\begin{eqnarray*}
    &&\frac{1}{2}\int_{\Omega^w(t)} \rho |\bb{u}|^2 \Big|_0^T +\frac{1}{2}\int_{\Gamma} |\partial_t w|^2\Big|_0^T + \frac{1}{2}\int_\Gamma |\Delta w|^2\Big|_0^T + \alpha_1  \int_{\Gamma_T} |\partial_t \nabla w|^2 + \frac{\alpha_2}{2} \int_\Gamma |\partial_t \nabla w|^2 \Big|_0^T \\
    &&\leq - \int_{\Omega^w(t)}\rho e(\rho,\vartheta) \Big|_0^T - \frac{1}{2}\int_\Gamma |\theta|^2 \Big|_0^T - \int_{\Gamma_T} |\nabla \theta|^2.
\end{eqnarray*}
which together imply
\begin{eqnarray*}
    \int_0^T \int_{\Omega^w(t)} p(\rho,\vartheta) (\nabla \cdot \bb{u})- \mathbb{S}(\vartheta,\nabla \bb{u}):\nabla \bb{u} \leq - \int_{\Omega^w(t)}\rho e(\rho,\vartheta) \Big|_0^T - \frac{1}{2}\int_\Gamma |\theta|^2 \Big|_0^T - \int_{\Gamma_T} |\nabla \theta|^2- \int_{\Gamma_T} \Delta \theta \partial_t w.
\end{eqnarray*}
Now, we sum up $\eqref{entineqweak}$ with $\varphi=\vartheta-\tilde\varphi$ with $\eqref{fakeineq}$ and obtain that $\eqref{entineqweak}$ holds as a strict inequality for $\varphi=\vartheta$. Then, one has by the Gibbs identity $\eqref{gibs}$
\begin{eqnarray*}
     \int_0^T \int_{\Omega^w(t)} p(\rho,\vartheta) (\nabla \cdot \bb{u})- \mathbb{S}(\vartheta,\nabla \bb{u}):\nabla \bb{u} > - \int_{\Omega^w(t)}\rho e(\rho,\vartheta) \Big|_0^T - \frac{1}{2}\int_\Gamma |\theta|^2 \Big|_0^T - \int_{\Gamma_T} |\nabla \theta|^2- \int_{\Gamma_T} \Delta \theta \partial_t w,
\end{eqnarray*}
which is obviously a contradiction. Therefore, $\eqref{entineqweak}$ must hold as an identity. Using the ideas from the above calculation, we easily obtain that now the energy inequality $\eqref{enineq}$ has to hold as an identity as well.

\vspace{.1in}
\noindent{\bf Acknowledgments:} B.M. was partially supported by the Croatian Science Foundation (Hrvatska Zaklada za Znanost) grant number IP-2018-01-3706. V. M., \v S. N., and A. R. have been supported by the Czech Science Foundation (GA\v CR) project GA19-04243S. The Institute of Mathematics, CAS is supported by RVO:67985840. A.R has also been supported by the Basque Government through the BERC 2018-2021 program and by the Spanish State Research Agency through BCAM Severo Ochoa excellence accreditation SEV-2017-0718 and through project PID2020-114189RB-I00 funded by Agencia Estatal de Investigación (PID2020-114189RB-I00 / AEI / 10.13039/501100011033).
\bibliography{NSF+thermo}

\begin{thebibliography}{10}

\bibitem{benevsova2020variational}
Barbora Bene{\v{s}}ov{\'a}, Malte Kampschulte, and Sebastian Schwarzacher.
\newblock A variational approach to hyperbolic evolutions and fluid-structure
  interactions.
\newblock {\em arXiv preprint arXiv:2008.04796}, 2020.

\bibitem{FSIBioMed}
Tom\'{a}\v{s} Bodn\'{a}r, Giovanni~P. Galdi, and \v{S}\'{a}rka
  Ne\v{c}asov\'{a}, editors.
\newblock {\em Fluid-structure interaction and biomedical applications}.
\newblock Advances in Mathematical Fluid Mechanics. Birkh\"{a}user/Springer,
  Basel, 2014.

\bibitem{boulakia}
M.~Boulakia and S.~Guerrero.
\newblock On the interaction problem between a compressible fluid and a
  {S}aint-{V}enant {K}irchhoff elastic structure.
\newblock {\em Adv. Differential Equations}, 22(1-2):1--48, 2017.

\bibitem{breit2021compressible}
Dominic Breit, Malte Kampschulte, and Sebastian Schwarzacher.
\newblock Compressible fluids interacting with 3d visco-elastic bulk solids.
\newblock {\em arXiv preprint arXiv:2108.03042}, 2021.

\bibitem{compressible}
Dominic Breit and Sebastian Schwarzacher.
\newblock Compressible fluids interacting with a linear-elastic shell.
\newblock {\em Arch. Ration. Mech. Anal.}, 228(2):495--562, 2018.

\bibitem{breit2021navier}
Dominic Breit and Sebastian Schwarzacher.
\newblock {N}avier-{S}tokes-{F}ourier fluids interacting with elastic shells.
\newblock {\em arXiv preprint arXiv:2101.00824}, 2021.

\bibitem{CasGraHil}
Jean-J\'{e}r\^{o}me Casanova, C\'{e}line Grandmont, and Matthieu Hillairet.
\newblock On an existence theory for a fluid-beam problem encompassing possible
  contacts.
\newblock {\em J. \'{E}c. polytech. Math.}, 8:933--971, 2021.

\bibitem{CDEG05}
Antonin Chambolle, Beno\^{\i}t Desjardins, Maria~J. Esteban, and C\'{e}line
  Grandmont.
\newblock Existence of weak solutions for the unsteady interaction of a viscous
  fluid with an elastic plate.
\newblock {\em J. Math. Fluid Mech.}, 7(3):368--404, 2005.

\bibitem{CS05}
Daniel Coutand and Steve Shkoller.
\newblock Motion of an elastic solid inside an incompressible viscous fluid.
\newblock {\em Arch. Ration. Mech. Anal.}, 176(1):25--102, 2005.

\bibitem{FeireislBook04}
Eduard Feireisl.
\newblock {\em Dynamics of viscous compressible fluids}, volume~26 of {\em
  Oxford Lecture Series in Mathematics and its Applications}.
\newblock Oxford University Press, Oxford, 2004.

\bibitem{FeireislHeat04}
Eduard Feireisl.
\newblock On the motion of a viscous, compressible, and heat conducting fluid.
\newblock {\em Indiana Univ. Math. J.}, 53(6):1705--1738, 2004.

\bibitem{MR3916800}
Eduard Feireisl.
\newblock Concepts of solutions in the thermodynamics of compressible fluids.
\newblock In {\em Handbook of mathematical analysis in mechanics of viscous
  fluids}, pages 1353--1379. Springer, Cham, 2018.

\bibitem{commoving}
Eduard Feireisl, Ond\v{r}ej Kreml, \v{S}\'{a}rka Ne\v{c}asov\'{a},
  Ji\v{r}\'{\i} Neustupa, and Jan Stebel.
\newblock Weak solutions to the barotropic {N}avier-{S}tokes system with slip
  boundary conditions in time dependent domains.
\newblock {\em J. Differential Equations}, 254(1):125--140, 2013.

\bibitem{feireisl2011convergence}
Eduard Feireisl, Ji{\v{r}}{\'\i} Neustupa, and Jan Stebel.
\newblock Convergence of a {B}rinkman-type penalization for compressible fluid
  flows.
\newblock {\em Journal of Differential Equations}, 250(1):596--606, 2011.

\bibitem{feireislnovotny}
Eduard Feireisl and Anton\'{\i}n Novotn\'{y}.
\newblock {\em Singular limits in thermodynamics of viscous fluids}.
\newblock Advances in Mathematical Fluid Mechanics. Birkh\"{a}user/Springer,
  Cham, 2017.
\newblock Second edition of [ MR2499296].

\bibitem{FeireislCompressible01}
Eduard Feireisl, Anton\'{\i}n Novotn\'{y}, and Hana Petzeltov\'{a}.
\newblock On the existence of globally defined weak solutions to the
  {N}avier-{S}tokes equations.
\newblock {\em J. Math. Fluid Mech.}, 3(4):358--392, 2001.

\bibitem{MR3537466}
Eduard Feireisl and Yongzhong Sun.
\newblock Conditional regularity of very weak solutions to the
  {N}avier-{S}tokes-{F}ourier system.
\newblock In {\em Recent advances in partial differential equations and
  applications}, volume 666 of {\em Contemp. Math.}, pages 179--199. Amer.
  Math. Soc., Providence, RI, 2016.

\bibitem{SunMarBor}
Marija Gali\'{c}, Boris Muha, and Sun\v{c}ica \v{C}ani\'{c}.
\newblock Analysis of a 3{D} nonlinear, moving boundary problem describing
  fluid-mesh-shell interaction.
\newblock {\em Trans. Amer. Math. Soc.}, 373(9):6621--6681, 2020.

\bibitem{MR2592281}
David G\'{e}rard-Varet and Matthieu Hillairet.
\newblock Regularity issues in the problem of fluid structure interaction.
\newblock {\em Arch. Ration. Mech. Anal.}, 195(2):375--407, 2010.

\bibitem{MR3281946}
David G\'{e}rard-Varet, Matthieu Hillairet, and Chao Wang.
\newblock The influence of boundary conditions on the contact problem in a 3{D}
  {N}avier-{S}tokes flow.
\newblock {\em J. Math. Pures Appl. (9)}, 103(1):1--38, 2015.

\bibitem{MR2438783}
C\'{e}line Grandmont.
\newblock Existence of weak solutions for the unsteady interaction of a viscous
  fluid with an elastic plate.
\newblock {\em SIAM J. Math. Anal.}, 40(2):716--737, 2008.

\bibitem{HilGra16}
C\'{e}line Grandmont and Matthieu Hillairet.
\newblock Existence of global strong solutions to a beam-fluid interaction
  system.
\newblock {\em Arch. Ration. Mech. Anal.}, 220(3):1283--1333, 2016.

\bibitem{MR3955112}
C\'{e}line Grandmont, Matthieu Hillairet, and Julien Lequeurre.
\newblock Existence of local strong solutions to fluid-beam and fluid-rod
  interaction systems.
\newblock {\em Ann. Inst. H. Poincar\'{e} Anal. Non Lin\'{e}aire},
  36(4):1105--1149, 2019.

\bibitem{gravina2020contactless}
Giovanni Gravina, Sebastian Schwarzacher, Ond{\v{r}}ej Sou{\v{c}}ek, and Karel
  T\r{u}ma.
\newblock Contactless rebound of elastic bodies in a viscous incompressible
  fluid.
\newblock {\em arXiv preprint arXiv:2011.01932}, 2020.

\bibitem{gruber2003two}
Christian Gruber, S{\'e}verine Pache, and Annick Lesne.
\newblock Two-time-scale relaxation towards thermal equilibrium of the
  enigmatic piston.
\newblock {\em Journal of statistical physics}, 112(5):1177--1206, 2003.

\bibitem{MR2354496}
Matthieu Hillairet.
\newblock Lack of collision between solid bodies in a 2{D} incompressible
  viscous flow.
\newblock {\em Comm. Partial Differential Equations}, 32(7-9):1345--1371, 2007.

\bibitem{HilTak09}
Matthieu Hillairet and Tak\'{e}o Takahashi.
\newblock Collisions in three-dimensional fluid structure interaction problems.
\newblock {\em SIAM J. Math. Anal.}, 40(6):2451--2477, 2009.

\bibitem{IKLT14}
Mihaela Ignatova, Igor Kukavica, Irena Lasiecka, and Amjad Tuffaha.
\newblock On well-posedness and small data global existence for an interface
  damped free boundary fluid-structure model.
\newblock {\em Nonlinearity}, 27(3):467--499, 2014.

\bibitem{JoSt}
Old\v{r}ich John and Jana Star\'{a}.
\newblock On the regularity of weak solutions to parabolic systems in two
  spatial dimensions.
\newblock {\em Comm. Partial Differential Equations}, 23(7-8):1159--1170, 1998.

\bibitem{KalKusLasTriWeb}
Barbara Kaltenbacher, Igor Kukavica, Irena Lasiecka, Roberto Triggiani, Amjad
  Tuffaha, and Justin~T. Webster.
\newblock {\em Mathematical theory of evolutionary fluid-flow structure
  interactions}, volume~48 of {\em Oberwolfach Seminars}.
\newblock Birkh\"{a}user/Springer, Cham, 2018.
\newblock Lecture notes from Oberwolfach seminars, November 20--26, 2016.

\bibitem{heatfl}
Ond\v{r}ej Kreml, V\'{a}clav M\'{a}cha, \v{S}\'{a}rka Ne\v{c}asov\'{a}, and
  Aneta Wr\'{o}blewska-Kami\'{n}ska.
\newblock Flow of heat conducting fluid in a time-dependent domain.
\newblock {\em Z. Angew. Math. Phys.}, 69(5):Paper No. 119, 27, 2018.

\bibitem{KukTuff12}
Igor Kukavica and Amjad Tuffaha.
\newblock Well-posedness for the compressible {N}avier-{S}tokes-{L}am\'{e}
  system with a free interface.
\newblock {\em Nonlinearity}, 25(11):3111--3137, 2012.

\bibitem{kukucka}
Peter Kuku\v{c}ka.
\newblock On the existence of finite energy weak solutions to the
  {N}avier-{S}tokes equations in irregular domains.
\newblock {\em Math. Methods Appl. Sci.}, 32(11):1428--1451, 2009.

\bibitem{lagnese1989boundary}
John~E Lagnese.
\newblock {\em Boundary stabilization of thin plates}.
\newblock SIAM Studies in Applied Mathematics, 1989.

\bibitem{LeeIntroduction}
John~M. Lee.
\newblock {\em Introduction to smooth manifolds}, volume 218 of {\em Graduate
  Texts in Mathematics}.
\newblock Springer-Verlag, New York, 2003.

\bibitem{lengeler}
Daniel Lengeler.
\newblock Weak solutions for an incompressible, generalized {N}ewtonian fluid
  interacting with a linearly elastic {K}oiter type shell.
\newblock {\em SIAM J. Math. Anal.}, 46(4):2614--2649, 2014.

\bibitem{LenRuz}
Daniel Lengeler and Michael R\r{u}\v{z}i\v{c}ka.
\newblock Weak solutions for an incompressible {N}ewtonian fluid interacting
  with a {K}oiter type shell.
\newblock {\em Arch. Ration. Mech. Anal.}, 211(1):205--255, 2014.

\bibitem{Leray34}
Jean Leray.
\newblock Sur le mouvement d'un liquide visqueux emplissant l'espace.
\newblock {\em Acta Math.}, 63(1):193--248, 1934.

\bibitem{LionsBook2}
Pierre-Louis Lions.
\newblock {\em Mathematical topics in fluid mechanics. {V}ol. 2}, volume~10 of
  {\em Oxford Lecture Series in Mathematics and its Applications}.
\newblock The Clarendon Press, Oxford University Press, New York, 1998.
\newblock Compressible models, Oxford Science Publications.

\bibitem{MR4189724}
Debayan Maity, Jean-Pierre Raymond, and Arnab Roy.
\newblock Maximal-in-time existence and uniqueness of strong solution of a 3{D}
  fluid-structure interaction model.
\newblock {\em SIAM J. Math. Anal.}, 52(6):6338--6378, 2020.

\bibitem{maity2021existence}
Debayan Maity, Arnab Roy, and Tak{\'e}o Takahashi.
\newblock Existence of strong solutions for a system of interaction between a
  compressible viscous fluid and a wave equation.
\newblock {\em Nonlinearity}, 34(4):2659, 2021.

\bibitem{MaityTak21}
Debayan Maity and Tak\'{e}o Takahashi.
\newblock Existence and uniqueness of strong solutions for the system of
  interaction between a compressible {N}avier-{S}tokes-{F}ourier fluid and a
  damped plate equation.
\newblock {\em Nonlinear Anal. Real World Appl.}, 59:103267, 34, 2021.

\bibitem{MR713680}
Akitaka Matsumura and Takaaki Nishida.
\newblock Initial-boundary value problems for the equations of motion of
  compressible viscous and heat-conductive fluids.
\newblock {\em Comm. Math. Phys.}, 89(4):445--464, 1983.

\bibitem{SM18}
Sourav Mitra.
\newblock Local existence of strong solutions of a fluid-structure interaction
  model.
\newblock {\em J. Math. Fluid Mech.}, 22(4):Paper No. 60, 38, 2020.

\bibitem{BorisTrace}
Boris Muha.
\newblock A note on the trace theorem for domains which are locally subgraph of
  a {H}\"{o}lder continuous function.
\newblock {\em Netw. Heterog. Media}, 9(1):191--196, 2014.

\bibitem{BorSeb}
Boris Muha and Sebastian Schwarzacher.
\newblock Existence and regularity of weak solutions for a fluid interacting
  with a non-linear shell in $3 d$.
\newblock {\em arXiv preprint arXiv:1906.01962}, 2019.

\bibitem{SubBorARMA}
Boris Muha and Sun\v{c}ica \v{C}ani\'{c}.
\newblock Existence of a weak solution to a nonlinear fluid-structure
  interaction problem modeling the flow of an incompressible, viscous fluid in
  a cylinder with deformable walls.
\newblock {\em Arch. Ration. Mech. Anal.}, 207(3):919--968, 2013.

\bibitem{multilayered}
Boris Muha and Sun\v{c}ica \v{C}ani\'{c}.
\newblock Existence of a solution to a fluid-multi-layered-structure
  interaction problem.
\newblock {\em J. Differential Equations}, 256(2):658--706, 2014.

\bibitem{BorSunSlip}
Boris Muha and Sun\v{c}ica \v{C}ani\'{c}.
\newblock Existence of a weak solution to a fluid-elastic structure interaction
  problem with the {N}avier slip boundary condition.
\newblock {\em J. Differential Equations}, 260(12):8550--8589, 2016.

\bibitem{MR3616663}
Antonin Novotn\'{y}.
\newblock Lecture notes on the {N}avier-{S}tokes-{F}ourier system: weak
  solutions, relative entropy inequality, weak strong uniqueness.
\newblock In {\em Topics on compressible {N}avier-{S}tokes equations},
  volume~50 of {\em Panor. Synth\`eses}, pages 1--42. Soc. Math. France, Paris,
  2016.

\bibitem{NP}
Antonin Novotn\'{y} and Milan Pokorn\'{y}.
\newblock Steady compressible {N}avier-{S}tokes-{F}ourier system for monoatomic
  gas and its generalizations.
\newblock {\em J. Differential Equations}, 251:270--315, 2011.

\bibitem{MR2807430}
Anton\'{\i}n Novotn\'{y} and Milan Pokorn\'{y}.
\newblock Weak solutions for steady compressible {N}avier-{S}tokes-{F}ourier
  system in two space dimensions.
\newblock {\em Appl. Math.}, 56(1):137--160, 2011.

\bibitem{PS21}
Milan Pokorn\'{y} and Emil Sk\v{r}\'{\i}\v{s}ovsk\'{y}.
\newblock Weak solutions for compressible {N}avier-{S}tokes-{F}ourier system in
  two space dimensions with adiabatic exponent almost one.
\newblock {\em Acta Appl. Math.}, 172:Paper No. 1, 31, 2021.

\bibitem{poul}
Luk\'{a}\v{s} Poul.
\newblock On dynamics of fluids in astrophysics.
\newblock {\em J. Evol. Equ.}, 9(1):37--66, 2009.

\bibitem{Raymond}
Jean-Pierre Raymond and Muthusamy Vanninathan.
\newblock A fluid-structure model coupling the {N}avier-{S}tokes equations and
  the {L}am\'{e} system.
\newblock {\em J. Math. Pures Appl. (9)}, 102(3):546--596, 2014.

\bibitem{MR832100}
V.~V. Shelukhin.
\newblock On the structure of generalized solutions of the one-dimensional
  equations of a polytropic viscous gas.
\newblock {\em Prikl. Mat. Mekh.}, 48(6):912--920, 1984.

\bibitem{srdjan1}
Sr{\dj}an Trifunovi\'{c} and Ya-Guang Wang.
\newblock Existence of a weak solution to the fluid-structure interaction
  problem in 3{D}.
\newblock {\em J. Differential Equations}, 268(4):1495--1531, 2020.

\bibitem{trwa3}
{Sr{\dj}an} Trifunovi{\'c} and Ya-Guang Wang.
\newblock On the interaction problem between a compressible viscous fluid and a
  nonlinear thermoelastic plate.
\newblock {\em arXiv preprint arXiv:2010.01639}, 2020.

\bibitem{MR826865}
Alberto Valli and Wojciech~M. Zajaczkowski.
\newblock Navier-{S}tokes equations for compressible fluids: global existence
  and qualitative properties of the solutions in the general case.
\newblock {\em Comm. Math. Phys.}, 103(2):259--296, 1986.

\bibitem{von2021falling}
Henry von Wahl, Thomas Richter, Stefan Frei, and Thomas Hagemeier.
\newblock Falling balls in a viscous fluid with contact: Comparing numerical
  simulations with experimental data.
\newblock {\em Physics of Fluids}, 33(3):033304, 2021.

\bibitem{Wright07}
Paul Wright.
\newblock The periodic oscillation of an adiabatic piston in two or three
  dimensions.
\newblock {\em Comm. Math. Phys.}, 275(2):553--580, 2007.

\end{thebibliography}
\bibliographystyle{plain}

\Addresses

\end{document}